\documentclass[11pt]{amsart}
\usepackage{amssymb,amscd}
\usepackage{pb-diagram}
\usepackage{verbatim}
\usepackage{graphicx}
\usepackage{amssymb,amscd,enumerate}
\newtheorem*{Whitney towers}{Theorem~\ref{Whitney towers}}
\newtheorem*{h-towers}{Theorems ~\ref{half} \& \ref{$(n)$-solvable}}

\newtheorem*{surgery curves}{Theorem~\ref{surgery curves}}
\newtheorem*{cg=0}{Theorem~\ref{vanish}}

\newtheorem{thm}{Theorem}[section]
\newtheorem{lem}[thm]{Lemma}
\newtheorem{cor}[thm]{Corollary}

\newtheorem{prop}[thm]{Proposition}
\newtheorem{defn}[thm]{Definition}
\newtheorem{rem}[thm]{Remark}

\newtheorem{ex}[thm]{Example}

\numberwithin{equation}{section}
\numberwithin{figure}{section}

\newcommand{\Hom}{\operatorname{Hom}}

\newcommand{\G}{\Gamma}

\newcommand{\FF}{\mathcal{F}}

\newcommand{\ra}{\longrightarrow}

\def\x{\times}
\def\p{\partial}
\def\ov{\overline}

\def\s{\sigma}
\def\lra{\longrightarrow}

\renewcommand{\i}{\iota}

\title{Knot Concordance and Blanchfield Duality}
\author{Tim D. Cochran, Shelly Harvey, Constance Leidy}
\address{Rice University, Houston, Texas, 77005-1892}
\email{cochran@math.rice.edu,shelly@math.rice.edu,cleidy@math.upenn.edu}

\thanks{}

\begin{document}

\begin{abstract} We introduce a new technique for showing classical knots and links are not slice. As one application we resolve a long-standing question as to whether certain natural families of knots contain topologically slice knots. We also present a simpler proof of the result of Cochran-Teichner that the successive quotients of the integral terms of the Cochran-Orr-Teichner filtration of the knot concordance group have rank $1$. For links we have similar results. We show that the iterated Bing doubles of many algebraically slice knots are not topologically slice. Some of the proofs do not use the existence of the Cheeger-Gromov bound, a deep analytical tool used by Cochran-Teichner. Our main examples are actually boundary links but cannot be detected in the algebraic boundary link concordance group, nor by any $\rho$ invariants associated to solvable representations into finite unitary groups.
\end{abstract}

\maketitle
\section{Introduction}\label{sec:Introduction}
We introduce a new technique for showing classical knots and links are not slice (first announced in ~\cite{CHL1}). As an application we report the partial resolution of a long-standing question about whether certain natural families of knots contain slice knots. We have similar results about analogous families of links.

A {\em link} $L=\{K_1,...,K_m\}$ of $m$-components is an ordered  collection of $m$ oriented circles disjointly embedded in $S^3$. A {\em knot} is a link of one component. A {\em topologically slice link} (abbreviated as \emph{slice} in this paper) is a link whose components bound a disjoint union of $m$ $2$-disks topologically and locally flatly embedded in $B^4$. The question of which links are slice links lies at the heart of the topological classification of $4$-dimensional manifolds.

The connected sum operation gives the set of all knots, modulo slice knots, the structure of an abelian group, called the {\em topological knot concordance group} $\mathcal{C}$, which is a quotient of its smooth analogue. For excellent surveys see ~\cite{Go1}~\cite{Li1}. For general links one must consider \emph{string} links to get a well-defined group structure, and this operation is not commutative ~\cite{Le7}. This paper gives new information about all of these groups, using techniques of noncommutative algebra and analysis, many of which have their origins in ~\cite{COT}. We employ the Cheeger-Gromov von Neumann $\rho$-invariants and higher-order Alexander modules that were introduced in ~\cite{COT}. Our new technique is to expand upon previous results of Leidy concerning higher-order Blanchfield forms \emph{without localizing the coefficient system} ~\cite{Lei1}~\cite{Lei3}. This is used to show that certain elements of $\pi_1$ of a slice knot (or link) exterior cannot lie in the kernel of the map  into any slice disk(s) exterior. In our results on links we also use recent results of Harvey on the \emph{torsion-free derived series of groups} ~\cite{Ha2}, and results of Cochran-Harvey on versions of Dwyer's Theorem for the derived series ~\cite{CH2}.

In the late $60$'s Levine \cite{L5} defined an epimorphism from $ \mathcal{C}$ to $\mathbb{Z}^\infty \oplus \mathbb{Z}_2^\infty \oplus \mathbb{Z}_4^\infty$, given by the Arf invariant, certain discriminants and twisted signatures associated to the infinite cyclic cover of the knot complement (using Stolzfus ~\cite{Sto}). A knot for which these invariants vanish is called an \emph{algebraically slice knot}. Thus the question at that time was ``Is every algebraically slice knot actually a slice knot?'' A nice way to create potential counterexamples is to begin with a known slice knot, $R_1$, such as the $9_{46}$ knot shown on the left-hand side of Figure~\ref{fig:ribbonCG}, and ``tie the bands into some knot $J_0$'', as shown on the right-hand side of Figure~\ref{fig:ribbonCG}. All of these genus one knots are algebraically slice since they have the same Seifert matrix as the slice knot $R_1$. Moreover certainly some of these knots \emph{are} slice knots, namely when $J_0$ is itself a slice knot. Similar knots have appeared in the majority of papers on this subject (for example ~\cite{Li1}\cite{Li5}\cite{Li7}\cite{Li10}\cite{GL1}).

\begin{figure}[htbp]\label{fig:ribbonCG}
\setlength{\unitlength}{1pt}
\begin{picture}(200,160)
\put(-70,10){\includegraphics[height=150pt]{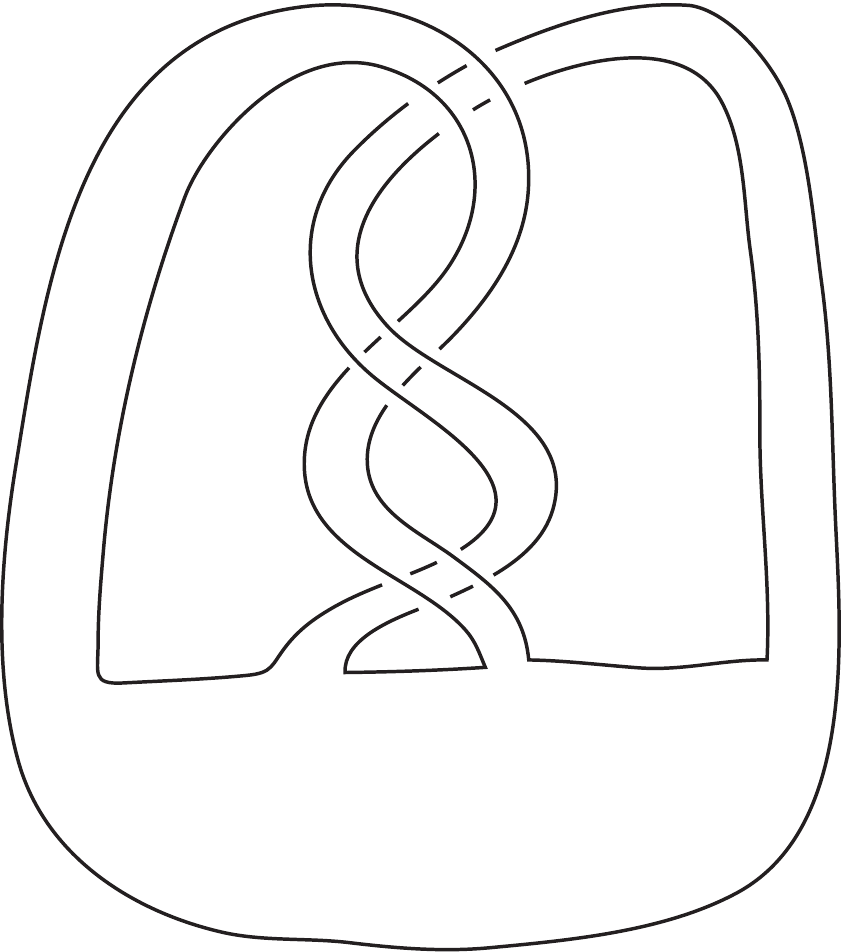}}
\put(130,10){\includegraphics[height=150pt]{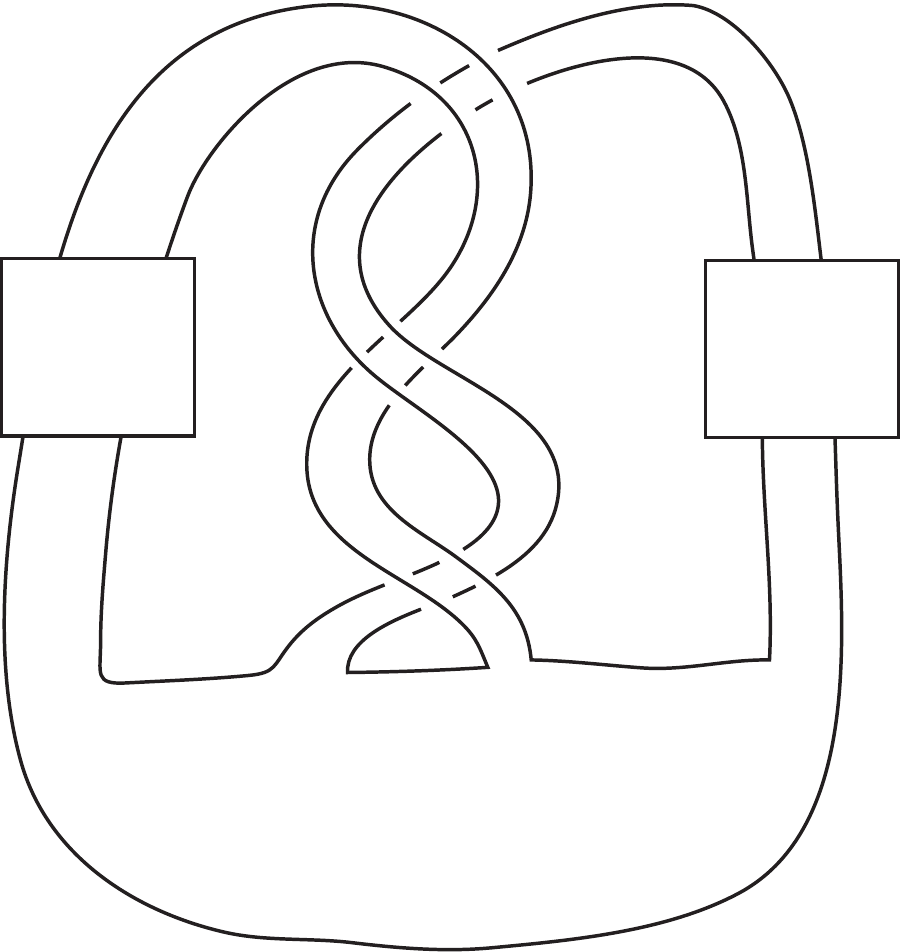}}
\put(251,101){$J_0$}
\put(140,100){$J_0$}
\put(100,98){$J_1\equiv$}
\put(-104,98){$R_1\equiv$}
\end{picture}
\caption{Algebraically Slice Knots $J_1$ Patterned on the Slice Knot $R_1$}
\end{figure}
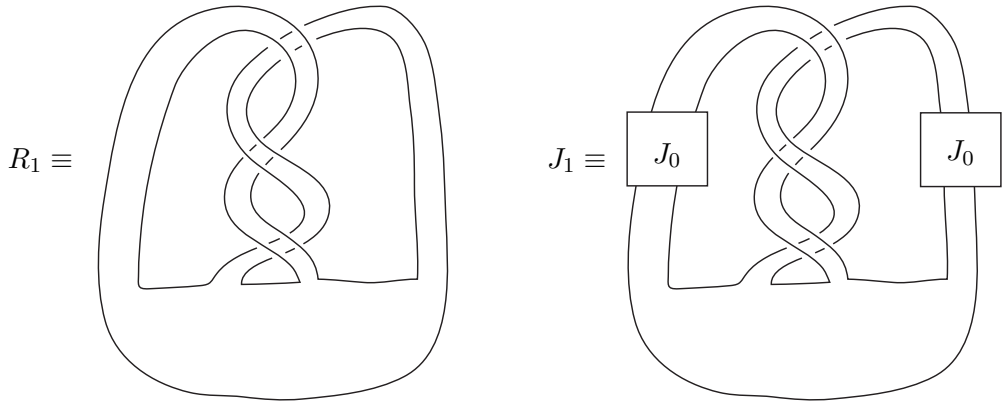
In the early $70$'s Casson and Gordon defined new knot concordance invariants via dihedral covers \cite{CG1}~\cite{CG2}. These `higher-order signature invariants' were used to show that some of the knots of Figure~\ref{fig:ribbonCG} need not be slice. P. Gilmer showed that these higher-order signature invariants for $J_1$ are equal to certain combinations of classical signatures of $J_0$ and thus the latter constituted higher-order obstructions to $J_1$ being a slice knot \cite{Gi3}\cite{Gi5}~(see ~\cite{Li6} for $2$-torsion invariants). These invariants were also used to show that the subgroup of algebraically slice knots has infinite rank ~\cite{Ji1}. Now the question arose:``What if $J_0$ itself were algebraically slice?'' Thus shortly after the work of Casson and Gordon the self-referencing family of knots shown in Figure~\ref{fig:family} was considered by Casson, Gordon, Gilmer and others ~\cite{Gi1}.

\begin{figure}[htbp]
\setlength{\unitlength}{1pt}
\begin{picture}(200,160)
\put(0,0){\includegraphics[height=150pt]{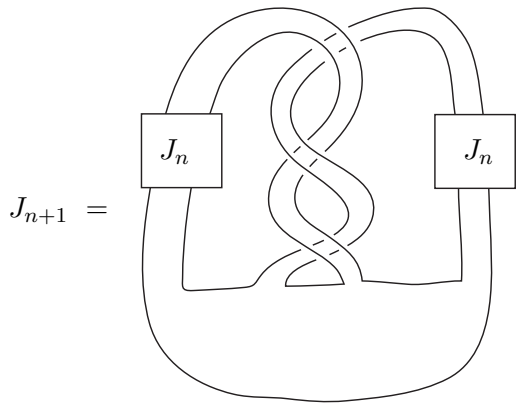}}
\put(-50,70){$J_{n+1}~=$}
\put(7,92){$J_{n}$}
\put(122,92){$J_{n}$}
\end{picture}
\caption{The recursive family $J_{n+1}, n\geq 0$}\label{fig:family}
\end{figure}

An example with $n=3$ and $J_0=U$, the unknot, is shown in Figure~\ref{fig:R3}
\begin{figure}[htbp]
\setlength{\unitlength}{1pt}
\begin{picture}(300,250)
\put(50,0){\includegraphics[height=250pt]{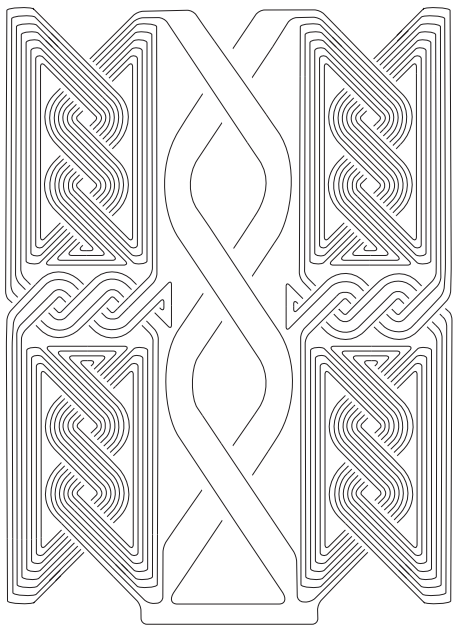}}
\end{picture}
\caption{The Ribbon Knot $J_3(U)$}\label{fig:R3}
\end{figure}

To summarize, each of the knots $J_n$, $n\geq 1$, is algebraically slice. If certain sums of classical signatures of $J_0=K$ are not zero, then Casson-Gordon invariants can be used to show that $J_1(K)$ is not a slice knot (see also torsion invariants ~\cite{Li6}). But all these invariants vanish for $J_n$ if $n\geq 2$. It was asked whether or not $J_n(K)$ is a slice knot assuming that some classical signature of $K$ is non-zero. In fact, Gilmer proved (unpublished) that $J_2(K)$, for certain $K$ is not a \emph{ribbon knot} ~\cite{Gi1}. Further attempts to create ``higher-order'' Casson-Gordon invariants for knots were not successful. 

Much more recently, Cochran, Orr and Teichner, Friedl, and Kim used higher-order signatures associated to solvable covers of the knot complement to find non-slice knots that could not be detected by the invariants of Levine or Casson-Gordon \cite{COT}\cite{COT2}\cite{Ki1}\cite{Fr2}. However the status of the knots $J_n$ above remained open. In fact the techniques of \cite{COT}, ~\cite{COT2}, ~\cite{CT} and ~\cite{CK} were limited to knots of genus at least $2$ (note each $J_n$ has genus $1$) because of their use of localization techniques.

We prove:

\newtheorem*{thm:main}{Theorem~\ref{thm:main}}
\begin{thm:main} For any $n\ge0$ there is a constant $C_n$ such that, if the absolute value of the integral of the Levine signature function of $J_0$ is greater than $C_n$ then $J_n$ is of infinite order in the topological knot concordance group. \end{thm:main}

\textbf{Remark}: We have proved that $C_n$ may be taken to be independent of $n$, but this requires a different proof that will appear separately.

These techniques can also be used to get new information about the topological concordance order of knots. For example, consider the family of knots below where $J_n$, $n\geq 1$, is one of the the algebraically slice knots above . For any such $E$, $E\# E$ is algebraically slice and has vanishing Casson-Gordon invariants. Therefore $E$ cannot be distinguished from an order $2$ knot by these invariants. 
\begin{figure}[htbp]
\setlength{\unitlength}{1pt}
\begin{picture}(200,160)
\put(0,0){\includegraphics[height=150pt]{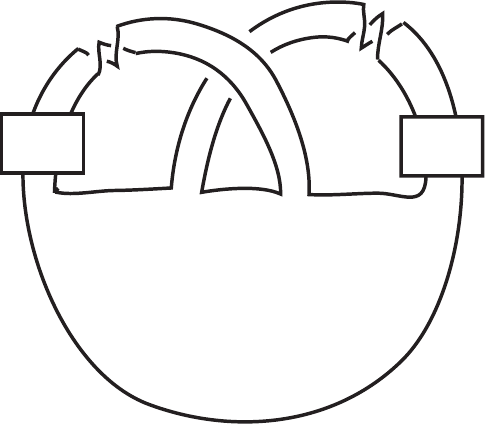}}
\put(-40,70){$E=$}
\put(7,95){$J_n$}
\put(147,95){$J_n$}
\end{picture}
\caption{Knots potentially of order 2 in the concordance group}\label{fig:torsionfigeight}
\end{figure}
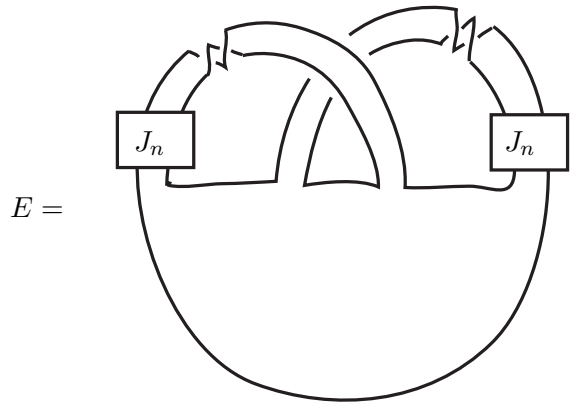
However we can show:
\newtheorem*{cor:torsion}{Corollary~\ref{cor:torsion}}
\begin{cor:torsion} There is a constant $D$ such that if the absolute value of the integral of the Levine signature function of $J_0$ is greater than $D$ then $E$ is of infinite order in the concordance group.
\end{cor:torsion}

Analogous to this family of knots, similar natural families of links have been considered (albeit much more recently). In particular, if $K$ is any knot then the \emph{Bing-double of $K$}, $BD(K)$ is the $2$-component link shown in Figure~\ref{fig:Bingdouble}.

\begin{figure}[htbp]
\setlength{\unitlength}{1pt}
\begin{picture}(111,109)
\put(0,0){\includegraphics{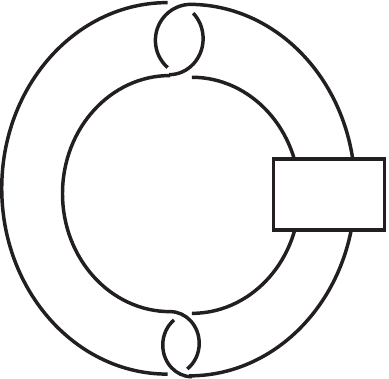}}
\put(88,50){$K$}
\put(137,54){=~~$BD(K)$}
\end{picture}
\caption{Bing double of $K$}\label{fig:Bingdouble}
\end{figure}

Again, if $K$ is slice then it is easy to see that $BD(K)$ is a slice link. A natural question is whether or not the converse is true. It was shown by Harvey that if the Bing double (or even an iterated Bing-double) of $K$ is topologically slice then the integral over the circle of the Levine signatures of $K$ is zero ~\cite[Corollary 5.6]{Ha2}. It was shown by Cimasoni that if $BD(K)$ is \emph{boundary slice} then $K$ is algebraically slice ~\cite{Ci}. After the announcement of our work, it was announced by Cha that if $BD(K)$ is a slice link then the entire signature function of $K$ vanishes (see ~\cite[Theorem 1.5]{Cha4}). Subsequently it was shown by Cha-Livingston-Ruberman that if $BD(K)$ is a slice link then $K$ must be an algebraically slice knot ~\cite{CRL}. Therefore the question remains: If $K$ is algebraically slice then must it follow that $BD(K)$ is a topologically slice link? What about iterated Bing doubles? We answer these questions in the negative by showing that certain higher-order signatures of $K$ offer further obstructions. For example, in Section~\ref{sec:Bingdoubles} we define \textbf{first-order signatures} of $K$, akin to Casson-Gordon invariants, and show that the first-order signatures of $K$, like the ordinary signatures, obstruct any \textbf{iterated Bing double} of $K$ from being a slice link. This improves on Harvey's theorem.

\newtheorem*{thm:Bingdouble}{Theorem~\ref{thm:Bingdouble}}
\begin{thm:Bingdouble} Let $K$ be an arbitrary knot. If some iterated Bing double of $K$ is topologically slice in a rational homology $4$-ball then one of the first-order signatures of $K$ is zero.
\end{thm:Bingdouble}

\newtheorem*{cor:Bingdouble}{Corollary~\ref{cor:Bingdouble}}
\begin{cor:Bingdouble}If $J_1$ is the algebraically slice knot of Figure~\ref{fig:ribbonCG} then there is a constant $C$ such that if the integral of the Levine signature function of $J_0$ is greater than $C$, then no iterated Bing double of $J_1$ is topologically slice. If $E_1$ is any knot as in Figure~\ref{fig:examplehighersigsfigeight2} (of order 2 in the algebraic concordance group) where the integral of the Levine signature function of $E_0$ is not zero, then no iterated Bing double of $E_1$ is slice in a rational homology ball.
\end{cor:Bingdouble}

\begin{figure}[htbp]
\setlength{\unitlength}{1pt}
\begin{picture}(200,160)
\put(0,0){\includegraphics[height=150pt]{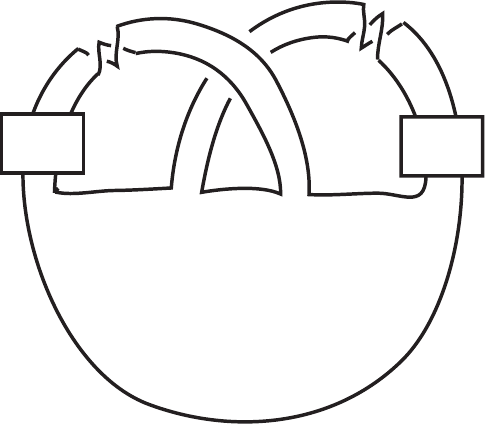}}
\put(-40,70){$E_1=$}
\put(7,95){$E_0$}
\put(145,94){$E_0$}
\end{picture}
\caption{$E_1$}\label{fig:examplehighersigsfigeight2}
\end{figure}

We remark that subsequent work of Cha shows that even many amphichiral knots have non-slice Bing doubles ~\cite{Cha4}. Amphichiral knots cannot be handled by the present paper.

We have similar results for iterated Bing doubles of the even more subtle knots of the family $J_n$ which, for $n>1$, recall not only are algebraically slice but also have vanishing Casson-Gordon invariants.

Furthermore recall that \cite{COT} introduced a filtration of $\mathcal{C}$ by \textbf{$(n)$-solvable knots}
$$
\cdots \subseteq \mathcal{F}_{n} \subseteq \cdots \subseteq
\mathcal{F}_1\subseteq \mathcal{F}_{0.5} \subseteq \mathcal{F}_{0} \subseteq \mathcal{C}.
$$
This is defined in Section~\ref{sec:appendix}. This filtration exhibits all of the previously known concordance invariants in its associated graded quotients of low degree, yet contains new information. In particular, it was shown in \cite{COT2} that $\mathcal{F}_{2}/\mathcal{F}_{2.5}$ contains an infinite rank summand of concordance classes of knots not detectable by previously known invariants. 

Our techniques provide a simplified proof of the following important result of Cochran and Teichner. 
\newtheorem*{cor:CT}{Theorem~\ref{cor:CT}}
\begin{cor:CT}[(Cochran-Teichner,~\cite{CT})]
For any $n \in \mathbb{N}_0$, the quotient groups $\mathcal{F}_{n}/\mathcal{F}_{n.5}$ contain a subgroup isomorphic to $\mathbb{Z}$.
\end{cor:CT}

In fact we prove this using the knots $J_n(K)$ (for suitably chosen $K$). This family is also simpler than the families of Cochran and Teichner. In fact this family is distinct even up to concordance from the examples of Cochran and Teichner, so one can show:

\begin{thm}\label{thm:main2}
For any $n \in \mathbb{N}_0$, the quotient groups $\mathcal{F}_{n}/\mathcal{F}_{n.5}$ contain a subgroup isomorphic to $\mathbb{Z}\oplus \mathbb{Z}$.
\end{thm}

However the proof of this result will not be given here, but in a separate paper where we will show that the quotients $\mathcal{F}_{n}/\mathcal{F}_{n.5}$ have \emph{infinite} rank.

We note that our construction of examples is all done in the smooth category  so that we actually also prove the corresponding statements about the smooth knot concordance group.

The specific families of knots and links of Figure~\ref{fig:family} and Figure~\ref{fig:Bingdouble} are important because of their simplicity and their history. However, they are merely particular instances of a more general `doubling' phenomenon to which our techniques may be applied. In order to state these results, we review a method we will use to construct examples. Let $R$ be a knot or link in $S^3$ and $\{\eta_1,\eta_2,\ldots,\eta_m\}$ be an oriented trivial link in $S^3$ which misses $R$ bounding a collection of disks that meet $R$ transversely. Suppose $\{K_1,K_2,\ldots,K_m\}$ is an $m$-tuple of auxiliary knots. Let $R(\eta_1,\ldots,\eta_m\,K_1,\ldots,K_m)$ denote the result of the operation pictured in Figure~\ref{fig:infection}, that is, for each $\eta_i$, take the embedded disk in $S^3$ bounded by $\eta_i$; cut off $R$ along the disk; grab the cut strands, tie them into the knot $K_i$ (with no twisting) and reglue as shown in Figure~\ref{fig:infection}.

\begin{figure}[htbp]
\setlength{\unitlength}{1pt}
\begin{picture}(262,71)
\put(10,37){$\eta_1$} \put(120,37){$\eta_m$} \put(52,39){$\dots$}
\put(206,36){$\dots$} \put(183,37){$K_1$} \put(236,38){$K_m$}
\put(174,9){$R(\eta_1,\dots,\eta_m,K_1,\dots,K_m)$}
\put(29,7){$R$} \put(82,7){$R$}
\put(20,20){\includegraphics{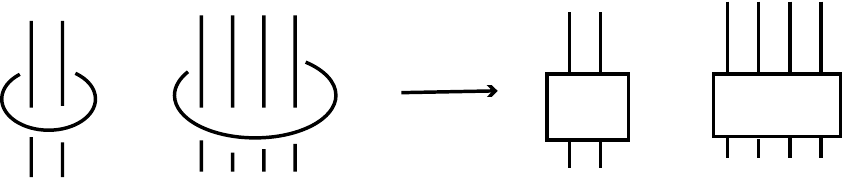}}
\end{picture}
\caption{$R(\eta_1,\dots,\eta_m,K_1,\dots,K_m)$:
Infection of $R$ by $K_i$ along $\eta_i$}\label{fig:infection}
\end{figure}
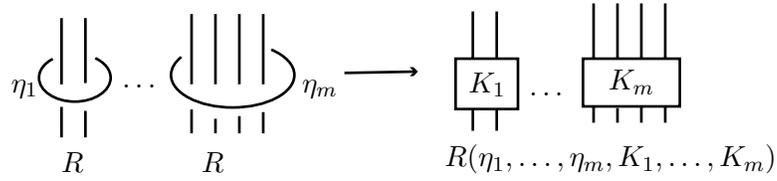
\noindent We will call this the result of \emph{infection performed on the link $R$ using the infection knots $K_i$ along the curves $\eta_i$}. This construction can also be
described in the following way. For each $\eta_i$, remove a tubular neighborhood of $\eta_i$ in $S^3$ and glue in the exterior of a tubular neighborhood of $K_i$ along their common boundary, which is a
torus, in such a way that the longitude of $\eta_i$ is identified with the meridian of $K_i$ and the meridian of $\eta_i$ with the reverse of the longitude of $K_i$. The resulting space can be seen to be homeomorphic to $S^3$ and the image of $R$ is the new link. In the case that $m=1$ this is the same as the classical satellite construction. In general it can be considered to be a `generalized satellite construction', widely utilized in the study of knot concordance. In the case that $m=1$ and $lk(\eta,R)=0$ it is precisely the same as forming a satellite of $J$ with winding number zero. This yields an operator
$$
R_{\eta}:\mathcal{C}\to \mathcal{C}.
$$
that has been studied (e.g. ~\cite{LiM}). For general $m$ with $lk(\eta_i,R)=0$, it should be considered as a \emph{generalized doubling operator}, $R_{\eta_i}$, parameterized by $(R,\{\eta_i\})$
$$
R_{\eta_i}:~\mathcal{C}\times\dots\times\mathcal{C}\to \mathcal{C}.
$$
If, for simplicity, we assume that all ``input knots'' assume identical then such an operator is a function
$$
R_{\eta_i}:~\mathcal{C}\to \mathcal{C}.
$$
Bing-doubling is an example of this ($m=1$) as suggested by Figure~\ref{fig:bingeta}.
\begin{figure}[htbp]
\setlength{\unitlength}{1pt}
\begin{picture}(111,109)
\put(0,0){\includegraphics{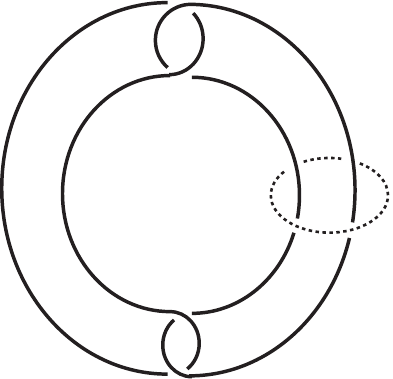}}
\put(120,50){$\alpha$}
\put(137,54){}
\end{picture}
\caption{Bing double of $K$ is infection on the trivial link along $\alpha$} using $K$ \label{fig:bingeta}
\end{figure}

 Another example is the ``$R_1$-doubling'' operation of going from the left-hand side of Figure~\ref{fig:ribbonCG} to the right-hand side ($m=2$). Most of the results of this paper concern to what extent these functions are injective. The point is that, because of the condition on ``winding numbers'', $lk(\eta_i,R)=0$,
if $R$ is a slice knot, the images of such operators $R$ contain only knots (or links) for which the classical Seifert-matrix-type invariants vanish. Moreover these operators respect the COT filtration.

\begin{prop}\label{prop:operatorsfiltration}(~\cite[proof of Proposition 3.1]{COT2}) If $R$ is a slice knot or link and $\eta_i\in \pi_1(S^3-R)^{(n)}$ then the operator $R_{\eta_i}$ satisfies
$$
R_{\eta_i}(\mathcal{F}_{0})\subset \mathcal{F}_{n}.
$$
\end{prop}

Thus iterations of these operators, \emph{iterated generalized doubling}, produce increasingly subtle knots and links. The family $J_n(K)$ is the result of $n$ iterations 
$$
\mathcal{C}\overset{R_1}\longrightarrow \mathcal{C}\to\dots\to \mathcal{C}\overset{R_1}\rightarrow\mathcal{C}
$$
applied to the initial knot $J_0=K$. More generally let us define an \emph{n-times iterated generalized doubling} to be precisely such a composition of operators using possibly different slice knots $R_j$, and different curves $\eta_{j1},\dots,\eta_{jm_j}$. Then our proof establishes:

\newtheorem*{thm:main3}{Theorem~\ref{thm:main3}}
\begin{thm:main3} If $R_j$, $1\leq j\leq n$, are slice knots and Arf($K$)$=0$, then the result, $R_n\circ\dots\circ R_1(K)$, of the n-times iterated generalized doubling lies in $\mathcal{F}_{n}$. If, additionally, , for each $j$, the submodule of the classical Alexander polynomial of $R_j$ generated by $\{\eta_{j1},\dots,\eta_{jm_j}\}$ contains elements $x,y$ such that $\mathcal{B}\ell_0^j(x,y)\neq 0$, where $\mathcal{B}\ell_0^j$ is the Blanchfield form of $R_j$, then there is a constant $C$, such that if the integral of the Levine signature function of $K$ is greater than $C$ in absolute value, then the resulting knot is not topologically slice, nor even in $\mathcal{F}_{n.5}$.
\end{thm:main3}

Note that any set $\{\eta_{j1},\dots,\eta_{jm_j}\}$ that generates a submodule whose $\mathbb{Q}$-rank is more than half the degree of the Alexander polynomial of $R_j$ necessarily satisfies the condition of Theorem~\ref{thm:main3}, because of the non-singularity of the Blanchfield form.

An analogous result for links (Theorem~\ref{thm:mainlink3}) is also shown.

\section{Higher-Order Signatures and How to Calculate Them}\label{signatures}

In this section we review the von Neumann $\rho$-invariants and explain to what extent they are concordance invariants. We also show how to calculate them for knots or links that are obtained from the infections defined in Section~\ref{sec:Introduction}. 

The use of variations of Hirzebruch-Atiyah-Singer signature defects associated to covering spaces is a theme common to most of the work in the field of knot and link concordance since the 1970's. In particular, Casson and Gordon initiated their use in cyclic covers ~\cite{CG1}~\cite{CG2}; Farber, Levine and Letsche initiated the use of signature defects associated to general (finite) unitary representations ~\cite{L6}~\cite{Let}; and Cochran-Orr-Teichner initiated the use of signatures associated to the left regular representations ~\cite{COT}. See ~\cite{Fr2} for a beautiful comparison of these approaches in the metabelian case.

Given a compact, oriented 3-manifold $M$, a discrete group $\G$, and a representation $\phi : \pi_1(M)
\to \G$, the \emph{von Neumann
$\rho$-invariant} was defined by Cheeger and Gromov by choosing a Riemannian metric and using $\eta$-invariants associated to $M$ and its covering space induced by $\phi$. It can be thought of as an oriented homeomorphism invariant associated to an arbitrary regular covering space of $M$ ~\cite{ChGr1}. If $(M,\phi) = \partial
(W,\psi)$ for some compact, oriented 4-manifold $W$ and $\psi : \pi_1(W) \to \G$, then it is known that $\rho(M,\phi) =
\s^{(2)}_\G(W,\psi) - \s(W)$ where $\s^{(2)}_\G(W,\psi)$ is the
$L^{(2)}$-signature (von Neumann signature) of the intersection form defined on
$H_2(W;\mathbb{Z}\G)$ twisted by $\psi$ and $\sigma(W)$ is the ordinary
signature of $W$ ~\cite{LS}. In the case that $\G$ is a poly-(torsion-free-abelian) group (abbreviated \textbf{PTFA group} throughout), it follows that $\mathbb{Z}\G$ is a right Ore domain that embeds into its (skew) quotient field of fractions $\mathcal{K}\G$ ~\cite[pp.591-592, ~Lemma 3.6ii p.611]{P}. In this case $\s^{(2)}_\G$  is a function of the Witt class of the equivariant intersection form on $H_2(W;\mathcal{K}\G)$ ~\cite[Section 5]{COT}. In the special case (such as $\beta_1(M)=1$) that this form is non-singular, it can be thought of as a homomorphism from
$L^0(\mathcal{K}\G)$ to $\mathbb{R}$. 

All of the coefficient systems $\G$ in this paper will be of the form $\pi/\pi^{(n)}_r$ where $\pi$ is the fundamental group of a space (usually a $4$-manifold) and $\pi^{(n)}_r$ is the $n^{th}$-term of the \textbf{rational derived series}. The latter was first considered systematically by Harvey. It is defined by
$$
\pi^{(0)}_r\equiv \pi,~~~ \pi^{(n+1)}_r\equiv \{x\in \pi^{(n)}_r|\exists k\neq 0, x^k\in [\pi_r^{(n)},\pi_r^{(n)}]\}.
$$
Note that $n^{th}$-term of the usual derived series $\pi^{(n)}$ is contained in the $n^{th}$-term of the rational derived series. For free groups and knot groups, they coincide. It was shown in ~\cite[Section 3]{Ha1} that $\pi/\pi^{(n)}_r$ is a PTFA group.
 
The utility of the von Neumann signatures lies in the fact that they obstruct knots from being slice knots. It was shown in ~\cite[Theorem 4.2]{COT} that, under certain situations, higher-order von Neumann signatures vanish for slice knots, generalizing the classical result of Murasugi and the results of Casson-Gordon. That proof fails for links, but the extension was later accomplished by Harvey (there is an extra obstruction). Moreover, Cochran-Orr-Teichner defined a filtration on knots and links and showed that certain higher-order signatures obstructed a knot's lying in a certain term of the filtration. Harvey also extended this to links. Here we state the needed results for slice knots and links. In an Section~\ref{sec:appendix} we review the filtration and the more general results. In the case of links we prove a more general result than Harvey's, which will be needed later.

First,

\begin{thm}\label{thm:oldsliceobstr}(Cochran-Orr-Teichner~\cite[Theorem 4.2]{COT}) If a knot $K$ is topologically slice in a rational homology $4$-ball and
$\phi:\pi_1(M_K)\to \G$ is a PTFA coefficient system that extends to the fundamental group of the exterior of the slicing disk, then $\rho(M_K,\phi)=0$.
\end{thm}

The analogous result for links has not specifically appeared, although it is implicit in and follows from the techniques of ~\cite{Ha2}. The proof will be given as a corollary of a more general in an appendix (Section~\ref{sec:appendix}).

\begin{thm}\label{thm:linksliceobstr}If a link $L$ is topologically slice in a rational homology $4$-ball and
$\phi:\pi_1(M_L)\to \G$ is a PTFA coefficient system that extends to the fundamental group of the exterior of the slicing disks, then $\rho(M_L,\phi)=0$.
\end{thm}

Some other useful properties of von Neumann
$\rho$-invariants are given below. One can find
detailed explanations of most of these in \cite[Section 5]{COT}. The last property, that for a fixed $3$-manifold, the set $\{\rho(M,\phi)\}$ is bounded above and below, is an analytical result of Cheeger and Gromov that we use in some (but not all) of our results here.

\begin{prop}
\label{prop:rho invariants}Let $M$ be a closed,oriented $3$-manifold and $\phi : \pi_1(M) \to \G$ as above.
\begin{itemize}
\item [(1)] If $(M,\phi) = \partial (W,\psi)$ for some compact
oriented 4-manifold $W$ such that the equivariant intersection form on $H_2(W;\mathcal{K}\G)/j_*(H_2(\partial W;\mathcal{K}\G))$ admits a half-rank summand on
which the form vanishes, then
$\s^{(2)}_\G(W,\psi)=0$ (see ~\cite[Lemma 3.1 and Remark 3.2]{Ha2} for a proper explanation of this for manifolds with $\beta_1>1$). Thus if  $\s(W) = 0$ then $\rho(M,\phi) = 0$ 

\item [(2)] If $\phi$ factors through $\phi' : \pi_1(M) \to \G'$ where
$\G'$ is a subgroup of $\G$, then $\rho(M,\phi') = \rho(M,\phi)$.

\item [(3)] If $\phi$ is trivial (the zero map), then $\rho(M,\phi) = 0$.

\item [(4)] If $M=M_K$ is zero surgery on a knot $K$ and $\phi:\pi_1(M)\to \mathbb{Z}$ is the abelianization, then $\rho(M,\phi)$ is equal to the integral over the circle of the Levine (classical) signature function of $K$, normalized so that the length of the circle is $1$ ~\cite[Prop. 5.1]{COT2}. \textbf{This real number will be denoted $\rho_0(K)$}.

\item [5] (Cheeger-Gromov ~\cite{ChGr1}) Given $M$, there is a positive constant $C_M$, the \textbf{Cheeger-Gromov constant of M}, such that for every $\phi$
$$
|\rho(M,\phi)|<C_M.
$$
\end{itemize}
\end{prop}

The following elementary lemma reveals the additivity of the $\rho$-invariant under infection. It is only slightly more general than \cite[Proposition 3.2]{COT2}. The use of a Mayer-Vietoris sequence to analyze the effect of a satellite construction on signature defects is common to essentially all of the previous work in this field (see for example ~\cite{Lith1}).

Suppose $L=R(\eta_i,K_i)$ is obtained by infection as described in Section~\ref{sec:Introduction}. Let the zero surgeries on $R$, $L$, and $K_i$ be denoted $M_R$ $M_L$, $M_i$ respectively. Suppose $\phi:\pi_1(M_L)\to \G$ is a map to an arbitrary PTFA group $\G$ such that, for each $i$, $\ell_i$, the longitude of $K_i$, lies in the kernel of $\phi$. Since $S^3-K_i$ is a submanifold of $M_L$, $\phi$ induces  a map on $\pi_1(S^3-K_i)$. Since $l_i$ lies in the kernel of $\phi$ this map extends uniquely to a map that we call $\phi_i$  on $\pi_1(M_i)$. Similarly, $\phi$ induces a map on $\pi_1(M_R-\coprod \eta_i)$. Since $M_R$ is obtained from $(M_R-\coprod \eta_i)$ by adding $m$ $2$-cells along the meridians of the $\eta_i$, $\mu(\eta_i)$ and $m$ $3-$cells, and since $\mu(\eta_i)=l_i^{-1}$ and $\phi_i(l_i)=1$, $\phi$ extends uniquely to $\phi_R$. Thus $\phi$ induces unique maps $\phi_i$ and $\phi_R$ on $\pi_1(M_i)$ and $\pi_1(M_R)$ (characterized by the fact that they agree with $\phi$ on $\pi_1(S^3-K_i)$ and $\pi_1(M_R-\coprod \eta_i)$ respectively). 

There is a very important case when the hypothesis above that $\phi(\ell_i)=1$ is always satisfied. Namely suppose $\G^{(n+1)}=1$ and $\eta_i\in \pi_1(M_R)^{(n)}$.  Since a longitudinal push-off of $\eta_i$, called $\ell_{\eta_i}$ or $\eta_i^+$, is isotopic to $\eta_i$ in the solid torus $\eta_i\times D^2\subset M_R$, $\ell_{\eta_i}\in \pi_1(M_R)^{(n)}$ as well. By ~\cite[Theorem 8.1]{C} or ~\cite{Lei3} it follows that $\ell_{\eta_i}\in \pi_1(M_L)^{(n)}$. Since $\mu_i$, the meridian of $K_i$, is identified to $\ell_{\eta_i}$, $\mu_i \in \pi_1(M_L)^{(n)}$ so $\phi(\mu_i)\in \G^{(n))}$  for each $i$.  Thus $\phi_i(\pi_1(S^3- K_i)^{(1)})\subset\G^{(n+1)}=\{e\}$ and in particular the longitude of each $K_i$ lies in the kernel of $\phi$. 

\begin{lem}\label{lem:additivity} In the notation of the two previous paragraphs (assuming $\phi(\ell_i)=0$ for all $i$),
$$
\rho(M_L,\phi) - \rho(M_R,\phi_R) = \sum^m_{i=1}\rho(M_i,\phi_i).
$$
Moreover if $\pi_1(S^3-K_i)^{(1)}\subset$ kernel($\phi_i$) then either $\rho(M_i,\phi_i)=\rho_0(K_i)$, or $\rho(M_i,\phi_i)=0$, according as $\phi_R(\eta_i)\neq 1$ or $\phi_R(\eta_i)= 1$. Specifically, if $\G^{(n+1)}=1$ and $\eta_i\in \pi_1(M_R)^{(n)}$ then this is the case.
\end{lem}

\begin{proof} Let $E$ be the $4$-manifold obtained from $M_R\times [0,1] \coprod -M_i\times [0,1]$ by identifying, for each $i$, the copy of $\eta_i\times D^2$ in $M_R\times \{1\}$ with the tubular neighborhood of $K_i$ in $M_i\times \{0\}$ as in Figure~\ref{fig:mickey}. 
\begin{figure}[htbp]
\setlength{\unitlength}{1pt}
\begin{picture}(150,150)
\put(0,0){\includegraphics[height=150pt]{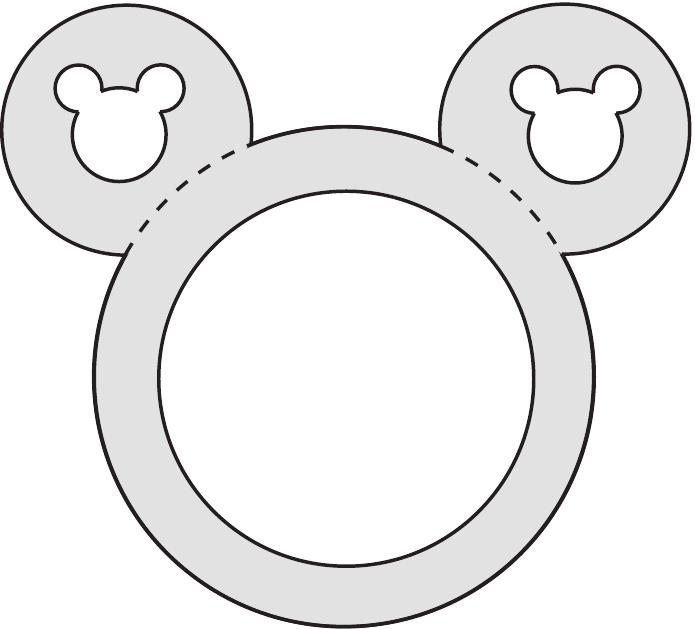}}
\put(54,37){$M_R\times [0,1]$}  
\put(-55,127){$M_1\times [0,1]$}
\put(170,127){$M_m\times [0,1]$}
\put(73,127){$\dots$}
\end{picture}
\caption{The cobordism $E$}\label{fig:mickey}
\end{figure}
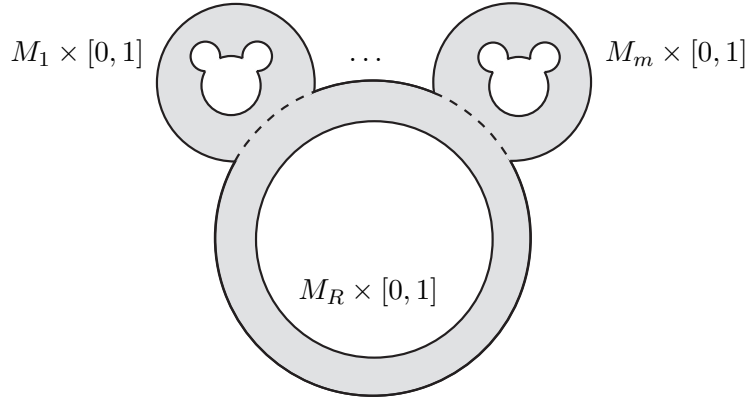
The dashed arcs in the figure represent the solid tori $\eta_i\times D^2$. Observe that the `outer' boundary component of $E$ is $M_L$.
Note that $E$ deformation retracts to $\overline{E}= M_L\cup (\coprod_i(\eta_i\times D^2))$, where each solid torus is attached to $M_L$ along its boundary. Hence $\overline{E}$ is obtained from $M_L$ by adding $m$ $2$-cells along the loops $\mu(\eta_i)=l_i$, and $m$ $3$-cells. Thus, by our assumption, $\phi$ extends uniquely to $\overline{\phi}:\pi_1(\overline{E})\to \G$ and hence $\ov\phi:\pi_1(E)\to\G$. Clearly the restrictions of $\overline{\phi}$ to $\pi_1(M_i)$ and $\pi_1(M_R\times \{0\})$ agree with $\phi_i$ and $\phi_R$ respectively. It follows that that
$$
\rho(M_L,\phi) - \rho(M_R,\phi_R) = \sum^m_{i=1}\rho(M_i,\phi_i) + \sigma^{(2)}(E,\overline{\phi})- \sigma(E).
$$
Now we claim that both the ordinary signature of $E$, $\s(E)$, as well as the $L^2$-signature $\s^{(2)}_\G(E)$, vanish. The first part of the proposition will follow immediately. 

\begin{lem}\label{lem:mickeysig} With respect to \emph{any} coefficient system, $\phi:\pi_1(E)\to \Gamma$,  the signature of the equivariant intersection form on the $H_2(E;\mathbb{Z}\G)$ is zero. 
\end{lem}
\begin{proof}[Proof of Lemma~\ref{lem:mickeysig}] We show that all of the (twisted) second homology of $E$ comes from its boundary. This immediately implies the claimed results.

Consider the Mayer-Vietoris sequence with coefficients twisted by $\phi$:
$$
H_2(M_R\x~I)\oplus_i H_2(M_{i}\x~I)\ra H_2(E)\ra H_1(\amalg\eta_i\x D^2)\ra H_1(M_R\x I)\oplus_i H_1(M_i\x I).
$$
We claim that each of the inclusion-induced maps
$$
H_1(\eta_i\x D^2)\ra H_1(M_i)
$$
is injective. If $\phi(\eta_i)=1$ then, since $\eta_i$ is equated to the meridian of $K_i$, $\phi(\mu_{K_i})=1$. Since $\mu_{K_i}$ normally generates $\pi_1(M_i)$, it follows that the coefficient systems on $\eta_i\x D^2$ and $M_i$ are trivial and hence the injectivity follows from the injectivity with $\mathbb{Z}$-coefficients, which is obvious since $\mu_{K_i}$ generates $H_1(M_i)$. Suppose now that that $\phi(\eta_i)\neq 1$. Since $\eta_i\x D^2$ is homotopy equivalent to a circle, it suffices to consider the cell structure on $S^1$ with one $1$-cell. Then the boundary map in the $\mathbb{Z}[\pi_1(S^1)]$ cellular chain complex for $S^1$ is multiplication by $t-1$ so the boundary map in the equivariant chain complex 
$$
C_1\otimes \mathbb{Z}\G \overset{\partial\otimes id}\longrightarrow C_0\otimes \mathbb{Z}\G
$$
is easily seen to be left multiplication by $\phi(\eta_i)-1$. Since $\phi(\eta_i)\neq 1$ and $\mathbb{Z}\G$ is a domain, this map is injective. Thus $H_1(\eta_i\x D^2;\mathbb{Z}\G)=~0$ so injectivity holds.

Now using the Mayer-Vietoris sequence, any element of $H_2(E)$ comes from $H_2(M_R\x\{0\})\oplus_i H_2(M_i\x\{0\})$, in particular from $H_2(\p E)$. Thus the intersection form on $H_2(E)$ is identically zero and any signature vanishes.
\end{proof}

This completes the proof of the first part of the proposition.

If $\pi_1(S^3-K_i)^{(1)}\subset$ kernel($\phi_i$) then $\phi_i$ factors through the abelianization of $H_1(S^3\backslash K_i)$ and so by parts $2,3$ and $4$ of Proposition~\ref{prop:rho invariants}, we are done. In particular if $\G^{(n+1)}=\{e\}$ and $\eta_i)\in \pi_1(M_R)^{(n)}$, then $\phi_i(\mu_i)\in\G^{(n)}$ for each $i$ as we have shown in the paragraph above the Lemma, so $\phi_i(\pi_1(S^3\backslash K_i)^{(1)})\subset\G^{(n+1)}=\{e\}$. Thus each $\phi_i$ factors through the abelianization.
\end{proof}

We want to collect, in the form of a Lemma,  the properties of the cobordism $E$ that we have established in the proofs above. These will be used in later sections.

\begin{lem}\label{lem:mickeyfacts} With regard to $E$ as above, the inclusion maps induce
\begin{itemize}
\item [(1)] an epimorphism $\pi_1(M_L)\to \pi_1(E)$ whose kernel is the normal closure of the longitudes of the infecting knots $K_i$ viewed as curves $\ell_i\subset S^3-K_i\subset M_L$;
\item [(2)] isomorphisms $H_1(M_L)\to H_1(E)$ and $H_1(M_R)\to H_1(E)$;
\item [(3)] and isomorphisms $H_2(E)\cong H_2(M_L)\oplus_i H_2(M_{K_i})\cong H_2(M_R)\oplus_i H_2(M_{K_i})$.
\item [(4)] The longitudinal push-off of $\eta_i$, $\ell_{\eta_i}\subset M_L$ is isotopic in $E$ to $\eta_i\subset M_R$ and to the meridian of $K_i$ , $\mu_i\subset M_{K_i}$.
\item [(5)] The longitude of $K_i$, $\ell_i\subset M_{K_i}$ is isotopic in $E$ to the reverse of the meridian of $\eta_i$, $\eta_i^{-1}\subset M_L$ and to the longitude of $K_i$ in $S^3-K_i\subset M_L$ and to the reverse of the meridian of $\eta_i$, $(\mu_{\eta_i})^{-1}\subset M_R$ (the latter bounds a disk in $M_R$).
\end{itemize}
\end{lem}
\begin{proof} We saw above that $E\sim\overline{E}$ is obtained from $M_L$ by adding $m$ $2$-cells along the loops $\mu_{\eta_i}=\ell_i$, and then adding $m$ $3$-cells that go algebraically zero over these $2$-cells. Property $(1)$ and the first part of properties $(2)$ and $(3)$ follow. Properties $(4)$ and $(5)$ are obvious from the definitions of infection and of $E$. Since we have assumed that $\eta_i$ are null-homologous in $M_R$, the second parts of properties $(2)$ and $(3)$ follow from an easy Mayer-Vietoris argument as in the proof just above.
\end{proof}

\vspace{4in}

\section{Higher-Order Blanchfield forms for knots and links}\label{Blanchfieldforms}

We have seen in Lemma~\ref{lem:additivity} that an infection will have an effect on a $\rho$-invariant only if the infection circle $\eta$ survives under the map defining the coefficient system. For example if one creates a knot $J$ by infecting a slice knot $R$ along a curve $\eta$ that dies in $\pi_1(B^4-\Delta)$ for some slice disk $\Delta$ for $R$, then this infection will have no effect on the $\rho$-invariants associated to any coefficient system that extends over $B^4-\Delta$. Therefore it is important to prove \emph{injectivity} theorems concerning $\pi_1(S^3-R)\to\pi_1(B^4-\Delta)$, that is to locate elements of $\pi_1(S^3-R)$ that survive under such inclusions. Moreover the curve $\eta$ must usually lie deep in the derived series of $\pi_1(S^3-R)$ to ensure that the infected knot cannot be detected by a less subtle invariant. If $\eta\in \pi_1(S^3-R)^{(n)}$ then $J$ will be rationally $n$-solvable and we may hope to show that it is not $(n.5)$-solvable. Therefore, loosely speaking, we need to be able to prove that $\eta$ survives under the map
$$
j_*:\pi_1(S^3-R)^{(n)}/\pi_1(S^3-R)^{(n+1)}\to \pi_1(B^4-\Delta)^{(n)}/\pi_1(B^4-\Delta)^{(n+1)}.
$$
For $n=1$ this is a question about ordinary Alexander modules and was solved by Casson-Gordon and Gilmer using linking forms on finite branched covers. In general this seems a daunting task. (Note that this is impossible if $\pi_1(B^4-\Delta)$ is solvable, which occurs, for example, for the standard slice disk for the ribbon knot $R_1$ of Figure~\ref{fig:ribbonCG}(e.g. see ~\cite{FrT})). To see that higher-order Alexander modules are relevant to this task, observe that the latter quotient is the abelianization of $\pi_1(B^4-\Delta)^{(n)}$ and thus can be interpreted as $H_1(W_n)$ where $W_n$ is the (solvable) covering space of $B^4-\Delta$ corresponding to the subgroup $\pi_1(B^4-\Delta)^{(n)}$. Such modules were named \emph{higher-order Alexander modules} in ~\cite{COT}~\cite{C}~\cite{Ha1}. We will employ higher-order Blanchfield linking forms on higher-order Alexander modules to find restrictions on the kernels of such maps. The logic of the technique is entirely analogous to the classical case ($n=1$): Any two curves $\eta_0, \eta_1$, say, that lie in the kernel of $j_*$ must satisfy $\mathcal{B}\ell(\eta_0,\eta_0)=\mathcal{B}\ell(\eta_0,\eta_1)=\mathcal{B}\ell(\eta_1,\eta_1)=0$ with respect to a higher order linking form $\mathcal{B}\ell$. Our major new insight is that, if the curves lie in a submanifold $S^3-K\hookrightarrow S^3-J$, a situation that arises whenever $J$ is formed from $R$ by infection using a knot $K$, then the values (above) of the higher-order Blanchfield form of $J$ can be expressed in terms of the values of the classical Blanchfield form of $K$! 

Higher-order Alexander modules and higher-order linking forms for classical knot exteriors and for closed $3$-manifolds with $\beta_1(M)=1$ were introduced in ~\cite[Theorem 2.13]{COT} and further developed in ~\cite{C} and ~\cite{Lei1}. These were defined on the so called higher-order Alexander modules. Higher-order Alexander modules for links and $3$-manifolds in general were defined and investigated in ~\cite{Ha1}. Blanchfield forms for $3$-manifolds with $\beta_1(M)>1$ were only recently defined by Leidy ~\cite{Lei3}. It is crucial to our techniques that we work with such Blanchfield forms without localizing the coefficient systems, as was investigated in ~\cite{Lei1}~\cite{Lei3}. It is in this aspect that our work deviates from that of ~\cite{COT}~\cite{COT2}~\cite{CK}. A non-localized Blanchfield form for knots also played a crucial role in ~\cite{FrT}.

First we recall that \emph{higher-order Blanchfield linking forms} have been defined under fairly general circumstances.
\begin{thm}\label{thm:blanchfieldexist}[~\cite[Theorem 2.3]{Lei3}] Suppose $M$ is a closed, connected, oriented $3$-manifold and $\phi:\pi_1(M)\to \Lambda$ is a PTFA coefficient system. Suppose $\mathcal{R}$ is a classical Ore localization of the Ore domain $\mathbb{Z}\Lambda$ (so $\mathbb{Z}\Lambda\subset\mathcal{R}\subset \mathcal{K}\Lambda$). Then there is a linking form:
$$
\mathcal{B}l^M_{\mathcal{R}}: TH_1(M;\mathcal{R})\to (TH_1(M;\mathcal{R}))^{\#}\equiv \overline{Hom_{\mathcal{R}}(TH_1(M;\mathcal{R}), \mathcal{K}\Lambda/\mathcal{R})}.
$$
\end{thm}

An \emph{Ore localization} of $\mathbb{Z}\Lambda$ is $\mathcal{R}=\mathbb{Z}\Lambda[S^{-1}]$ for some right-Ore set $S$ ~\cite{Ste}. When we speak of the \emph{unlocalized} Blanchfield form we mean that $\mathcal{R}=\mathbb{Z}\Lambda$ or $\mathcal{R}=\mathbb{Q}\Lambda$. $TH_1(M;\mathcal{R})$ denotes the $\mathcal{R}$-torsion submodule. In general $TH_1(M;\mathcal{R})$ need not have homological dimension one nor even be finitely-generated, and these linking forms are \emph{singular}.

Leidy analyzed the effect of an infection on the unlocalized Blanchfield forms in ~\cite{Lei1}\cite{Lei3}. This generalizes the result on the classical Blanchfield form for satellite knots ~\cite{LiM}. If $L$ is obtained by infection on a link $R$ along a circle $\alpha$ using the knot $K$ and $\phi:\pi_1(M_L)\to \Lambda$ is a PTFA coefficient system, and $\mathbb{Z}\Lambda\subset\mathcal{R}\subset \mathcal{K}\Lambda$ then $\mathcal{B}l^{L}_{\mathcal{R}}$ is defined. On the other hand, by definition, exterior of the knot $K$ is a submanifold of $M_L$ and there is an induced coefficient system, that we also call $\phi$, with respect to which there is a Blanchfield linking form (first defined in ~\cite[Theorem 2.13]{COT})
$$
\mathcal{B}l_{\mathcal{R}}^K: TH_1(S^3-K;\mathcal{R})\to (TH_1(S^3-K;\mathcal{R}))^{\#}.
$$
(We note that if $\phi$ is nontrivial when restricted to $\pi_1(S^3-K)$ then $TH_1(S^3-K;\mathcal{R})=H_1(S^3-K;\mathcal{R})$. Otherwise $TH_1(S^3-K;\mathcal{R})=0$ ~\cite[Proposition 2.11]{COT}). Then it is an easy exercise for the reader using the geometric definition of these Blanchfield forms (or see ~\cite[Theorem 4.6, proof of property 1]{Lei1}), that these forms are compatible:
\begin{prop}\label{prop:submanifoldcompatible}~\cite[Theorem 3.7]{Lei3} In the situation above the following diagram commutes
\begin{equation}\label{diag:compatible}
\begin{diagram}\dgARROWLENGTH=1em
\node{TH_1(S^3-K;\mathcal{R})} \arrow{e,t}{i_*}
\arrow{s,r}{\mathcal{B}l^K_\mathcal{R}}\node{TH_1(M_L;\mathcal{R})} \arrow{s,r}{\mathcal{B}l^{M_L}_\mathcal{R}}\\
\node{TH_1(S^3-K;\mathcal{R})^{\#}} \node{TH_1(M_L;\mathcal{R})^{\#}}\arrow[1]{w,t}{i^{\#}}
\end{diagram}
\end{equation}
that is, for all $x,y\in H_1(S^3-K;\mathcal{R})$ 
$$
\mathcal{B}l^{M_L}_{\mathcal{R}}(i_*(x),i_*(y))= \mathcal{B}l^K_\mathcal{R}(x,y).
$$
\end{prop}

Moreover, in some important situations, the induced coefficient system $\phi:\pi_1(S^3-K)\to \Lambda$ factors through, $\mathbb{Z}$, the abelianization of the knot exterior. In particular if $L$ is obtained by infection on a link $R$ along a circle $\alpha\in \pi_1(M_R)^{(k-1)}$ where $\Lambda^{(k)}=1$, then this is the case. Furthermore the higher-order Blanchfield form $\mathcal{B}l^K_\Lambda$ is merely the classical Blanchfield form on the classical Alexander module, ``tensored up''. What is meant by this is the following. Supposing that $\phi$ is both nontrivial and factors through the abelianization, the induced map $\text{image}(\phi)\equiv\mathbb{Z}\hookrightarrow \Lambda$ is an embedding so it induces embeddings
$$
\phi:\mathbb{Q}[t,t^{-1}]\hookrightarrow \mathbb{Q}\Lambda,~~~ \phi:\mathbb{Q}(t)\hookrightarrow \mathcal{K}\Lambda,
$$
and hence an embedding
$$
\ov\phi:\mathbb{Q}(t)/\mathbb{Q}[t,t^{-1}]\hookrightarrow \mathcal{K}\Lambda/\mathbb{Q}\Lambda.
$$
Then there is an isomorphism 
$$
H_1(S^3\backslash K;\mathbb{Q}\Lambda)\cong H_1(S^3\backslash K;\mathbb{Q}[t,t^{-1}])\otimes_{\mathbb{Q}[t,t^{-1}]}\mathbb{Q}\Lambda \cong \mathcal{A}_0(K)\otimes_{\mathbb{Q}[t,t^{-1}]}\mathbb{Q}\Lambda,
$$
where $\mathcal{A}_0(K)$ is the classical (rational) Alexander module of $K$ and where $\mathbb{Q}\Lambda$ is a $\mathbb{Q}[t,t^{-1}]$-module via the map $t\to \phi(\alpha)$ ~\cite[Theorem 8.2]{C}. Moreover
$$
\mathcal{B}l_{\Lambda}^K(x\otimes 1,y\otimes 1)=\ov\phi(\mathcal{B}l^K_0(x,y))
$$
for any $x,y\in \mathcal{A}_0(K)$, where $\mathcal{B}l_0^K$ is the classical Blanchfield form on the rational Alexander module of $K$ ~\cite[Proposition 3.6]{Lei3}~\cite[Theorem 4.7]{Lei1} (see also ~\cite[Section 5.2.2]{Cha2}). 

Then, finally, Leidy shows that the Blanchfield form on $M_L$ the sum of that on $H_1(M_R)$ and that on the infecting knot $K$ (generalizing the classical result for satellites ~\cite{LiM}). We state this below although, in this paper, we shall not need this nontrivial fact that the module $H_1(M_L;\mathbb{Q}\Lambda)$ decomposes, nor even that $\mathcal{A}_0(K)\otimes_{\mathbb{Q}[t,t^{-1}]}\mathbb{Q}\Lambda$ is a submodule of it. We will only need the almost obvious fact that the inclusion of the $3$-manifolds $S^3-K_i\hookrightarrow M_L$ induces a (natural) map on the Blanchfield forms and that the induced Blanchfield form on $S^3-K$ is the classical form ``tensored up''.

\begin{thm}\label{thm:decomposition}[Theorem 3.7,Proposition 3.4 ~\cite{Lei3}] Suppose $L=R(\alpha_i,K_i)$ is obtained by infection as above with $\alpha_i\in \pi_1(M_R)^{(k-1)}$ for all $i$. Let the zero surgeries on $R$, $L$, and $K_i$ be denoted $M_R$ $M_L$, $M_i$ respectively. Suppose $\Lambda$ is a PTFA group such that $\Lambda^{(k)}=1$. Suppose $\phi:\pi_1(M_L)\to \Lambda$ is a coefficient system. Then the inclusions induce an isomorphism
$$
H_1(M_R;S^{-1}\mathbb{Z}\Lambda)\oplus_{i\in A} H_1(S^3\backslash K_i;S^{-1}\mathbb{Z}\Lambda)\overset{i_*}{\to} H_1(M_L;S^{-1}\mathbb{Z}\Lambda).
$$
where $A=\{i~|~\phi((\alpha_i)^+\neq 1\}$. Moreover there is an isomorphism 
$$
H_1(S^3\backslash K_i;\mathbb{Q}[t,t^{-1}])\otimes_{\mathbb{Q}[t,t^{-1}]}S^{-1}\mathbb{Z}\Lambda \cong H_1(S^3\backslash K_i;S^{-1}\mathbb{Z}\Lambda).
$$
Restricting to $S^{-1}\mathbb{Z}\Lambda=\mathbb{Q}\Lambda$ for simplicity, for any $x,y\in H_1(S^3\backslash K_i;\mathbb{Q}[t,t^{-1}])$,
$$
\mathcal{B}l_{\mathbb{Q}\Lambda}^{M_L}(i_*(x\otimes 1),i_*(y\otimes 1))=\ov\phi_i(\mathcal{B}l_0^i(x,y))
$$
where $\mathcal{B}l_{\Lambda}^{M_L}$ is the Blanchfield form on $M_L$ induced by $\phi$, $\mathcal{B}l_0^i$ is the classical Blanchfield form on the classical rational Alexander module of $K_i$, and
$$
\ov\phi_i: \mathbb{Q}(t)/\mathbb{Q}[t,t^{-1}]\to \mathcal{K}\Lambda/\mathbb{Q}\Lambda
$$
is the monomorphism induced by $\phi:\mathbb{Z}\to \Lambda$ sending $1$ to $\phi(\alpha_i)$.
\end{thm}

\textbf{Remarks:} Under our hypotheses the coefficient system $\phi$ extends over the cobordism $E$, as in the discussion preceding Lemma~\ref{lem:additivity}, and there is a unique induced coefficient system $\phi_R$ on $M_R$. By Property $(4)$ of Lemma~\ref{lem:mickeyfacts}, $\alpha_i$ and its longitudinal push-off $\alpha_i^+$ are isotopic in $E$ so $\phi((\alpha_i)^+)=\phi_R(\alpha_i)$. Thus $\phi((\alpha_i)^+)\neq 1$ if and only if $\phi_R(\alpha_i\neq 1)$. Moreover, since the meridian of $K_i$ is equated to $(\alpha_i)^+$, $\phi_i(\mu_i)=\phi((\alpha_i)^+)=\phi_R(\alpha_i)$.

 The following is perhaps the key result of the paper, that we use to establish certain ``injectivity'' as discussed in the first paragraph of this section. Recall that the notions of \emph{(n)-solvable} and \emph{rationally (n)-solvable} are defined in Section~\ref{sec:appendix}. For the reader who is just concerned with proving that knots and links are not slice, replace the hypothesis below that ``$W$ is a rational $(k)$-solution for $M_L$'' with the hypothesis that ``$L$ is a slice link and $W$ is the exterior in $B^4$ of a set of slice disks for $L$''. Such an exterior is a rational $(k)$-solution for any $k$.

\begin{thm}\label{thm:nontriviality} Suppose $L=R(\alpha_i,K_i)$ is obtained by infection. Let the zero surgeries on $R$, $L$, and $K_i$ be denoted $M_R$ $M_L$, $M_i$ respectively. Suppose $\alpha_i\in \pi_1(M_R)^{(k-1)}$ for all $i$. Suppose $W$ is a rational $(k)$-solution for $M_L$, $\Lambda$ is a PTFA group such that $\Lambda^{(k)}=1$, and $\psi:\pi_1(W)\to \Lambda$ is a nontrivial coefficient system whose restriction to $\pi_1(M_L)$ is denoted $\phi$. Let $A=\{i~|~\phi((\alpha_i)^+)\neq 1\}$. For each $i\in A$ , let $P_i$ be the kernel of the composition
$$
\mathcal{A}_0(K_i)\overset{id\otimes 1}\lra  ~(\mathcal{A}_0(K_i) \otimes_{\mathbb{Q}[t,t^{-1}]}\mathbb{Q}\Lambda)\overset{i_*}{\to} H_1(M_L;\mathbb{Q}\Lambda)\overset{j_*}\to H_1(W;\mathbb{Q}\Lambda).
$$
Then $P_i\subset P_i^\perp$ with respect to $\mathcal{B}l_0^i$, the classical Blanchfield linking form on the rational Alexander module, $\mathcal{A}_0(K_i)$, of $K_i$.
\end{thm}

\textbf{Remark:} Under the hypotheses of Theorem~\ref{thm:nontriviality}, the coefficient system extends over the cobordism $E$ of Figure~\ref{fig:mickey} and hence extends to $\pi_1(M_R)$. If this extension is (sloppily) also called $\phi$ then $\phi(\alpha_i)=\phi((\alpha_i)^+)$ since $\alpha_i$ and its longitude $(\alpha_i)^+$ are isotopic in $M_R$ and hence freely homotopic in $E$.

\begin{proof}[Proof of Theorem~\ref{thm:nontriviality}] We need the following result that was proved in ~\cite[Lemma 4.5, Theorem 4.4]{COT} in the special case that $\beta_1(M)=1$. The proof in this more general case is identical, except for Lemma~\ref{lem:fourmanBlanch}.

\begin{lem}\label{lem:selfannihil} Suppose $M$ is connected and is rationally
$(k)$-solvable via $W$ and
$\phi:\pi_1(W)\ra\Lambda$ is a non-trivial coefficient system where
$\Lambda$ is a PTFA group with $\Lambda^{(k)}=1$. Let $\mathcal{R}$ be an Ore localization of $\mathbb{Z}\Lambda$ so $\mathbb{Z}\Lambda\subset\mathcal{R}\subset \mathcal{K}\Lambda$. Then 
$$
TH_2(W,M;\mathcal{R})\xrightarrow{\partial}TH_1(M;\mathcal{R})\xrightarrow{j_{\ast}} TH_1(W;\mathcal{R})
$$
is exact. Moreover, any submodule $P\subset \text{kernel}~j_*$ satisfies
$P\subset ($ker~$j_*)^\perp\subset P^\perp$ with respect to the Blanchfield form on $TH_1(M;\mathcal{R})$.
\end{lem}

\begin{proof}[Proof of Lemma~\ref{lem:selfannihil}] Let $2m= \text{rank}_\mathbb{Q}(H_2(W;\mathbb{Q}))$. Let
$\{\ell_1,\ell_2,\ldots, \ell_m\}$ generate a rational
$k$-Lagrangian for $W$ and $\{d_1,d_2,\ldots, d_m\}$ its
$k$-duals in $H_2(W;\mathbb{Z}[\pi_1(W)/\pi_1(W)^{(k)}])$. Since $\Lambda^{(k)}=1$, $\phi$ factors through
$\phi' : \pi_1(W)/\pi_1(W)^{(k)} \lra \Lambda$. We denote by $\ell_i'$ and
$d_i'$ the images of $\ell_i$ and $d_i$ in $H_2(W;\mathcal{R})$. By
naturality of intersection forms, the intersection form $\lambda$
defined on $H_2(W;\mathcal{R})$ vanishes on the module generated by
$\{\ell_1',\ell_2',\ldots, \ell_m'\}$. Let $\mathcal{R}^m\oplus \mathcal{R}^m$ be the
free module on $\{\ell_i',d_i'\}$ and let $A^*$ denote $\overline{Hom_\mathcal{R}(A,\mathcal{R})}$ for any right $\mathcal{R}$-module $A$.
The following composition
$$
\mathcal{R}^m\oplus \mathcal{R}^m \xrightarrow{j_*} H_2(W;\mathcal{R}) \xrightarrow{\lambda}
H_2(W;\mathcal{R})^* \xrightarrow{j^*} (\mathcal{R}^m\oplus \mathcal{R}^m)^*
$$
is represented by a block matrix
$$
\left(\begin{matrix} 0 & I\cr I & X
\end{matrix}\right).
$$
This matrix has an inverse which is
$$
\left(\begin{matrix} -X & I\cr I & 0
\end{matrix}\right).
$$
Thus the composition is an isomorphism. This implies that both $j^*$ and $j^*\circ \lambda$ are epimorphisms. Consequently the rank of $H_2(W;\mathcal{R})$ is at least $2m$. But by ~\cite[Proposition 4.3]{COT}, the rank of $H_2(W;\mathcal{R})$ is at most $2m$ and so is precisely $2m$. Hence the rank of $(H_2(W;\mathcal{R}))^*$ is also $2m$. Thus the kernel of $j^{\ast}$ is the torsion submodule of $(H_2(W;\mathcal{R}))^*$. But the latter is torsion-free since $\mathcal{R}$ is a domain. Hence $j^{\ast}$ is an isomorphism and $(H_2(W;\mathcal{R}))^*$ is free. It follows that $\lambda$ is surjective and hence $H_2(W,; \mathcal{R})$ is the direct sum of a free module of rank $2m$ and its torsion submodule.  Now consider the commutative diagram below with
$\mathcal{R}$-coefficients.
$$
\begin{diagram}\dgARROWLENGTH=1.2em
\node{H_2(W)} \arrow{e,t}{\pi_{\ast}}
\arrow{sse,b}{\lambda} \node{H_2(W,\partial W)} \arrow{e,t}{\partial}
\arrow{s,lr}{\cong}{P.D.}\node{H_1(M)}
\arrow{e,t}{j_*} \node{H_1(W)}\\
\node[2]{H^2(W)} \arrow{s,r}{\kappa}\\
\node[2]{H_2(W)^*}
\end{diagram}
$$ 
Note $\kappa$ is a split surjection between modules of the same rank and thus the kernel of $\kappa\circ\operatorname{P.D.}$ is torsion. Now,
given $p\in TH_1(M;\mathcal{R})$ such that $j_*(p)=0$, choose $x$ such that
$\partial x=p$. Let $y$ be an element of the set $\lambda^{-1}(\kappa\circ
\operatorname{P.D.}(x))$. Then $\partial(x-\pi_{\ast}(y))=p$ and
$x-\pi_{\ast}(y)$ is torsion since it lies in the  kernel of
$\kappa\circ\operatorname{P.D.}$. Thus we have shown that every element of  ker$j_*$ is in the image of an element of $TH_2(W,M;\mathcal{R})$. This concludes the proof of the first part of Lemma~\ref{lem:selfannihil}.

For the second part we need:

\begin{lem}\label{lem:fourmanBlanch} There is a Blanchfield form, $\mathcal{B}l^{rel}$,
$$
\mathcal{B}l^{rel}:~TH_2(W,\partial W;\mathcal{R})\to TH_1(W)^{\#}
$$
such that the following diagram, with coefficients in $\mathcal{R}$ unless specified otherwise, is commutative up to sign:
$$
\begin{diagram}\label{diag:compatible2}\dgARROWLENGTH=1em
\node{TH_2(W,\partial W;\mathcal{R})} \arrow{e,t}{\partial_*}
\arrow{s,r}{\mathcal{B}l^{rel}_\mathcal{R}}\node{TH_1(M;\mathcal{R})} \arrow{s,r}{\mathcal{B}l^{M}_\mathcal{R}}\\
\node{TH_1(W;\mathcal{R})^{\#}} \arrow{e,t}{j^{\#}}\node{TH_1(M;\mathcal{R})^{\#}}
\end{diagram}
$$
\end{lem}

\begin{proof}[Proof of Lemma~\ref{lem:fourmanBlanch}] (See also ~\cite[Lemmas 3.2, 3.3]{Cha3})  Consider the following commutative diagram where homology and cohomology is with $\mathcal{R}$ coefficients unless specified and $\mathcal{K}$ denotes the quotient field of $\mathcal{R}$:
$$
\begin{diagram}[small]
\node{H_3(W,M;\mathcal{K})} \arrow[2]{e,t}{\partial_*} \arrow[2]{s,l}{P.D.} \arrow{se} \node[2]{H_2(M;\mathcal{K})} \arrow{s,-} \arrow{se}\\
\node[2]{H_3(W,M;\mathcal{K}/\mathcal{R})} \arrow[2]{e} \arrow[2]{s} \node{} \arrow{s} \node{H_2(M;\mathcal{K}/\mathcal{R})} \arrow[2]{s} \\
\node{\overline{H^1(W;\mathcal{K})}} \arrow[2]{s,l}{\kappa} \arrow{se} \arrow{e,-} \node{} \arrow{e} \node{\overline{H^1(M;\mathcal{K})}} \arrow{s,-} \arrow{se} \\
\node[2]{\overline{H^1(W;\mathcal{K}/\mathcal{R})}} \arrow[2]{e} \arrow[2]{s} \node{} \arrow{s} \node{\overline{H^1(M;\mathcal{K}/\mathcal{R})}} \arrow[2]{s} \\
\node{\overline{\Hom_{\mathcal{R}}(H_1(W),\mathcal{K})}} \arrow[2]{s,l}{\iota} \arrow{se} \arrow{se} \arrow{e,-} \node{} \arrow{e} \node{\overline{\Hom_{\mathcal{R}}(H_1(M),\mathcal{K})}} \arrow{s,-} \arrow{se} \\
\node[2]{\overline{\Hom_{\mathcal{R}}(H_1(W),\mathcal{K}/\mathcal{R})}} \arrow[2]{e} \arrow[2]{s} \node{} \arrow{s} \node{\overline{\Hom_{\mathcal{R}}(H_1(M),\mathcal{K}/\mathcal{R})}} \arrow[2]{s} \\
\node{\overline{\Hom_{\mathcal{R}}(TH_1(W),\mathcal{K})}} \arrow{se} \arrow{se} \arrow{e,-} \node{} \arrow{e} \node{\overline{\Hom_{\mathcal{R}}(TH_1(M),\mathcal{K})}} \arrow{se} \\
\node[2]{\overline{\Hom_{\mathcal{R}}(TH_1(W),\mathcal{K}/\mathcal{R})}} \arrow[2]{e,t}{j^{\#}} \node[2]{\overline{\Hom_{\mathcal{R}}(TH_1(M),\mathcal{K}/\mathcal{R})}} \\
\end{diagram}
$$
where $\iota$ is the map induced from the inclusion map of the torsion submodule. Since
$$
\overline{\Hom_{\mathcal{R}}(TH_1(W;\mathcal{R}),\mathcal{K})}=0,
$$
it follows that the image of $H_3(W,M;\mathcal{K}) \to H_3(W,M;\mathcal{K}/\mathcal{R})$ is contained in the kernel of the composition $\iota \circ \kappa\circ P.D.$. Furthermore, from the exact sequence,
$$
H_3(W,M;\mathcal{K})\overset{\pi}{\to} H_3(W,M;\mathcal{K}/\mathcal{R})\to H_2(W,M;\mathcal{R})\to H_2(W,M;\mathcal{K})
$$
since $H_2(W,M;\mathcal{K})$ is $\mathcal{R}$-torsion-free, $TH_2(W,M;\mathcal{R})$ is isomorphic to the cokernel of $\pi$. It follows that there is a well-defined map $\mathcal{B}l^{rel}_\mathcal{R}:
TH_2(W,M;\mathcal{R}) \to TH_1(W;\mathcal{R})^{\#}$. Similarly, since
$$
\overline{\Hom_{\mathcal{R}}(TH_1(M;\mathcal{R}),\mathcal{K})}=0,
$$
there is a well-defined map
$\mathcal{B}l^{M}_\mathcal{R}: TH_1(M,\mathcal{R}) \to TH_1(M;\mathcal{R})^{\#}$ such that the following diagram commutes.
$$
\begin{diagram}[small]
\node{H_3(W,M;\mathcal{K}/\mathcal{R})} \arrow[2]{s,l}{\iota \circ \kappa
\circ P.D.} \arrow[2]{e} \arrow{se} \node[2]{H_2(M;\mathcal{K}/\mathcal{R})} \arrow{s,-} \arrow{se} \\
\node[2]{TH_2(W,M)} \arrow{sw,r}{\mathcal{B}l^{rel}_\mathcal{R}} \arrow[2]{e,t}{\partial_*\hspace{.25in}} \node{} \arrow{s} \node{TH_1(M)} \arrow{sw,r}{\mathcal{B}l^{M}_\mathcal{R}}\\
\node{TH_1(W)^{\#}} \arrow[2]{e,t}{j^{\#}} \node[2]{TH_1(M)^{\#}} \\
\end{diagram}
$$
\end{proof}
Finally we can complete the proof of Lemma~\ref{lem:selfannihil}. Suppose $P\subset\text{kernel}~j_*\subset TH_1(M;\mathcal{R})$. Suppose $x\in P$ and $y\in$ kernel~$j_*$. According to the first part of the lemma,
we have $x=\partial_*(\tilde{x})$ for some $\tilde{x}\in TH_2(W,M;\mathcal{R})$. Thus by Diagram~\ref{diag:compatible2},
$$
\mathcal{B}l^M_\mathcal{R}(x)(y)=\mathcal{B}l^M_\mathcal{R}(\partial_*\tilde{x})(y)=\tilde{j^{\#}}(\mathcal{B}l^{rel}_\mathcal{R}(\tilde{x}))(y)= \mathcal{B}l^{rel}(\tilde{x})(j_*(y))=0
$$
since $j_*(y)=0$. Hence $P\subset ~($ker$j_*)^\perp \subset P^\perp$ with respect to the Blanchfield form on $TH_1(M;\mathcal{R})$.

This concludes the proof of Lemma~\ref{lem:selfannihil}.
\end{proof}

We continue with the proof of Theorem~\ref{thm:nontriviality}. Suppose $x,y\in P_i$ as in the statement. Let $\mathcal{R}=\mathbb{Q}\Lambda$, $M=M_L$ and let $P$ be the submodule of $H_1(M_L;\mathbb{Q}\Lambda)$  generated by $\{i_*(x\otimes 1),i_*(y\otimes 1)\}$. Then $P\subset \text{kernel}~j_*$. Apply Lemma~\ref{lem:selfannihil} to conclude that
$$
\mathcal{B}l^{M_L}_{\mathbb{Q}\Lambda}(i_*(x\otimes 1)),(i_*(y\otimes 1))=0.
$$
By Theorem~\ref{thm:decomposition},
$$
\ov\phi_i(\mathcal{B}l_o^i(x,y))=0.
$$
Since $\ov\phi$ is a monomorphism by hypothesis, it follows that $\mathcal{B}l_o^i(x,y)=0$. Thus $P_i\subset P_i^\perp$ with respect to the classical Blanchfield form on $K_i$. This concludes the proof of Theorem~\ref{thm:nontriviality}.
\end{proof}

\section{The family $J_n$}\label{familyJ}

In this section we prove our main theorem that the family of knots $J_n$ of Figure~\ref{fig:family} contains many non-slice knots. The simple ideas of the proof can be lost in the details of the induction, so we first present a proof of the simplest new result. Recall that $J_0=K$ and $J_n=J_n(K)$ is obtained from $J_0$ by applying the ``operator'' $R_1$ $n$ yielding the inductive definition of Figure~\ref{fig:family}. Recall also that, for $J_2$, all classical invariants as well as those of Casson-Gordon vanish.

\begin{thm}\label{thm:J2notslice} There is a constant $C$ such that if $|\rho_0(J_0)|>C$ then $J_2$ is not a slice knot.
\end{thm}

\begin{rem}\label{rem:betterJ2} This theorem can be improved but this result will appear in a later paper: There is a constant $C$ (which may be $0$) such that if $J_2$ is slice then $\rho_0(J_0)\in \{0,C\}$. 
\end{rem}

\begin{proof}[Proof of Theorem~\ref{thm:J2notslice}] Let $R_2$ be the ribbon knot $J_2(U)$, that is, start with $R_1$ and tie the knots $R_1$ into each band as shown in Figure~\ref{fig:R2}. Let $C$ be the Cheeger-Gromov constant of $M_{R_2}$, that is, a positive constant such that:
$$
|\rho(M_{R_2},\phi)| < C
$$
for \emph{any} homomorphism $\phi:\pi_1(M_{R_2})\to \G$ ($\G$ arbitrary). Now assume that $|\rho_0(J_0)|>C$. We proceed by contradiction. Suppose $J_2$ were slice and let $V$ denote the exterior of a slice disk. Thus $\partial V=M_{J_2}$. Let $M=M_{J_2}$, $\pi=\pi_1(V)$, $\G=\pi/\pi^{(3)}_r$ and let $\phi$ denote both the projection $\pi\to \G$ and its restriction to $\pi_1(M)$. Recall that $\pi^{(3)}_r$ denoted the third term of the \textbf{rational} derived series as defined in Section~\ref{signatures}. Recall that $\G$ is a PTFA group by ~\cite[Section 3]{Ha1}. By Remark~\ref{rem:n-solvable}, $V$ is a $(2.5)$-solution for $M$ so by Theorem~\ref{thm:oldsliceobstr} 
$$
\rho(M,\phi)=0.
$$

We reach a contradiction by computing $\rho(M,\phi)$ in another way. We will argue that $J_2$ may be obtained from a ribbon knot $R_2=J_2(U)$ by infections on $4$ curves. Recall that $J_1$ can be obtained from the ribbon knot $R_1$ by infection on $2$ curves labelled $\{\eta^1_+, \eta^1_-\}$ in Figure~\ref{fig:exs_eta2} using the knot $J_0$ as the infecting knot in each case. 

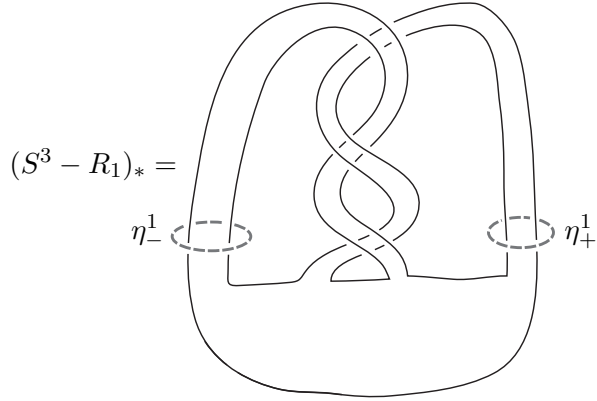
\begin{figure}[htbp]
\setlength{\unitlength}{1pt}
\begin{picture}(200,160)
\put(0,0){\includegraphics[height=150pt]{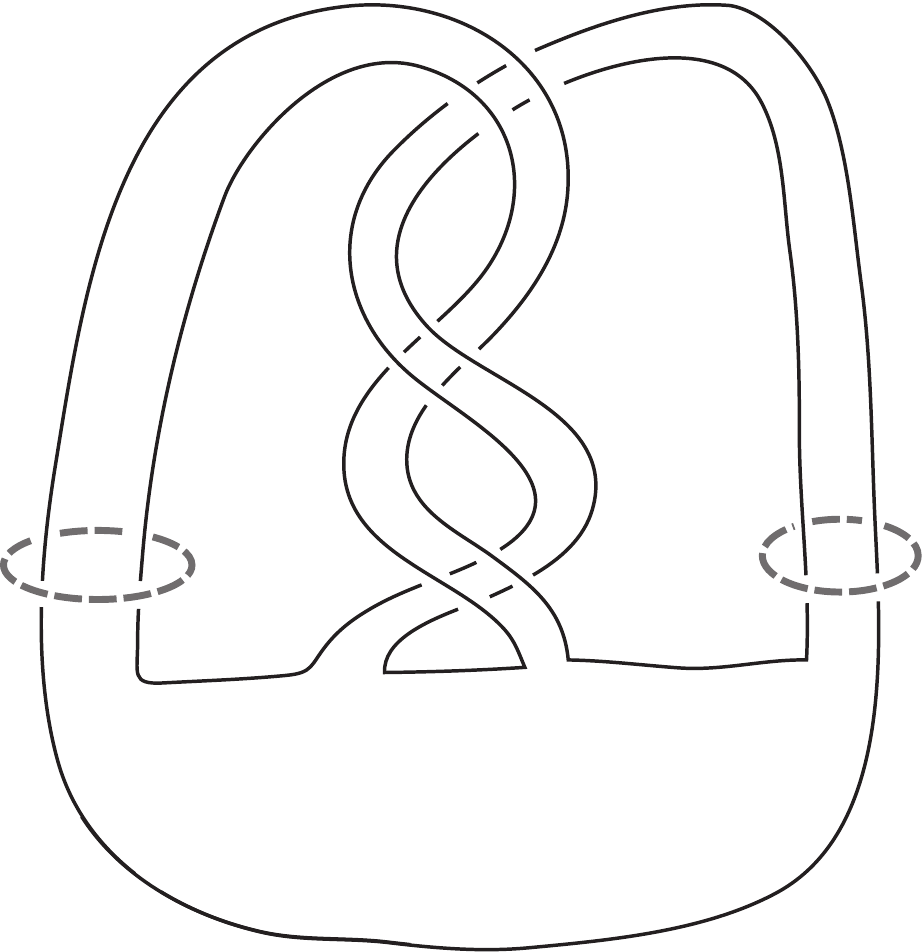}}
\put(-61,85){$(S^3-R_1)_*=$}
\put(-15,60){$\eta^1_-$}
\put(149,60){$\eta^1_+$}
\end{picture}
\caption{$J_{1}$ obtained from $R_1$ by infection on $\eta_\pm ^1$}\label{fig:exs_eta2}
\end{figure}
\noindent Instead of replacing neighborhoods of $\eta_\pm ^1$ by copies of $S^3-J_0$, merely leave them as marked curves to be replaced later. Denote this disguised description of $S^3-J_1$ by $(S^3-R_1)_*$. Now recall that we form $J_2$ from $R_1$ by replacing solid toral neighborhoods of $\{\eta ^1_\pm\}$ by two copies of $S^3-J_1$, which we now think of as two copies $(S^3-R_1)_*^\pm$. If we ignore the marked circles, we obtain a ribbon knot, denoted $R_2$, obtained by infection on $R_1$ along $\{\eta^1_\pm\}$ using the knot $R_1$ as the infecting knot in each case. Thus there exist $4$ marked circles in $S^3-R_2$ (two for each of the two marked infection knots $R_1$). Neighborhoods of these four marked circles must be replaced by copies of $S^3-J_0$ in order to complete the formation of $J_2$. 
These are shown in Figure~\ref{fig:R2}. 
\begin{figure}[htbp]
\setlength{\unitlength}{1pt}
\begin{picture}(500,170)
\put(40,20){\includegraphics[width=5 in]{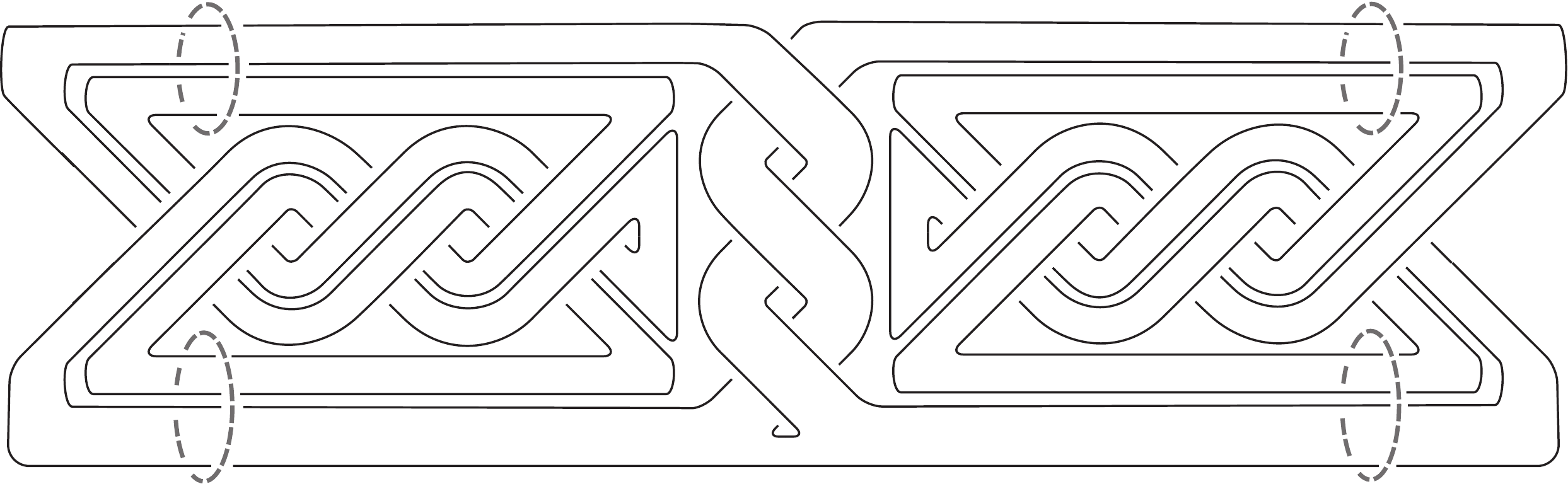}}
\put(77,5){$\eta^2_{++}$}
\put(77,137){$\eta^2_{+-}$}
\put(347,5){$\eta^2_{-+}$}
\put(337,137){$\eta^2_{--}$}
\end{picture}
\caption{$J_2$ obtained from $R_2$ by $4$ infections}\label{fig:R2}
\end{figure}
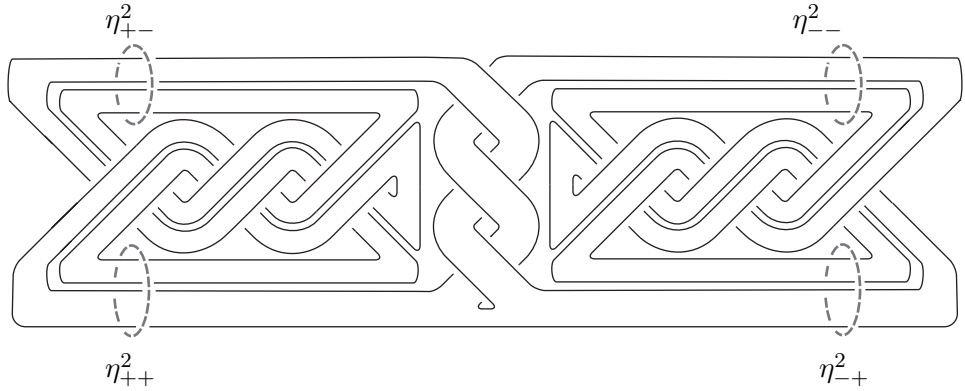
Thus we have shown shown that $J_2$ may be obtained from a ribbon knot $R_2$ by infections on $4$ curves $\{\eta_i\}$=$\{\eta^2_{++},\eta^2_{+-},\eta^2_{-+},\eta^2_{--}\}$, and these curves clearly lie in $\pi_1(S^3-R_2)^{(2)}$ (this last assertion will be discussed in more detail in our general proof later in this section).

Using this knowledge, by our additivity result, Lemma~\ref{lem:additivity} (use $n=2$),
$$
\rho(M,\phi)-\rho(M_{R_2},\phi_{R_2})=\sum^4_{i=1}\epsilon_i\rho_0(K)
$$
where $\epsilon_i=0$ or $1$ according as $\phi(\eta_i)=1$ or not. Since $\rho(M,\phi)=0$ and
$$
|\rho_0(K)|>C> |\rho(M_{R_2},\phi_{R_2})|,
$$
this is a contradiction if any $\epsilon_i\neq 0$. Thus the proof has quickly been reduced to showing that at least one of the curves $\eta_i$ (their longitudinal push-offs) survives in $\pi^{(2)}_r/\pi_r^{(3)}$. But by definition
$$
\pi^{(2)}_r/\pi_r^{(3)}=(\pi^{(2)}_r/[\pi^{(2)}_r,\pi^{(2)}_r])/(\mathbb{Z}-\text{torsion})
$$
so 
$$
\pi^{(2)}_r/\pi_r^{(3)}\hookrightarrow \pi^{(2)}_r/[\pi^{(2)}_r,\pi^{(2)}_r])\otimes_\mathbb{Z}\mathbb{Q}.
$$
The latter has a strictly homological interpretation as the first homology with $\mathbb{Q}$ coefficients of the covering space of $V$ whose fundamental group is $\pi^{(2)}_r$. In other words
$$
\pi^{(2)}_r/[\pi^{(2)}_r,\pi^{(2)}_r])\otimes_\mathbb{Z}\mathbb{Q}\cong H_1(V;\mathbb{Q}[\pi/\pi^{(2)}_r]).
$$
Therefore it suffices to prove that one of the curves $\eta_i$ survives under the map
$$
H_1(M;\mathbb{Q}[\pi/\pi^{(2)}_r])\overset{j_*}\to H_1(V;\mathbb{Q}[\pi/\pi^{(2)}_r]).
$$
Since $J_2$, by definition, is obtained from $R_1$ by infections on the two curves $\{\eta^1_+, \eta^1_-\}$ using the knot $J_1$ as the infecting knot in each case and where $\eta_{\pm}^1\in\pi_1(M_{R_1})^{(1)}$, we may apply Theorem~\ref{thm:nontriviality} with $k=2$, $\Lambda=\pi/\pi^{(2)}_r$, $\{\alpha_0,\alpha_1\}=\{\eta^1_+, \eta^1_-\}$ and $\psi:\pi_1(V)\to \Lambda$ the quotient map. This gives us information about the two compositions
$$
\mathcal{A}_0(J_1)\overset{id\otimes 1}\lra  ~(\mathcal{A}_0(J_1) \otimes_{\mathbb{Q}[t,t^{-1}]}\mathbb{Q}\Lambda)\overset{i_*}{\to} H_1(M;\mathbb{Q}\Lambda)\overset{j_*}\to H_1(V;\mathbb{Q}\Lambda),
$$
one for each of the two infecting knots $J_1$. But the four curves $\eta_i=\{\eta^2_{++},\eta^2_{+-},\eta^2_{-+},\eta^2_{--}\}$ in question are precisely the generators of these two copies of $\mathcal{A}_0(J_1)$ (which we can identify with $\mathcal{A}_0(R_1)$ since they have the same Seifert matrix). By Theorem~\ref{thm:nontriviality} the kernels of these inclusions satisfy
$P_{\pm}\subset P_{\pm}^\perp$ with respect to the classical Blanchfield form on $\mathcal{A}_0(J_1)$ as long as $\psi(\eta^1_\pm)\neq 1$. Assuming momentarily that, say, $\psi(\eta^1_+)\neq 1$, it cannot be that both of the corresponding $\eta^1_{++}$ and $\eta^1_{+-}$ lie in the kernel $P_+$ since then the submodule they generate, $\mathcal{A}_0(J_1)$, would be contained in $P_+$, contradicting the nonsingularity of the Blanchfield from on $\mathcal{A}_0(J_1)$. More generally, the reader will see that the precise condition needed at this point is that, in the submodule of $\mathcal{A}_0(R_1)$ generated by $\eta^1_{++}$ and $\eta^1_{+-}$, is contained some $x$ and $y$ such that $\mathcal{B}\ell_0(x,y)\neq 0$. This implies that at least one of the two curves must survive. Thus we are reduced to showing that either $\psi(\eta^1_+)\neq 1$ or $\psi(\eta^1_-)\neq 1$. Note that we began by seeking to show that at least one of the $\{\eta^2_{++},\eta^2_{+-},\eta^2_{-+},\eta^2_{--}\}$ does not map into $\pi^{(3)}_r$ and we reduced this to showing that at least one of the curves $\{\eta^1_+, \eta^1_-\}$ does not map into $\pi^{(2)}_r$. In the proof of our more general result below this induction continues downwards. Here, this can be accomplished by again using Theorem~\ref{thm:nontriviality} this time with coefficient system $\lambda=\pi/\pi^{(1)}_r$  ($k=1$) and proceeding just as above, this time using the fact that $J_2$ can be obtained from $R_2$ by infection on the \emph{meridian} of $R_2$. However, we do not include details since the fact we need is just the well-known classical result that the kernel of the inclusion from the rational Alexander module of a slice knot $J_2$ to the Alexander module of its slice exterior $V$ is self-annihilating with respect to the Blanchfield form and hence cannot include the entire generating set $\{\eta^1_\pm\}$. This completes the proof.
\end{proof}

Our main theorem is then:

\begin{thm}\label{thm:main} For any $n\ge0$ there is a constant $C_n$ such that, if $|\rho_0(K)|> C_n$, then $J_n(K)$ is of infinite order in the topological concordance group. Moreover $J_n$ is rationally $(n)$-solvable ($(n)$-solvable if Arf$(K)=0$), yet no non-zero multiple of $J_n$ is rationally $(n.5)$-solvable.
\end{thm}

\begin{rem} Using a different proof one can choose the constant $C_n$ independent of $n$. This will appear in another paper.
\end{rem}

\begin{cor}\label{cor:CG} For any $n\ge 1$ there exist knots $J\in\mathcal{F}_{(n-1)}$ for which the knot $R_1(J)$, shown in Figure~\ref{fig:family3}, is not a slice knot nor even in $\mathcal{F}_{n.5}$. 
\end{cor}

\begin{figure}[htbp]
\setlength{\unitlength}{1pt}
\begin{picture}(200,160)
\put(0,0){\includegraphics[height=150pt]{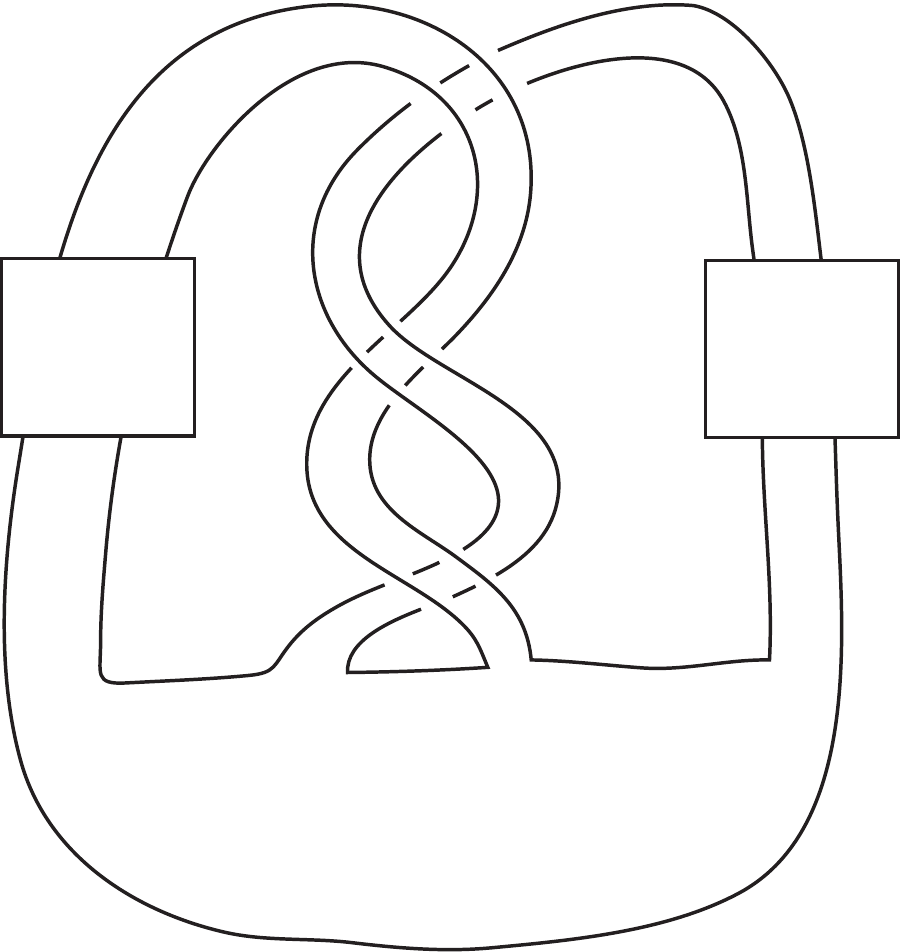}}
\put(-50,70){$R_1(J)~=$}
\put(9,92){$J$}
\put(124,92){$J$}
\end{picture}
\caption{}\label{fig:family3}
\end{figure}

\begin{proof}[Proof of Corollary~\ref{cor:CG}] Let $J=J_{n-1}(K)$ for some $K$ with $|\rho_0(K)|> C_n$ (for example a connected sum of a suitably large even number of trefoil knots). Then the knot on the right-hand side of Figure~\ref{fig:ribbonCG} is merely $J_n(K)$ which, by Theorem~\ref{thm:main}, is $(n)$-solvable hence in $\mathcal{F}_{(n)}$, but is not slice nor even rationally $(n.5)$-solvable; hence not in $\mathcal{F}_{(n.5)}$. Since $J\in \mathcal{F}_{(n-1)}$, if $n\geq 2$ then $J$ is algebraically slice and if $n\geq 3$ then $J$ has vanishing Casson-Gordon invariants ~\cite[Theorem 9.11]{COT}.
\end{proof}
As another immediate consequence, using the knots $J_n$ of Theorem~\ref{thm:main}, we have an easier proof of the following major result of Cochran and Teichner:

\begin{cor}\label{cor:CT} (Cochran-Teichner ~\cite{CT}) For any $n\ge0$, $\FF_n/\FF_{n.5}$ has rank at least 1.
\end{cor}
\begin{proof}[Proof of Corollary~\ref{cor:CT}] The knot $J_n$ wherein $J_0=K$ is a suitably large connected sum of an even number of trefoil knots is an element of infinite order in $\FF_n/\FF_{n.5}$ by Theorem~\ref{thm:main}.
\end{proof}

\begin{proof}[Proof of Theorem \ref{thm:main}] The proof follows the lines of the proof of Theorem~\ref{thm:J2notslice}, but the inductions are notationally complicated. 

First we establish that $J_n(K)$ has an alternative description as the result of $2^n$ infections on the ribbon knot $R_n=J_n(U)$ using the knot $K$ as the infecting knot each time, along curves that lie in $\pi_1(S^3\backslash R_n)^{(n)}$. This will be established as part of a much more general result that says that $J_n(K)$ has many alternative descriptions due to its `fractal' nature.

To this end note that if $K$ is the trivial knot $U$ then it is easily seen by induction that each $J_n(U)$ is a ribbon knot that we denote $R_n$, $n\ge 0$, as shown in Figure~\ref{fig:ribbonfamily} (set $R_0=U$). 

\begin{figure}[htbp]
\setlength{\unitlength}{1pt}
\begin{picture}(200,160)
\put(0,0){\includegraphics[height=150pt]{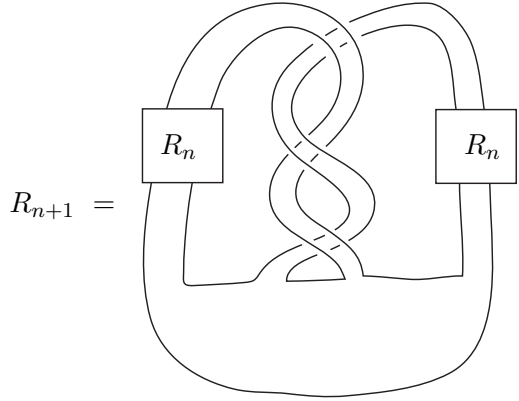}}
\put(-50,70){$R_{n+1}~=$}
\put(7,92){$R_{n}$}
\put(122,92){$R_{n}$}
\end{picture}
\caption{The recursive family of ribbon knots $R_{n+1}$}\label{fig:ribbonfamily}
\end{figure}

\noindent First, note that, for each $1\le i\le n$, because of the alternative description of infection as described in Section~\ref{sec:Introduction}, there are two inclusion maps
$$
f_\pm^{i}: S^3-R_{i-1}\ra S^3- R_{i}
$$
as suggested by Figure~\ref{fig:mapsf}. 

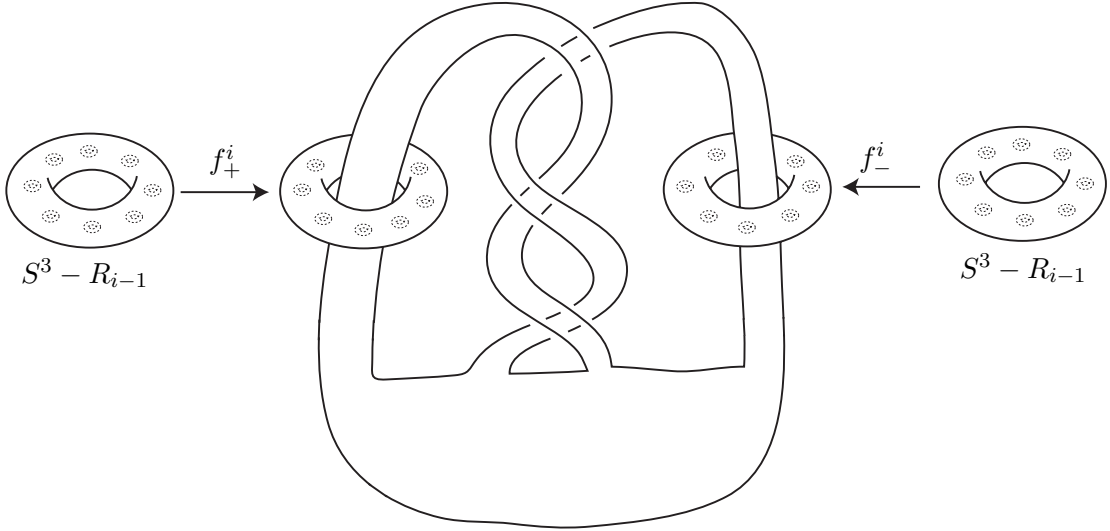
\begin{figure}[htbp]
\setlength{\unitlength}{1pt}
\begin{picture}(400,200)
\put(0,0){\includegraphics[height=200pt]{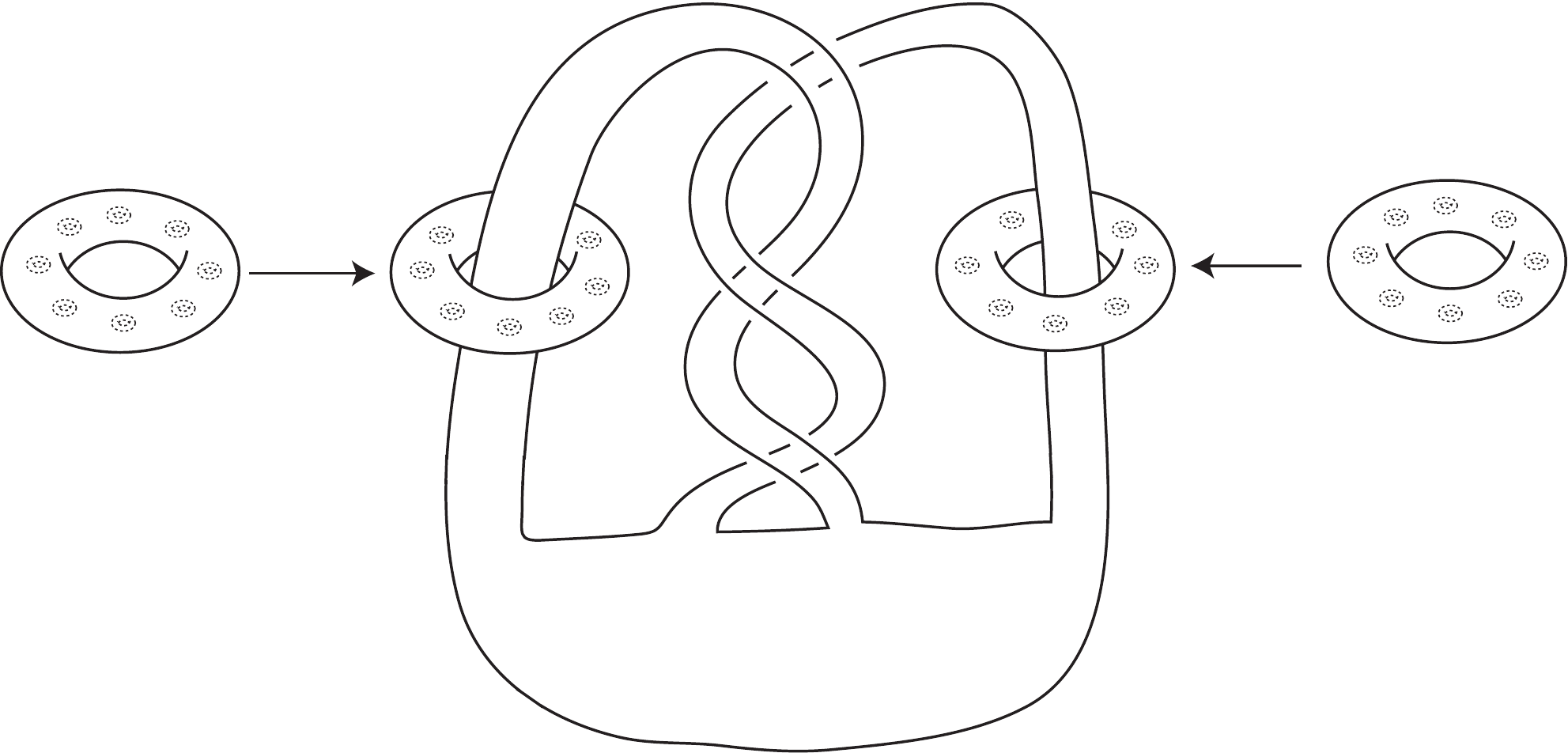}}
\put(77,136){$f^{i}_+$}
\put(323,137){$f^{i}_-$}
\put(6,93){$S^3-R_{i-1}$}
\put(361,95){$S^3-R_{i-1}$}
\put(95,71){}
\put(306,71){}
\put(122,11){}
\end{picture}
\caption{The embeddings  $S^3-R_{i-1}\hookrightarrow S^3-R_{i}$}\label{fig:mapsf}
\end{figure}
Let $\eta^0$ denote the meridian of $R_0$, the trivial knot. Let $\eta^1_+, \eta^1_-$ denote the two images $f_\pm^1(\eta^0)$ in $S^3-R_1$, shown in Figure~\ref{fig:exs_eta2} We call these `clones' of $\eta^0$. More generally, let $\{\eta^i_*\}$ denote the set of $2^i$ images of $\eta^0$ under the $2^i$ compositions $f_\pm^{i}\circ\dots\circ f_\pm^1$. Note that the induced maps
$$
(f_\pm^i)_*: \pi_1(S^3\backslash R_{i-1})\ra\pi_1(S^3\backslash R_{i})
$$
have images contained in the commutator subgroup. Thus the composition
$$
(f_\pm^{i})_*\circ\dots\circ(f_\pm^1)_*:  \pi_1(S^3\backslash R_0)\ra\pi_1(S^3\backslash R_1)^{(1)}\ra\dots\ra\pi_1(S^3\backslash R_i)^{(i)}
$$
has image in $\pi_1(S^3\backslash R_i)^{(i)}$. Therefore we see that each of the clones $\{\eta^i_*\}$ lies in $\pi_1(S^3\backslash R_i)^{(i)}$ and in particular each of the clones $\{\eta_*^n\}$ lies in $\pi_1(S^3\backslash R_n)^{(n)}$. The superscript $i$ of $\{\eta^i_*\}$ can serve to remind the reader in which term of the derived series it lies. 

The following establishes that $J_n(K)$ has a variety of different descriptions.

\begin{prop}\label{prop:altdescriptions} For any knot $K$ and $i$, $0\leq i \leq n$, $J_n(K)$ can be obtained from $R_i$ by multiple infections along the $2^i$ clones
$$
\{\eta^i_*\}= ~\{f^{i}_{\pm}\circ\dots\circ f^1_{\pm}(\eta^0)\},
$$
using knot $J_{n-i}(K)$ as the infecting knot in each case, and each clone $\eta^i_*$ lies in $\pi_1(S^3-R_i)^{(i)}$.
\end{prop}
\begin{proof} We proceed by `induction' on $i$. In the base case, $i=0$, for any $n$, there is only one clone, namely $\eta^0$ itself. Then the claim is merely that if one infects the unknot by $J_n(K)$ along a meridian then the result is $J_n(K)$, which is obviously true.

Assume that the Proposition is true for some fixed $i-1$ for \emph{any} $n$ such that $n\geq i-1$. Then consider fixed $i$ and arbitrary $n$ subject to $n\geq i$. Recall that $S^3-J_n(K)$ can be obtained by deleting the two solid tori as shown in the Figure~\ref{fig:family2} and replacing them with two copies of $S^3-J_{n-1}(K)$.
\begin{figure}[htbp]
\setlength{\unitlength}{1pt}
\begin{picture}(400,200)
\put(0,0){\includegraphics[height=200pt]{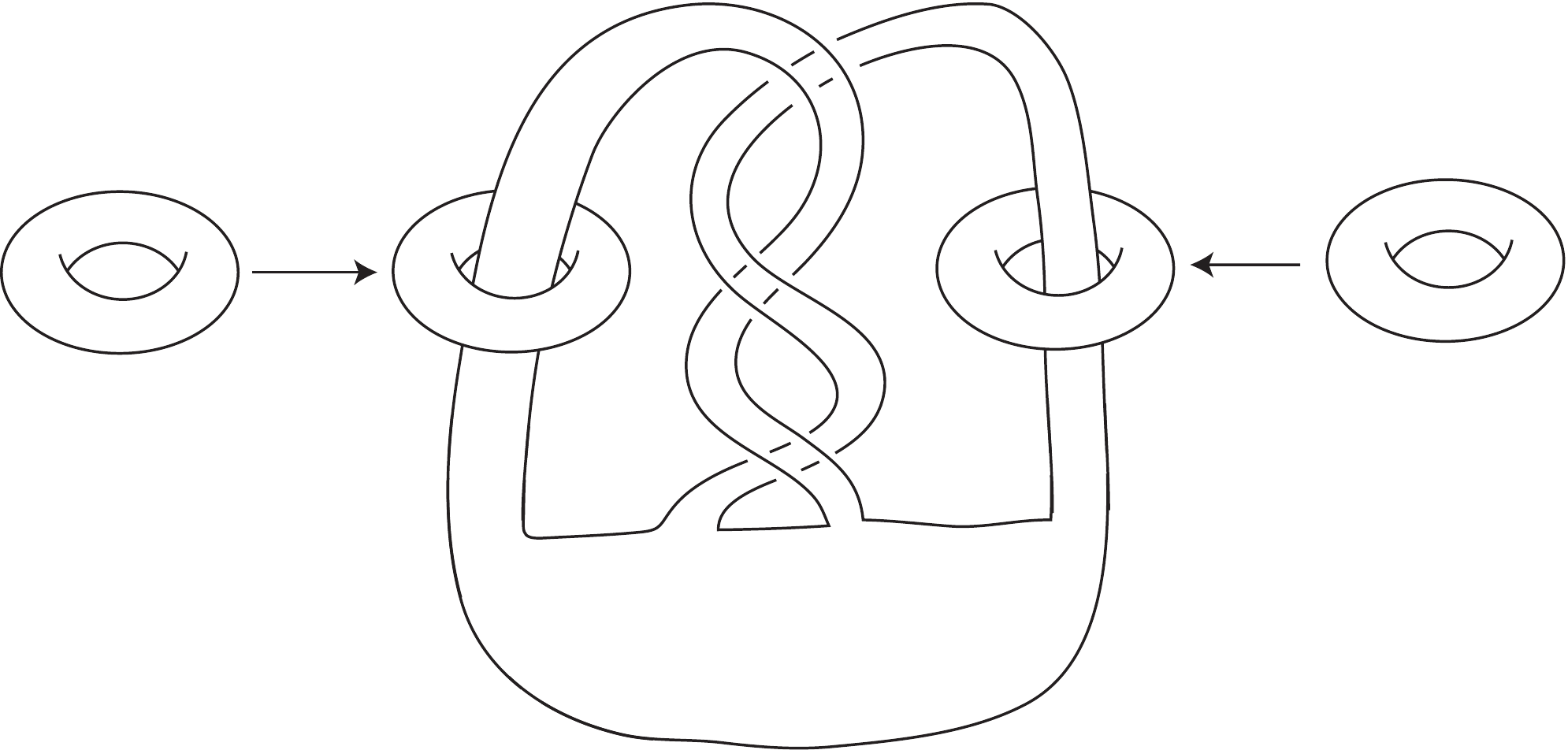}}
\put(77,136){}
\put(323,137){}
\put(6,93){$S^3-J_{n-1}$}
\put(361,95){$S^3-J_{n-1}$}
\put(95,71){}
\put(306,71){}
\put(122,11){}
\end{picture}
\caption{One definition of $S^3-J_{n}$}\label{fig:family2}
\end{figure}
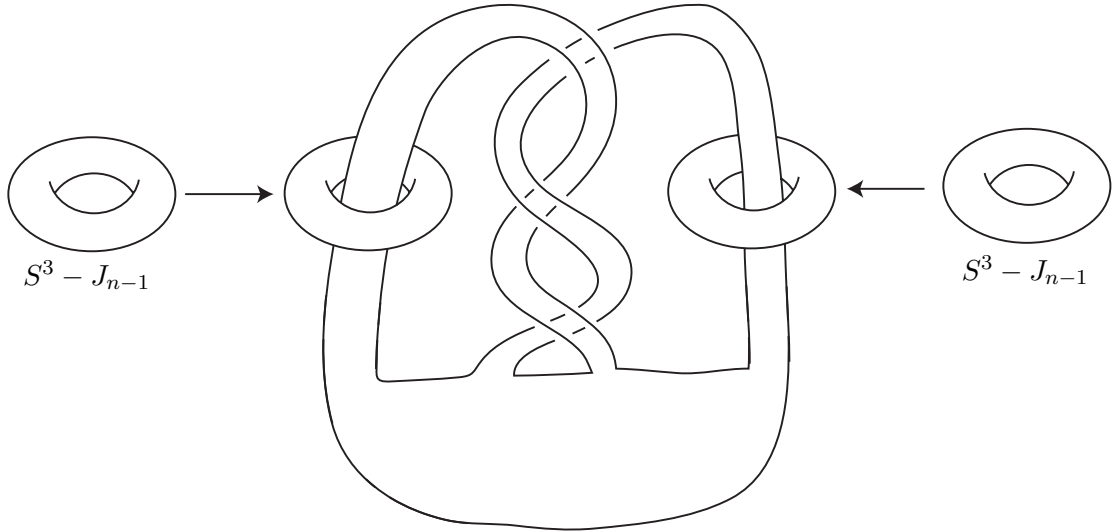
By the inductive hypothesis for (n-1,i-1),  $S^3-J_{n-1}$ can be obtained from $S^3-R_{i-1}$ by infections on the $2^{i-1}$ clones $\{\eta^{i-1}_*\}\equiv ~\{f^{i-1}_{\pm}\circ\dots\circ f^1_{\pm}(\eta^0)\}$ (shown schematically by the very small solid tori in Figure~\ref{fig:alternativeviews} ) using the knot $J_{n-i}(K)$ as the infecting knot in each case. Thus replacing the $2^i$ solid tori shown in Figure~\ref{fig:alternativeviews} by copies of $S^3-J_{n-i}(K)$ yields $S^3-J_n$.
\begin{figure}[htbp]
\setlength{\unitlength}{1pt}
\begin{picture}(400,200)
\put(0,0){\includegraphics[height=200pt]{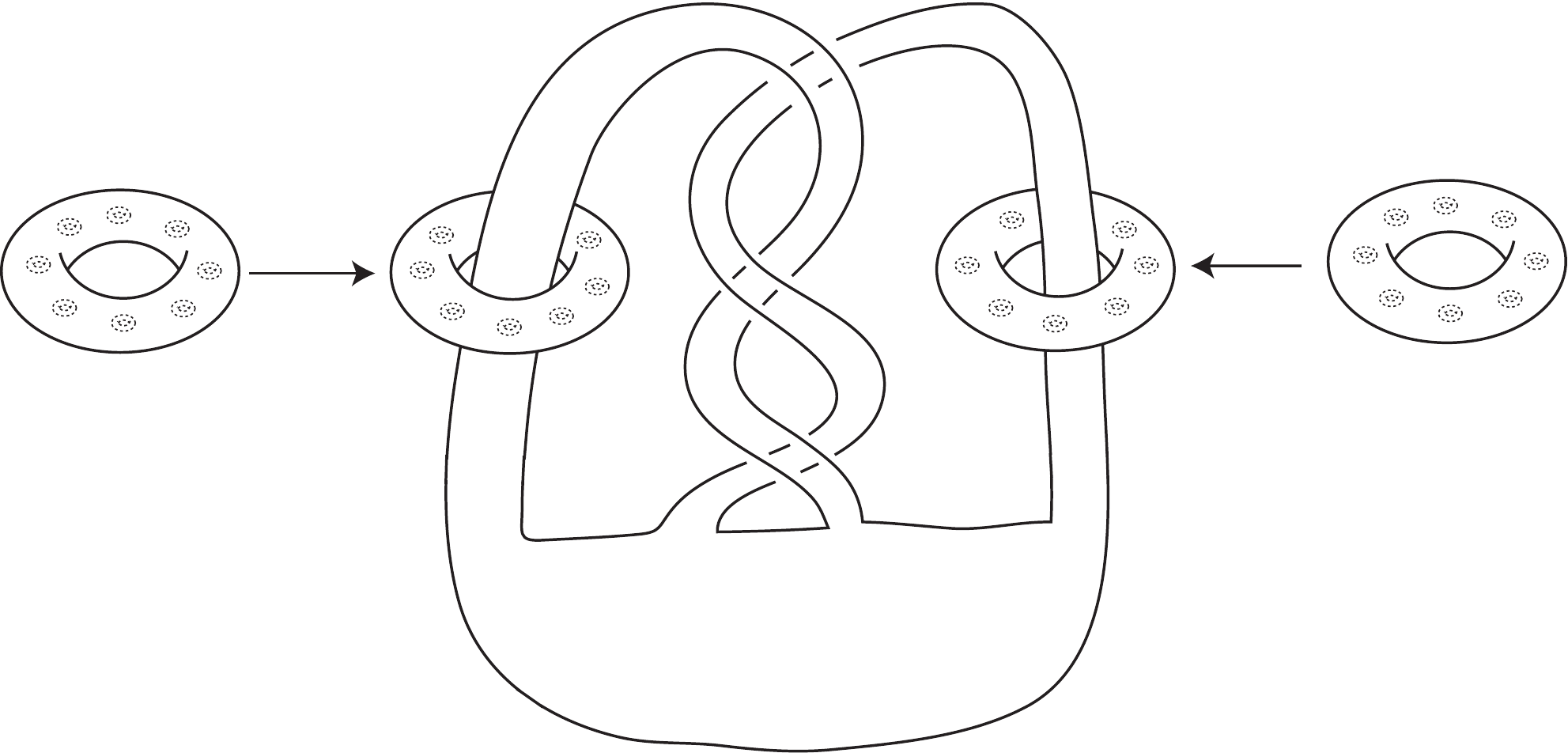}}
\put(77,136){$f^{i}_+$}
\put(323,137){$f^{i}_-$}
\put(6,93){$S^3-R_{i-1}$}
\put(361,95){$S^3-R_{i-1}$}
\end{picture}
\caption{$J_{n}$ as the result of $2^i$ infections on $R_i$}\label{fig:alternativeviews}
\end{figure}
If we alter our point of view by \emph{postponing} (ignoring for the moment) the infections, then we are precisely in the situation of Figure~\ref{fig:mapsf}, that is if we first replace the two fat solid tori by two copies of $S^3-R_{i-1}$ (by convention the maps are named $f_{\pm}^{i}: S^3-R_{i-1}\to S^3-R_{i}$), then we arrive, by definition, at $R_i$. The two collections of images in $S^3-R_i$ of the $2^{i-1}$ clones are precisely the $2^i$ clones $\{\eta^{i}_*\}\equiv ~\{f^{i}_{\pm}\circ\dots\circ f^1_{\pm}(\eta^0)\}$. If we \emph{then} perform these $2^i$ infections using the knot $J_{n-i}(K)$ as the infecting knot in each case, we arrive at the description claimed in the Proposition. This completes the inductive step.
\end{proof}

\begin{cor}\label{cor:infection} $J_n(K)$ may be obtained from the ribbon knot $R_n$ as the result of $2^n$ infections along clones, $\{f^{n}_{\pm}\circ\dots\circ f^1_{\pm}(\eta^0)\}$, that lie in $\pi_1(S^3\backslash R_n)^{(n)}$, using the knot $K$ as the infecting knot each time.
\end{cor}

\begin{proof}[Proof of Corollary~\ref{cor:infection}] Apply Proposition~\ref{prop:altdescriptions} in the case $i=n$.
\end{proof}

Returning to the proof of Theorem~\ref{thm:main}, let $C_n$, for $n\ge1$, be the Cheeger-Gromov constant for $M_{R_n}$.  Now suppose $K$ is chosen so that $|\rho_0(K)|>C_n$. We shall show that no non-zero multiple of $J_n=J_n(K)$ is rationally $(n.5)$-solvable. In particular this will demonstrate that $J_n$ is of infinite order in the smooth and topological concordance groups. 
In view of Corollary~\ref{cor:infection}, we can apply the following theorem of Cochran-Teichner to finish the proof of Theorem~\ref{thm:main}, once we verify that the hypotheses of Theorem~\ref{thm:infection} do indeed hold in the case that $R=R_n$ $J=J_n(K)$ and the collection $\{\eta_i\}$ is the collection of $m=2^n$ clones, $\{f^{n}_{\pm}\circ\dots\circ f^1_{\pm}(\eta^0)\}$ described above. Note that the first criterion on $\{\eta_i\}$ in the hypotheses of Theorem~\ref{thm:infection} is already guaranteed by Corollary~\ref{cor:infection}.

\begin{thm}[Theorem 4.2 ~\cite{CT}]\label{thm:infection} Let $R$ be a slice knot and $M$ the $0$-framed surgery on $R$. Let $\{\eta_1,\dots,\eta_m\}$ be an oriented link in $S^3\smallsetminus R$ that is a trivial link in $S^3$. Suppose that
the $\{\eta_i\}$ have the following two properties:
\begin{itemize}

\item $[\eta_i]\in\pi_1(M)^{(n)}, \quad 1\leq i\leq m$,
\item  For \emph{any}
$(n)$-solution $W$ of $M$ there exists {\bf some} $i$ such that
$j_*(\eta_i)\notin\pi_1(W)^{(n+1)}_r$ where $j_*:\pi_1(M)\to\pi_1(W)$.
\end{itemize}

Then for any Arf invariant zero knot $K$ $\{K_1,\dots,K_m\}$ for which
$|\rho_0(K)|>C_M$ (the Cheeger-Gromov constant of M) the knot
\[
J=R(\eta_1,...,\eta_m,K,...,K)
\]
formed from $R$ by infection on the $\{\eta_i\}$ is $(n)$-solvable but not rationally $(n.5)$-solvable. Moreover, $J$ is of infinite order in $\mathcal{F}_{n}/\mathcal{F}_{n.5}$. If the Arf invariant of $K$ is not zero then the result still holds except that $J$ may be only rationally $(n)$-solvable (and also $(n-1)$-solvable).
\end{thm}

The proof of Theorem~\ref{thm:main} has thus been reduced to the following theorem. This theorem replaces the difficult and somewhat mysterious results of Cochran-Teichner ~\cite[Section 6]{CT} and Cochran-Kim ~\cite[Section 6]{CK}

\begin{defn}\label{defn:ghosts} Let $\mu_i$ denote a meridian of $R_i$. For $n\geq i$, a ghost of $\mu_i$, denoted $(\mu_i)_*$ is an element of the set of $2^{n-i}$ circles $\{f^{n}_{\pm}\circ\dots\circ f^{i+1}_{\pm}(\mu_i)\}$. Thus, for any $i$, the ghosts of $\mu_i$ live in $S^3-R_n$ and $(\mu_i)_*\in \pi_1(S^3-R_n)^{(n-i)}$. These circles are precisely the meridians of the \textbf{copies} of $S^3-R_i$ that are embedded in $S^3-R_n$ by the maps $\{f^{n}_{\pm}\circ\dots\circ f^{i+1}_{\pm}\}$. Note that $\mu_0$ is the meridian of $R_0=U$ so $\mu_0=\eta^0$. Thus in particular, taking $i=0$, the ghosts of $\mu_0$ coincide with the clones $\{\eta^n_*\}$, that is $\{(\mu_0)_*\}=\{\eta^n_*\}$.
\end{defn}

An example is shown for $R_2$ ($n=2$) in Figure~\ref{fig:ghosts} where the $4$ ghosts of $\mu_0$ are shown. Notice that they coincide with the $4$ clones $\{\eta^2_*\}$ which were shown in Figure~\ref{fig:R2}. The $2$ ghosts of $\mu_1$ and the single ghost of $\mu_2$ are also shown in Figure~\ref{fig:ghosts}.

\begin{thm}\label{thm:iteratednontriviality} Let $R_n$ be the ribbon knot $J_n(U)$ as above and $0\leq k \leq n$. Suppose $W$ is an \emph{arbitrary} rational $k$-solution for $M_{R_n}$. Then at least one of the ghosts of $\mu_{n-k}$ maps non-trivially under the inclusion-induced map
$$
j_*:\pi_1(M_{R_n})\to \pi_1(W)/\pi_1(W)_r^{(k+1)}.
$$
In particular, taking $k=n$, at least one of the clones $\{\eta^n_*\}$ maps non-trivially under the inclusion-induced map
$$
j_*:\pi_1(M_{R_n})\to \pi_1(W)/\pi_1(W)_r^{(n+1)},
$$
as required in the second hypothesis of Theorem~\ref{thm:infection}.
\end{thm}

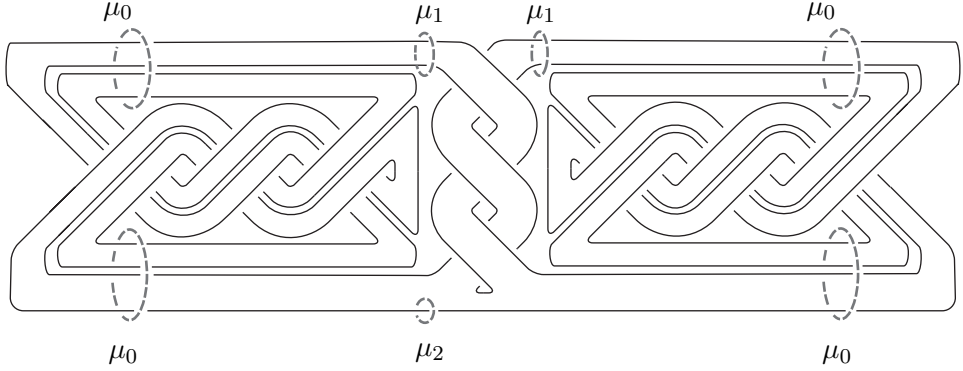
\begin{figure}[htbp]
\setlength{\unitlength}{1pt}
\begin{picture}(500,150)
\put(40,20){\includegraphics[width= 5 in]{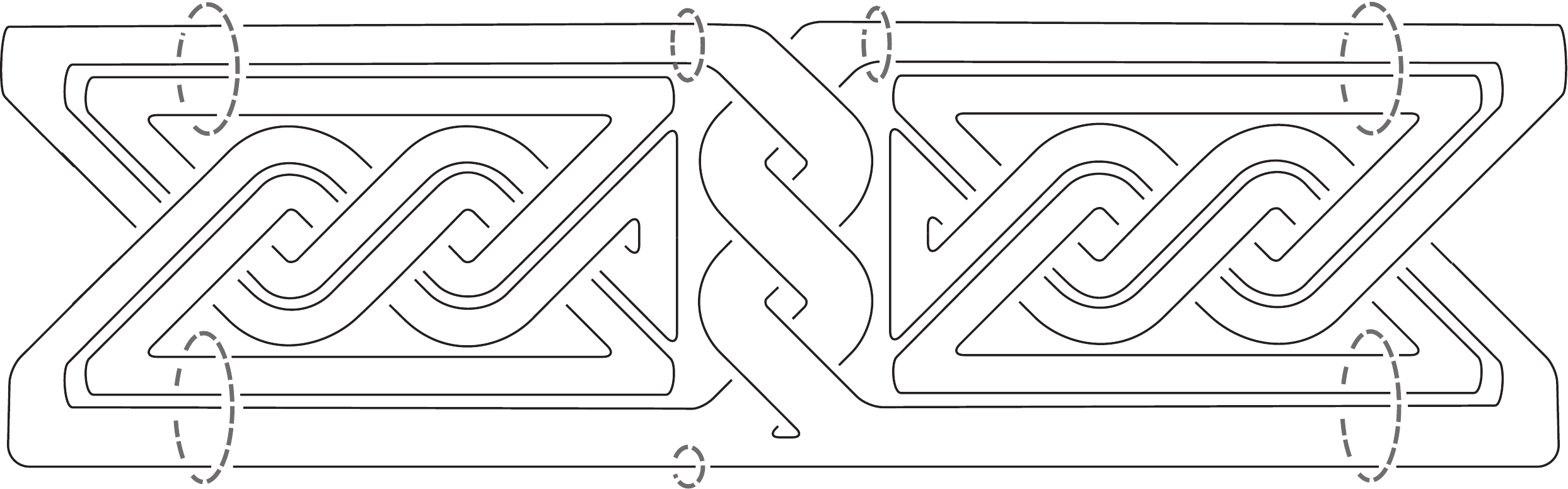}}
\put(79,7){$\mu_0$}
\put(77,137){$\mu_0$}
\put(349,7){$\mu_0$}
\put(343,137){$\mu_0$}
\put(195,8){$\mu_2$}
\put(195,136){$\mu_1$}
\put(237,136){$\mu_1$}
\end{picture}
\caption{The ghosts of $\mu_0$, $\mu_1$ and $\mu_2$ in $R_2$}\label{fig:ghosts}
\end{figure}

\begin{proof}[Proof of Theorem~\ref{thm:iteratednontriviality}] Here we view $n$ as fixed and proceed by induction on $k$. First suppose $k=0$. In this degenerate case $\mu_{n-k}=\mu_{n}$ is merely the meridian of $R_n$. Then there is only one ghost of $\mu_n$, namely $\mu_n$ itself. Clearly $\mu_n$ generates $H_1(M_{R_n};\mathbb{Q})$. Since $W$ is a rational $0$-solution
$$
j_*:~H_1(M_{R_n};\mathbb{Q})\to H_1(W;\mathbb{Q})
$$ 
is an isomorphism. Thus $j_*(\mu_n)\neq 0$ in $H_1(W;\mathbb{Q})$. But 
$$
\pi_1(W)/\pi_1(W)_r^{(1)}\cong (H_1(W;\mathbb{Z})/\text{Torsion})\hookrightarrow H_1(W;\mathbb{Q}),
$$
so $j_*(\mu_n)\neq 0$ in $\pi_1(W)/\pi_1(W)_r^{(1)}$. Thus the conclusion of Theorem~\ref{thm:iteratednontriviality} holds for $k=0$.

Now suppose that the Lemma is true for some $k-1$ where $1\leq k \leq n$. We will establish it for $k$. So consider a rational $(k)$-solution $W$ for $M_{R_n}$. Note that $W$ is \emph{a fortiori} a rational ($k-1$)-solution. Let $\Lambda=\pi_1(W)/\pi_1(W)^{(k)}_r$ and let $\psi:\pi_1(W)\to \Lambda$, and $\phi:\pi_1(M_{R_n})\to \Lambda$ be the induced coefficient systems. Then the inductive hypothesis applies to $W$ for the value $k-1$ and allows us to conclude that for at least one ghost of $\mu_{n-k+1}$,  we have ~$\phi(\mu_{n-k+1})_*)\neq 1$.  

We can then apply Proposition~\ref{prop:altdescriptions} with $K=U$ , the unknot, and $i=k-1$, to deduce that $J_n(U)$, i.e. $R_n$,
can be obtained from $R_{k-1}$ by infections along the clones $\{\eta^{k-1}_*\}\subset (S^3-R_{k-1})$ using the knot $R_{n-k+1}$ as infecting knot in each case, where
$$
\{\eta^{k-1}_*\}=~\{f^{k-1}_{\pm}\circ\dots\circ f^1_{\pm}(\eta^0)\}.
$$
In summary then, in the notation of Theorem~\ref{thm:nontriviality},
$$
R_n=R_{k-1}(\eta^{k-1}_i,(R_{n-k+1})_i)
$$
where $(R_{n-k+1})_i$ is a copy of $R_{n-k+1}$. Applying Theorem~\ref{thm:nontriviality} we see that the kernel, $P_i$, of the composition
$$
\mathcal{A}_0(R_{n-k+1})\to \mathcal{A}_0(R_{n-k+1}) \otimes\mathbb{Q}\Lambda\overset{\i_*}{\to} H_1(M_{R_n};\mathbb{Q}\Lambda)\overset{j_*}\to H_1(W;\mathbb{Q}\Lambda),
$$
satisfies $P_i\subset P_i^\perp$ for any clone $\eta^{k-1}_i$ such that $\phi(\eta^{k-1}_i)\neq 1$. We claim that there is at least one such clone $\eta^{k-1}_i$. For, by definition of infection, when we infect $R_{k-1}$ along $\eta^{k-1}_i$, the circle $\eta^{k-1}_i$  or more precisely, the longitude of such a circle, becomes identified to the meridian of that copy of the infecting knot $(R_{n-k+1})^i$. This meridian is not really a meridian of the abstract knot $R_{n-k+1}$, but rather an embedded copy of that meridian in $S^3-R^n$. In fact it is precisely one of the ghosts of $\mu_{n-k+1}$, $\{f^{n}_{\pm}\circ\dots\circ f^{i+1}_{\pm}(\mu_{n-k+1})\}$. By our inductive assumption, for at least one of these ghosts, ~$\phi((\mu_{n-k+1})_*)\neq 1$. Thus we have verified that there is at least one such clone such that $\phi(\eta^{k-1}_i)\neq 1$. We now restrict attention to such a value of $i$.

The two circles
$$
f^{n-k+1}_\pm(\mu_{n-k}) \in \pi_1(S^3-R_{n-k+1})^{(1)}
$$
as shown in the Figure~\ref{fig:twoghosts}, form a generating set for $\mathcal{A}_0(R_{n-k+1})$ (which is isomorphic to $\mathcal{A}_0(R_{1})$ and hence nontrivial). 
\begin{figure}[htbp]
\setlength{\unitlength}{1pt}
\begin{picture}(400,200)
\put(0,0){\includegraphics[height=200pt]{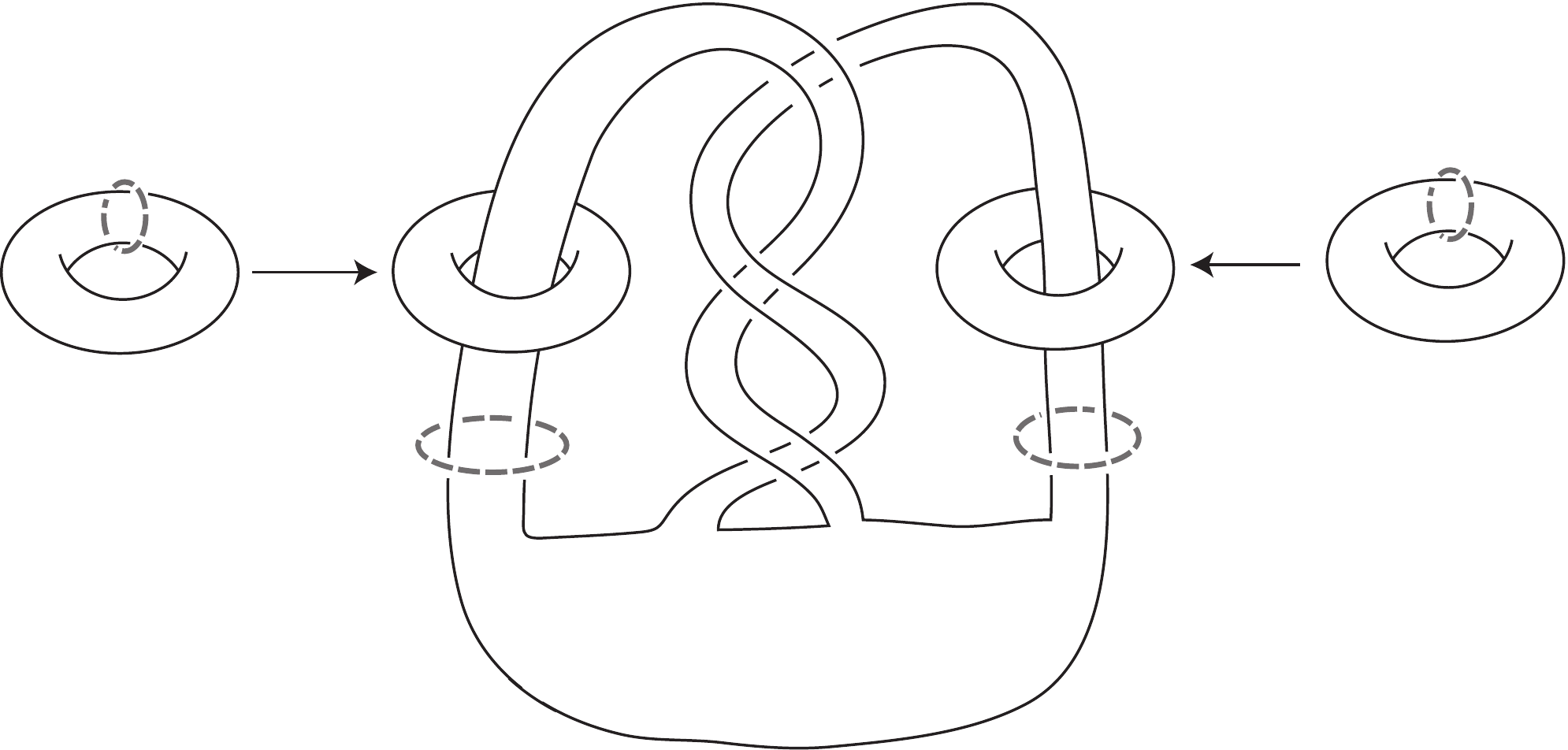}}
\put(74,136){$f^{n-k+1}_+$}
\put(321,137){$f^{n-k+1}_-$}
\put(6,93){$S^3-R_{n-k}$}
\put(361,95){$S^3-R_{n-k}$}
\put(52,69){$f^{n-k+1}_+(\mu_{n-k})$}
\put(306,71){$f^{n-k+1}_-(\mu_{n-k})$}
\put(18,160){$\mu_{n-k}$}
\put(370,160){$\mu_{n-k}$}
\end{picture}
\caption{$S^3-R_{n-k+1}$}\label{fig:twoghosts}
\end{figure}
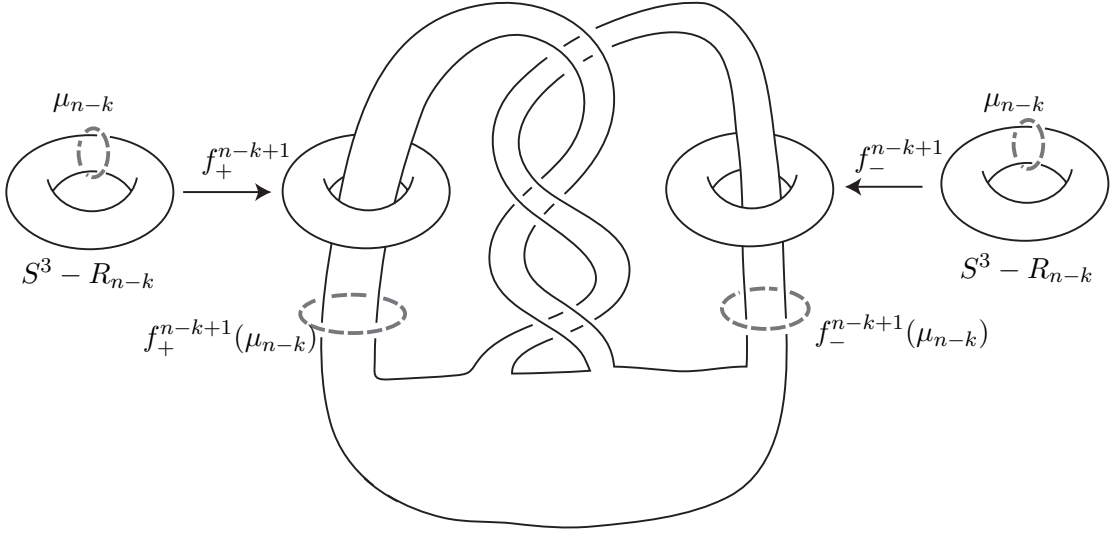
We can be assured that at least one of the generators is not in $P_i$ since otherwise
$$
P_i=\mathcal{A}_0(R_{n-k+1})\subset\mathcal{A}_0(R_{n-k+1})^\perp,
$$
contradicting the nonsingularity of the classical Blanchfield form of $\mathcal{A}_0(R_{n-k+1})$. Finally, consider the commutative diagram below, where we abbreviate $\pi_1(W)$ by $\pi$. 
\begin{equation*}
\begin{CD}
\pi_1(S^3-R_{n-k+1})^{(1)}      @>i_*>>    \pi_1(M_{R_n})^{(k)}  @>j_*>>   \pi^{(k)}_r  @>>>
\pi^{(k)}_r/\pi^{(k+1)}_r \\
  @VVV   @VVV        @VVV       @VVjV\\
\mathcal{A}_0(R_{n-k+1})     @>i_*>>  H_1(M_{R_n};\mathbb{Q}\Lambda)    @>j_*>> H_1(W;\mathbb{Q}\Lambda) @>\cong>>
  (\pi^{(k)}_r/[\pi^{(k)}_r,\pi^{(k)}_r])\otimes_{\mathbb{Z}} \mathbb{Q}\\
\end{CD}
\end{equation*}
Recall that $H_1(W;\mathbb{Q}\Lambda)$ is identifiable as the ordinary rational homology of the covering space of $W$ whose fundamental group is the kernel of $\psi:\pi\to \Lambda$. Since this kernel is precisely $\pi^{(k)}_r$, we have that
$$
H_1(W;\mathbb{Q}\Lambda)\cong (\pi^{(k)}_r/[\pi^{(k)}_r,\pi^{(k)}_r])\otimes_{\mathbb{Z}} \mathbb{Q}
$$
as indicated in the diagram. Note that, essentially by definition, the vertical map $j$ is injective. Therefore, since the composition in the bottom row sends one of the two homology classes $[f^{n-k+1}_\pm(\mu_{n-k})]$ to non-zero, the composition in the top row sends at least one of the two $f^{n-k+1}_\pm(\mu_{n-k})$ to a non-zero homotopy class.
Now observe that the inclusion-induced map $i_*$ in the top row above is induced by one of the compositions $f_\pm^{n}\circ\dots\circ f_\pm^{n-k+2}$. Thus
$$
f_i(f^{n-k+1}_\pm(\mu_{n-k}))=f_\pm^{n}\circ\dots\circ f_\pm^{n-k+2}\circ f^{n-k+1}_\pm(\mu_{n-k}).
$$
But for various values of $i$ these are precisely the ghosts $(\mu_{n-k})_*$. Consequently we have shown that at least one ghost such that
$$
j_*((\mu_{n-k})_*)\neq 1 ~\text{in} ~\pi^{(k)}_r/\pi^{(k+1)}_r
$$
as desired. 

This finishes the inductive proof of Theorem~\ref{thm:iteratednontriviality} and hence the proof of Theorem~\ref{thm:main}.
\end{proof}
\end{proof}

More generally, the proof above proves this more general result about iterated generalized doublings of knots.

\begin{thm}\label{thm:main3} If $R_j$, $1\leq j\leq n$, are slice knots and Arf($K$)$=0$, then the result, $R_n\circ\dots\circ R_1(K)$, of the n-times iterated generalized doubling lies in $\mathcal{F}_{n}$. If, additionally, , for each $j$, the submodule of the classical Alexander polynomial of $R_j$ generated by $\{\eta_{j1},\dots,\eta_{jm_j}\}$ contains elements $x,y$ such that $\mathcal{B}\ell_0^j(x,y)\neq 0$, where $\mathcal{B}\ell_0^j$ is the Blanchfield form of $R_j$, then there is a constant $C$, such that if the integral of the Levine signature function of $K$ is greater than $C$ in absolute value, then the resulting knot is of infinite order in the topologically concordance group (moreover no multiple lies in $\mathcal{F}_{n.5}$).
\end{thm}

A nice application of the more general theorem is the following which gives new information about the concordance order of knots that previously could not be distinguished from an order two knot.

\begin{cor}\label{cor:torsion} There is a constant $D$ such that if the absolute value of the integral of the Levine signature function of $J_0$ is greater than $D$ then the knot $E$ of Figure~\ref{fig:torsionfigeight} is of infinite order in the concordance group.
\end{cor}

\begin{proof}[Proof of Corollary~\ref{cor:torsion}] Any odd multiple of $E$ has Arf invariant one and hence is not a slice knot, nor even $(0)$-solvable. Let $J=\#^{2m}E=\#^m(E\# E)$. Since $E\# E$ is obtained from a connected-sum of two copies of the figure-eight knot (a slice knot $R$) by $4$ infections along a basis of the Alexander module of $R$, using the knot $J_1$ in each case, Theorem~\ref{thm:main3} applies ($n=1$). Arf($J_0)=0$ is not necessary for the second half of this theorem.
\end{proof}

\section{Iterated Bing doubles and higher-order $L^{(2)}$-signatures}\label{sec:Bingdoubles}

In this section we investigate higher-order signature invariants that obstruct any iterated Bing double of $K$ being a topologically slice link. We first state and prove the simplest results and later generalize.

Suppose $K$ is a knot in $S^3$, $G=\pi_1(M_K)$ and $\mathcal{A}_0=\mathcal{A}_0(K)$ is its classical rational Alexander module. Note that since the longitudes of $K$ lie in $\pi_1(S^3-K)^{(2)}$,
$$
\mathcal{A}_0\equiv G^{(1)}/G^{(2)}\otimes_{\mathbb{Z}[t,t^{-1}]}\mathbb{Q}[t,t^{-1}]
$$
Each submodule $P\subset \mathcal{A}_0$ corresponds to a unique metabelian quotient of $G$,
$$
\phi_P:G\to G/\tilde{P},
$$
by setting 
$$
\tilde{P}\equiv \{x~| x\in \text{kernel}(G^{(1)}\to G^{(1)}/G^{(2)}\to \mathcal{A}_0/P)\}.
$$
\noindent Note that $G^{(2)}\subset \tilde{P}$ so $G/\tilde{P}$ is metabelian. Therefore to any such submodule $P$ there corresponds a real number, the Cheeger-Gromov invariant, $\rho(M_K, \phi_P:G\to G/\tilde{P})$.

\begin{defn}\label{defn:highordersignatures} The first-order $L^{(2)}$-signatures of a knot $K$ are the real numbers
$\rho(M_K, \phi_P)$ where $P\subset \mathcal{A}_0(K)$ satisfies $P\subset P^\perp$. 
\end{defn}

The first-order signatures that correspond to metabolizers, that is submodules $P$ for which $P=P^\perp$, have been previously studied and are closely related to Casson-Gordon-Gilmer invariants ~\cite{Let}~\cite{Fr2}~\cite{Fr3}~\cite{Ki1}. Since $P=0$ always satisfies $P\subset P^\perp$, we give a special name to the signature corresponding to this case.

\begin{defn}\label{defn:rho1} $\rho^1(K)$ of a knot $K$ is the first-order $L^2$-signature given by the Cheeger-Gromov invariant $\rho(M_K, \phi:G\to G/G^{(2)})$.
\end{defn}

We remark that $\rho^1$ vanishes for a $(\pm)$-amphichiral knot by Proposition~\ref{prop:amphi} but it is not true that all the first-order signatures vanish for an amphichiral knot.

\begin{prop}\label{prop:amphi} If a $3$-manifold $M$ admits an orientation-reversing homeomorphism, then $\rho(M,\phi)=0$ for any $\phi$ whose kernel is a characteristic subgroup of $\pi_1(M)$.
\end{prop}
\begin{proof}[Proof of Proposition~\ref{prop:amphi}] Suppose $h:-M\to M$ is an orientation preserving homeomorphism. Then for any $\phi$,
$$
\rho(M,\phi)=\rho(-M, \phi\circ h_*)=-\rho(M,\phi\circ h_*).
$$
Since the $\rho$ invariant depends only on the kernel of $\phi$, which, being characteristic, is the same as the kernel of $\phi\circ h_*$, the last term equals $-\rho(M,\phi)$. Since the $\rho$ invariant is real-valued, it is zero.
\end{proof}

\begin{ex}\label{ex:first-ordersigs} A genus one algebraically slice knot has precisely $3$ first-order signatures, two corresponding to the two metabolizers , $P_1$, $P_2$ of the Seifert form and the third corresponding to $P_3=0$. Consider the knot $K$ in Figure~\ref{fig:examplehighersigs}. Since this knot is obtained from the ribbon knot $R_1$ by two infections on the band meridians $\eta_1, \eta_2$, we may apply Lemma~\ref{lem:additivity} to show
$$
\rho(M_K, \phi_P)=\rho(M_{R_1}, \phi_P)+\epsilon^1 _P\rho_0(K_1)+\epsilon^2 _P\rho_0(K_2)
$$
where $\epsilon^i_P$ is $0$ or $1$ according as $\phi_P(\eta_i)=0$ or not. Both $P_1$ and $P_2$ correspond to the kernels of actual ribbon disks for $R_1$ and so the maps $\phi_P$ on $M_{R_1}$ extend over ribbon disk exteriors in these cases. Consequently $\rho(M_{R_1}, \phi_P)=0$ for $P=P_1$ and $P=P_2$. Of course $\rho(M_R, \phi_{P_3})=\rho^1(R_1)$ by definition. Also $\epsilon^1 _{P_2}=0$ and $\epsilon^2 _{P_1}=0$. Therefore the first-order $L^2$-signatures of the knot $K$ are $\{\rho_0(K_1),\rho_0(K_2),\rho^1(R_1)+\rho_0(K_1)+\rho_0(K_2)\}$.

\begin{figure}[htbp]
\setlength{\unitlength}{1pt}
\begin{picture}(200,160)
\put(0,0){\includegraphics[height=150pt]{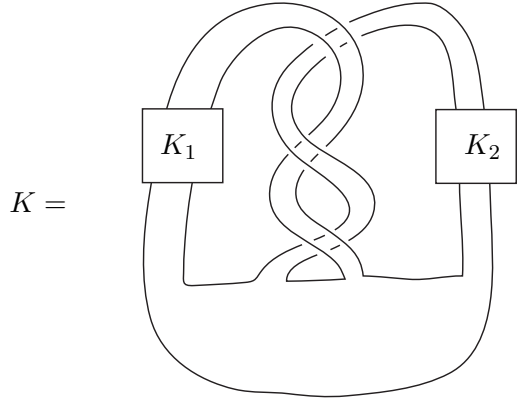}}
\put(-50,70){$K=$}
\put(7,92){$K_1$}
\put(122,92){$K_2$}
\end{picture}
\caption{A genus $1$ algebraically slice knot $K$}\label{fig:examplehighersigs}
\end{figure}

A genus one knot that is \textbf{not} zero in the rational algebraical concordance group (that is there is no metabolizer for the rational Blanchfield form) has precisely one first-order signature, namely $\rho^1(K)$ since any \textbf{proper} submodule $P$ of the rational Alexander module satisfying $P\subset P^\perp$ would have to be a (rational) metabolizer. The knot $K$ in Figure~\ref{fig:examplehighersigsfigeight} is of order two in the rational algebraic concordance group, but, using Lemma~\ref{lem:additivity}, we see that $\rho^1(K)=\rho^1(\text{figure eight})+2\rho_0(K_1)=2\rho_0(K_1)$.
\begin{figure}[htbp]
\setlength{\unitlength}{1pt}
\begin{picture}(200,160)
\put(0,0){\includegraphics[height=150pt]{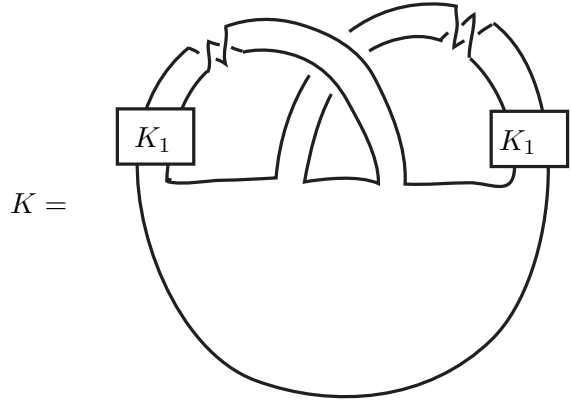}}
\put(-40,70){$K=$}
\put(7,95){$K_1$}
\put(145,94){$K_1$}
\end{picture}
\caption{Order 2 in algebraic concordance group}\label{fig:examplehighersigsfigeight}
\end{figure}
\end{ex}

Since the first-order signatures are very similar to Casson-Gordon invariants, the following is not surprising.

\begin{prop}\label{prop:firstordervanish} If $K$ is topologically slice in a rational homology $4$-ball (or more generally if $K$ is rationally $(1.5)$-solvable) then one of the first-order signatures of $K$ is zero.
\end{prop}
\begin{proof}[Proof of Proposition~\ref{prop:firstordervanish}] Let $V$ be a rational $(1.5)$-solution for $M_K$, $G=\pi_1(M_K)$, $\pi=\pi_1(V)$ and $\phi:\pi\to \pi/\pi^{(2)}_r$. By  ~\cite[Theorem 4.2]{COT} $\rho(M_K,\phi)=0$. Clearly $\phi\circ j*$ factors through $G/G^{(2)}$. Now, by Theorem~\ref{thm:nontriviality}, if $P$ denotes the kernel of the map
$$
\mathcal{A}_0(K)\overset{i_*}{\to} H_1(M_K;\mathbb{Q}[\pi/\pi^{(1)}_r])\overset{j_*}\to H_1(V;\mathbb{Q}[\pi/\pi^{(1)}_r]),
$$
then $P\subset P^\perp$ with respect to the classical Blanchfield form of $K$ (see also ~\cite[Theorem 4.4]{COT}). It follows that $\rho(M_K,\phi)$ is one of the first-order signatures of $K$. The details in verifying this final claim are carried out in more detail in the proof of the more general Theorem~\ref{thm:Bingdouble} below.
\end{proof}

The definition of the first-order signatures is not quite the same as that implicit in the work of Casson-Gordon-Gilmer and in more generality in ~\cite[Theorem 4.6]{COT}. One would hope that one need only consider those $P$ such that $P=P^\perp$. However this is false in the context of rational concordance. The knots in Figure~\ref{fig:examplehighersigsfigeight} are in general not slice in a rational homology ball, but this fact is \textbf{not} detected by signatures associated to metabolizers of the classical rational Blanchfield form. But this \emph{is} detected by $\rho^1$. Note that the figure eight knot is slice in a rational homology $4$-ball in such a way that the Alexander module of the figure-eight knot \textbf{injects} into $\pi/\pi^{(2)}_r$ where $\pi$ is the fundamental group of the complement of the slicing disk! 

We will now show that the first-order signatures of $K$, like the ordinary signatures, obstruct any \textbf{iterated Bing double} of $K$ from being a (topologically) slice link. This improves on Harvey's theorem which showed this same fact for the integral of the classical signatures ~\cite[Corollary 5.6]{Ha2}. There are several ways to define iterated Bing Doubling. In the most general way, one doubles one component at a time. However for simplicity, let us focus on the notion of Bing Doubling wherein we Bing double \textbf{each} component, then an $n$-fold iterated Bingdouble of $K$, $BD^{n}(K)$, is a $2^n$ component link. Note that once we show that none of these restricted Bing doubles is slice then it follows that none of the more general iterated Bing doubles is slice. 

\begin{thm}\label{thm:Bingdouble} Let $K$ be an arbitrary knot. If some $n$-fold iterated Bing double of $K$ ($n\geq 1$) is topologically slice in a rational homology $4$-ball (or is a rationally $(n+1.5)$-solvable link) then one of the first-order signatures of $K$ is zero.
\end{thm}
\begin{cor}\label{cor:Bingdouble}If $K$ is the algebraically slice knot of Figure~\ref{fig:examplehighersigs}, where $\rho_o(K_1)\neq 0$, $\rho_o(K_2)\neq 0$ and $\rho_o(K_1)+\rho_o(K_2)+\rho^1(R_1)\neq 0$ then no iterated Bing double of $K$ is topologically slice (nor even $(n+1.5)$-solvable). Therefore there is a constant $C$ ($=|\rho^1(R_1)|$) such that if $|\rho_0(K_1)|>C$ and Arf($K_1$)$=0$, then the $n$-fold iterated Bing double of $J_1(K_1)$ (the knot in Figure~\ref{fig:examplehighersigs} with $K_2=K_1$), is $(n+1)$-solvable but not slice nor even rationally $(n+1.5)$-solvable. If $K$ is the knot of Figure~\ref{fig:examplehighersigsfigeight} where $\rho_0(K_1)\neq 0$ then no iterated Bing double of $K$ is topologically slice (nor even rationally $(n+1.5)$-solvable). 
\end{cor}

\begin{proof}[Proof of Theorem~\ref{thm:Bingdouble}] Let $L=BD^{n}(K)$ for some $n\geq 1$ and $M=M_{L}$. Suppose $M$ is rationally $(n+1.5)$-solvable via $V$. We shall show that one of the first-order signatures of $K$ is zero.

Recall that $BD(K)$ can be obtained from the trivial link of two components by infection on the circle $\alpha$ shown dashed in Figure~\ref{fig:bingeta}, using $K$ as the infecting knot. This curve $\alpha$ can be expressed as $[x,y]$ in the fundamental group of the zero surgery on the trivial link where $x$ and $y$ are the meridians. If one now doubles each component of the this trivial link, then the image of the curve $\alpha$ becomes a curve that represents the double commutator $[[x,x'],[y,y']]$. Continuing in this manner, one sees that the iterated Bing double $L$ can be obtained from the trivial $2^n$ component link $T$ by a single infection, using the knot $K$, along a circle $\alpha$ representing, in $\pi_1(M_T)$, an element in $F^{(n)}$ but not in $F^{(n+1)}$. At this point we note that we need not assume that we are dealing with an iterated Bing double, but rather this previous sentence is all that we need assume. Thus our proof is really going to prove:

\begin{thm}\label{thm:basiclink} Suppose $T$ is a trivial link of $m$ components, $n\geq 1$ and $\alpha$ is an unknotted circle in $S^3-T$ that represents an element in $F^{(n)}-F^{(n+1)}$ where $F=\pi_1(S^3-T)$, and $L$ denotes $T(\alpha,K)$, the result of infection of $T$ along $\alpha$ using the knot $K$. If $L$ is topologically slice in a rational homology $4$-ball (or is even a rationally $(n+1.5)$-solvable link) then one of the first order signatures of $K$ is zero.
\end{thm}

Proceeding, since $L=T(\alpha,K)$, there exists a cobordism $E$ as in Figure~\ref{fig:mickey} whose boundary is $M_T\sqcup M_K\sqcup -M$. We form a null-bordism $W$ as follows. Cap off $M\subset \partial E$ using $V$. Thus $\partial W=M_K\cup M_T$. Let $\pi=\pi_1(W)$ and consider $\phi:\pi\to\pi/\pi^{(n+2)}_r$. In the case that $V$ is a slice disk exterior then we can apply Theorem~\ref{thm:linksliceobstr} to conclude that 
$$
\rho(M,\phi)=0.
$$
If $V$ is merely a rational $(n+1.5)$-solution, we would like to apply Theorem~\ref{thm:rho=0} to arrive at the same conclusion. But we must first verify that $L$ satisfies the conditions of Lemma~\ref{lem:rank}. This requires only that $\phi(\ell_K)=1$. This is certainly the case since, by property $(5)$ of Lemma~\ref{lem:mickeyfacts}, $\ell_K$ is identified with the reverse of meridian of $\alpha$ which bounds a disk in $M_T$, hence is null-homotopic in $W$. Let $\ov\phi$ be restriction of $\phi$ to $\pi_1(M_K)$ and $\phi_T$ denote the restriction of $\phi$ to $\pi_1(M_T)$. Thus, by Lemma~\ref{lem:additivity}
$$
\rho(M_K,\ov\phi)+ \rho(M_T,\phi_T)=0.
$$
Since $T$ is a trivial link, $M_T=\partial Y$ where $Y$ is a boundary connected-sum of copies of $S^1\times B^3$. Since $\pi_1(\partial Y)\cong \pi_1(Y)$, $\phi_T$ extends to $Y$. Hence by Theorem~\ref{thm:linksliceobstr},
$$
\rho(M_T,\phi_T)=0.
$$
Therefore 
$$
\rho(M_K,\ov\phi)=0.
$$

It remains only to identify $\rho(M_K,\ov\phi)$ as one of the first-order signatures of $K$. First note that the meridian of $K$ is isotopic in $E$ to the infection circle $\alpha$ in $M_T$. Since $\alpha\in \pi_1(S^3-T)^{(n)}$, this meridian represents an element of $\pi_1(E)^{(n)}$ and hence an element of $\pi^{(n)}$. Since $G\equiv\pi_1(M_K)$ is normally generated by this meridian, 
$$
i_*(G)\subset \pi^{(n)}
$$
and so 
$$
i_*(G^{(2)})\subset \pi^{(n+2)}.
$$
Consequently $\ov\phi$ factors through $G/G^{(2)}$ and the image of $\ov\phi$ is contained in $\pi^{(n)}/\pi^{(n+2)}_r$. By Property $2$ of Proposition~\ref{prop:rho invariants}, $\rho(M_K,\ov\phi)$ depends only on the image of $\ov\phi$. Thus
$$
\rho(M_K,\ov\phi)=\rho(M_K,G\to G/G^{(2)}\to G/\tilde{P})
$$
where $\tilde{P}=\text{ker}\ov\phi$. Therefore we need only characterize $\tilde{P}$. To this end, let $\tilde{\pi}=\pi_1(V)$. From property $(1)$ of Lemma~\ref{lem:mickeyfacts}
$$
\pi_1(M)\to \pi_1(E)
$$
is surjective with kernel the normal closure of the longitude $\ell_K$ of $K$ (here we are considering that $S^3-K\subset M$). Therefore the kernel of the map
$$
\tilde{\pi}\to \pi
$$
induced by the inclusion $V\hookrightarrow V\cup E$ is the normal closure of $\ell_K$. We claim that this induces an isomorphism
$$
\tilde{\pi}/\tilde{\pi}^{(n+2)}_r\cong \pi/\pi^{(n+2)}_r.
$$
This will follow if we show $\ell_K\in \tilde{\pi}^{(n+2)}$.  Recall that $\alpha\in \pi_1(S^3-T)^{(n)}$. It follows, as shown in ~\cite[Proof of Theorem 8.1]{C} that a stronger fact holds, namely that the longitudinal push-off of $\alpha$, $\ell_\alpha$, lies in $\pi_1(M)^{(n)}$. But $\ell_\alpha$ is identified to the meridian, $\mu_K$, of $S^3-K\subset M$. Since $\ell_K\in \pi_1(S^3-K)^{(2)}$ and $\pi_1(S^3-K)$ is normally generated by $\mu_K$, 
$$
\ell_K\in \pi_1(M)^{(n+2)}\subset \tilde{\pi}^{(n+2)},
$$
as required.
Hence
$$
\tilde{P}=\text{ker}(G\to \tilde{\pi}/\tilde{\pi}^{(n+2)}_r).
$$
Moreover, since the copy of $S^3-K$ that is a subset of $M_K$ and the copy of $S^3-K$ that is a subset of $M$ are isotopic in $E$, we are now free to think of $G$ as $\pi_1$ of the latter.

Now consider $\Lambda=\tilde{\pi}/\tilde{\pi}^{(n+1)}_r\cong \pi/\pi^{(n+1)}_r$ and $\psi:\tilde{\pi}\to \Lambda$. We seek to apply Theorem~\ref{thm:nontriviality} to $L=T(\alpha,K)$, $\alpha\in \pi_1(S^3-T)^{(n)}$, $k=n+1$ and the rational $(n+1)$-solution $V$ for $M$. To apply Theorem~\ref{thm:nontriviality}, we first need to verify that $\psi(\alpha)\neq 1$. 

Consider the inclusion $i:M_T\to W$. By property $(2)$ of Lemma~\ref{lem:mickeyfacts} and since $V$ is a rational $(n)$-solution, this map induces an isomorphism on $H_1(-;\mathbb{Q})$. By property $(3)$ of Lemma~\ref{lem:mickeyfacts}
$$
H_2(W;\mathbb{Q})\cong H_2(V;\mathbb{Q})\oplus i_*(H_2(M_K;\mathbb{Q}).
$$
Since $V$ is a rational $(n)$-solution, $H_2(V;\mathbb{Q})$ has a basis consisting of surfaces $\Sigma$ wherein $\pi_1(\Sigma)\subset \pi^{(n)}$. $H_2(M_K)$ is represented by a capped off Seifert surface $\overline{\Sigma}$ for $K$. Since $\pi_1(M_K)$ is normally generated by the meridian of $K$, which lies in $\pi^{(n)}$, $\pi_1(\overline{\Sigma})\subset \pi^{(n)}$. Thus, by ~\cite[Theorem 2.1]{CH2}, there is a monomorphism
$$
i_H:\pi_1(M_T)/\pi_1(M_T)^{(n+1)}_H\hookrightarrow \pi/\pi^{(n+1)}_H
$$
where the subscript $H$ denotes Harvey's torsion-free derived series ~\cite[Section 2]{Ha2}. Since the rational derived series is contained in the torsion-free derived series we have the commutative diagram
\begin{equation}\label{diag:harvey}
\begin{diagram}\dgARROWLENGTH=1em
\node{\pi_1(M_T)/\pi_1(M_T)^{(n+1)}_r} \arrow{e,t}{i_*}
\arrow{s,l}{\pi}\node{\pi/\pi^{(n+1)}_r} \arrow{s,l}{}\arrow{e,t}{\cong}\node{\Lambda}\\
\node{\pi_1(M_T)/\pi_1(M_T)^{(n+1)}_H}\arrow{e,t}{i_H} \node{\pi/\pi^{(n+1)}_H}
\end{diagram}
\end{equation}
From this diagram we see that if $\alpha\in \pi_1(M_T)$ mapped to zero in $\Lambda$ then 
$\pi(\alpha)=1$ meaning that $\alpha\in \pi_1(M_T)^{(n+1)}_H$. But this contradicts our hypothesis on $\alpha$ since , for the free group $\pi_1(M_T)$, the torsion-free derived series coincides with the derived series ~\cite[Proposition 2.3]{Ha2}. Hence $\psi(\alpha)\neq 1$ and therefore Theorem~\ref{thm:nontriviality} can be applied. Also note that since $\tilde{\pi}^{(n)}_r/\tilde{\pi}^{(n+1)}_r$ is $\mathbb{Z}$-torsion free, no power of $\alpha$ maps to zero. This implies that the kernel of $\ov\phi$ is contained in $G^{(1)}$ since $G/G^{(1)}$ is generated by $\alpha$.

Now, by Theorem~\ref{thm:nontriviality}, if $P$ denotes the kernel of the map
$$
\mathcal{A}_0(K)\overset{i_*}{\to} H_1(M;\mathbb{Q}\Lambda)\overset{j_*}\to H_1(V;\mathbb{Q}\Lambda).
$$
then $P\subset P^\perp$ with respect to the classical Blanchfield form of $K$. Examine the commutative diagram below where $P$ is the kernel of the bottom horizontal composition. To justify the isomorphism in the bottom row, recall that $H_1(V;\mathbb{Q}\Lambda)$ is identifiable as the ordinary rational homology of the covering space of $V$ whose fundamental group is the kernel of $\psi:\tilde{\pi}\to \Lambda$. Since this kernel is precisely $\tilde{\pi}^{(n+1)}_r$, we have that
$$
H_1(V;\mathbb{Q}\Lambda)\cong (\tilde{\pi}^{(n+1)}_r/[\tilde{\pi}^{(n+1)}_r,\tilde{\pi}^{(n+1)}_r])\otimes_{\mathbb{Z}} \mathbb{Q}
$$
as indicated in the diagram
\begin{equation*}
\begin{CD}
G^{(1)}      @>i_*>>    \pi_1(M)^{(n+1)}  @>j_*>>   \tilde{\pi}^{(n+1)}_r  @>>>
\tilde{\pi}^{(n+1)}_r/\tilde{\pi}^{(n+2)}_r \\
  @VV{\pi}V   @VVV        @VVV       @VVjV\\
\mathcal{A}_0(K)     @>i_*>>  H_1(M;\mathbb{Q}\Lambda)    @>j_*>> H_1(V;\mathbb{Q}\Lambda) @>\cong>>
  (\tilde{\pi}^{(n+1)}_r/[\tilde{\pi}^{(n+1)}_r,\tilde{\pi}^{(n+1)}_r])\otimes_{\mathbb{Z}} \mathbb{Q}\\
\end{CD}
\end{equation*}
Since, by definition, 
$$
\tilde{\pi}^{(n+2)}_r\equiv \text{kernel}(\tilde{\pi}^{(n+1)}_r\to (\tilde{\pi}^{(n+1)}_r/[\tilde{\pi}^{(n+1)}_r,\tilde{\pi}^{(n+1)}_r])\otimes_{\mathbb{Z}} \mathbb{Q}))
$$
the far-right vertical map $j$ is injective. Thus the kernel of the top horizontal composition is precisely $\pi^{-1}(P)$, which is precisely $\tilde{P}$. This identifies the image of the map $G\to \pi/\pi^{(n+2)}_r$ as $G/\tilde{P}$ for a submodule $P\subset \mathcal{A}_0(K)$ where $P\subset P^\perp$. Thus $\rho(M_K,\ov\phi)$ is a first-order signature.

\noindent This completes the proof of Theorem~\ref{thm:Bingdouble}.
\end{proof}

In examining the proof above, one sees that we made almost no use of the fact that $T$ was a trivial link. Indeed the proof really proves this more general result. The more general result says that if the first-order signatures are large then the infected link is not slice. This generalizes Harvey's ~\cite[Theorem 5.4]{Ha2} where it was shown under identical hypotheses that $\rho_0(K)$ obstructs $T(\alpha,K)$ from being slice.

\begin{thm}\label{thm:basiclink2} Suppose $T$ is a slice link of $m$ components, $n\geq 1$ and $\alpha$ is an unknotted circle in $S^3-T$ that represents an element in $\pi_1(S^3-T)^{(n)}$ that does not lie in $\pi_1(M_T)^{(n+1)}_H$. Let $L$ denotes $T(\alpha,K)$, the result of infection of $T$ along $\alpha$ using the knot $K$. If $L$ is topologically slice in a rational homology $4$-ball (or is even a rationally $(n+1.5)$-solvable link) then one of the first order signatures of $K$ is less in absolute value than the Cheeger-Gromov constant of $M_T$.
\end{thm}

\begin{proof}[Proof of Theorem~\ref{thm:basiclink2}] The proof is identical except that instead of $\rho(M_K,\ov\phi)=0$ we have only that
$$
|\rho(M_K,\ov\phi)|=|\rho(M_T,\phi_T)|<C_{M_T}.
$$
\end{proof}

Before moving on to more general results, we give another application.

\vspace{.3in}

In ~\cite[Section 6]{Ha2} Harvey considered a filtration $\mathcal{F}^m_{
(n)}$ of the $m$-component string link concordance group wherein a string link $L$ is $(n)$-solvable if its closure $\hat{L}$ is an $(n)$-solvable link in the sense of ~\cite[Section 8]{COT}. The restriction of this filtration to boundary string links, $\mathcal{B}(m)$ was denoted $\mathcal{BF}^m_{(n)}$. Harvey defined specific real-valued higher-order signature invariants, $\rho_n$ of string links. She showed that \emph{each} $\rho_i$ gave a \textbf{homomorphism} $\rho_i:\mathcal{B}(m)\to \mathbb{R}$. Moreover she showed that $\rho_n$ induces a homomorphism
$$
\rho_n:\mathcal{BF}^m_{(n)}/\mathcal{BF}^m_{(n+1)}\to \mathbb{R}
$$
whose image, for any $m>1$, contains an infinitely generated rational vector subspace of $\mathbb{R}$. This was slightly improved to $\mathcal{BF}^m_{(n)}/\mathcal{BF}^m_{(n.5)}$ in ~\cite[Theorem 4.5]{CH2}. From this she concluded that (we incorporate the improvement of ~\cite[Theorem 4.5]{CH2})

\begin{thm}\label{thm:shellythm} ~\cite[Theorem 6.8]{Ha2} For any $m>1$ the abelianization of $\mathcal{BF}^m_{(n)}/\mathcal{BF}^m_{(n.5)}$ has infinite rank, and so $\mathcal{BF}^m_{(n)}/\mathcal{BF}^m_{(n.5)}$ is an infinitely generated subgroup of $\mathcal{F}^m_{(n)}/\mathcal{F}^m_{(n.5)}$
\end{thm}

Our examples cannot be detected by any of Harvey's $\{\rho_i\}$ and so we can use them to show that

\begin{cor}\label{cor:kernelrhon} For any $m>1$, $n\geq 2$, the kernel of Harvey's
$$
\rho_n:\mathcal{BF}^m_{(n)}/\mathcal{BF}^m_{(n.5)}\to \mathbb{R}
$$
contains an infinitely generated subgroup.
\end{cor}

\begin{proof}[Proof of Corollary~\ref{cor:kernelrhon}] Let $\{K_i\}$ be an infinite set of Arf invariant zero knots such that $\{\rho_0(K_i)\}$ is $\mathbb{Q}$-linearly independent subset of $\mathbb{R}$ (the existence of such a set was established in ~\cite[Proposition 2.6]{COT2}). Let $R_1$ be the ribbon knot $9_{46}$. It is easy to see that by taking a subset if necessary, that  we can assume that $\{\rho_0(K_i), \rho^1(M_{R_1})\}$ is linearly independent. Let $J_i$ denote the knot of Figure~\ref{fig:examplehighersigs} with $K_1=K_2=K_i$. By ~\cite[Proposition 3.1]{COT2} $J_i$ is a $(1)$-solvable knot. Fix $m>1$ and let $T$ denote the trivial $m$-string link in $D^2\times I$. Fixing $n\geq 2$, choose a circle $\alpha\in F^{(n-1)}-F^{(n)}$, where $F$ is the group of the exterior of $T$, such that $\alpha$ bounds a disk in $D^2\times I$. Let $L_i$ denote $T(\alpha,J_i)$, the string link obtained by infecting $T$ along $\alpha$ using the knot $J_i$. The closure $\hat{L}_i$ is obtained from the trivial link  (which is $(n)$-solvable) by a $(1)$-solvable knot along a circle in $F^{(n-1)}$. Thus by Lemma~\ref{lem:nsolv}, $\hat{L}_i$ is $(n)$-solvable in the sense of ~\cite{COT} (one must check that $\pi_1$ of the $(1)$-solution for $J_i$ produced by ~\cite[Proposition]{COT2} is normally generated by the meridian). In any case we will show this more generally in the proof of our main theorem Theorem~\ref{thm:mainlink}. Consequently $L_i\in \mathcal{F}^m_{(n)}$. It is easily seen that $L_i$ is a boundary string link (see ~\cite[Section 2]{CO2}), so
$$
L_i\in \mathcal{BF}^m_{(n)}.
$$
It follows directly from Harvey's formula ~\cite[Theorem 5.4]{Ha2} that $\rho_n(L_i)=0$ (indeed all her $\rho_j$ vanish for these links). Consider the subgroup of $\mathcal{BF}^m_{(n)}$ generated by $\{L_i\}$. Suppose this were finitely generated. Then there is a subset $\{L_1,...,L_k\}$ that is a generating set. Consider $L_N$ for some $N>k$. Then the closure of the product $L=L_NL_{i_1}^{\epsilon_1}L_{i_2}^{\epsilon_2}...L_q^{\epsilon_q}$ is $(n.5)$-solvable for $i_j\in \{1,...,k\}$ and $\epsilon_j\in\{\pm 1\}$. A crucial point is now the observation that $\hat{L}$ can be obtained from the trivial link by multiple infections on curves $\alpha$ and $\alpha_i$, all lying in $F^{(n)}-F^{(n+1)}$, where the infection along $\alpha_N$ is done using $J_N$ and the other infections are done using copies of $J_1,...,J_k$ or their mirror images (if $\epsilon_j=-1$). The proof of Theorem~\ref{thm:basiclink} applies verbatim to this situation (although it was stated above for only one infection) because the crucial Theorem~\ref{thm:nontriviality} applies to the Alexander module of each infection knot separately. The conclusion is that some first first-order signature of $J_N$ is equal to some linear combination of first-order signatures of the knots $\{J_1,...,J_k\}$. We saw in Example~\ref{ex:first-ordersigs} that a first order signature for $J_i$ is an element of the set $\{\rho_0(K_i),\rho^1(R_1)+2\rho_0(K_i)\}$. It follows that $\rho_0(K_N)$ is a (possibly trivial) linear combination of $\{\rho_0(K_1),...,\rho_0(K_k),\rho^1(M_{R_1})\}$, contradicting our choice of $\{K_i\}$. Therefore the subgroup of $\mathcal{BF}^m_{(n)}$ generated by $\{L_i\}$ is infinitely generated. 
\end{proof}

The techniques of the proof of Theorem~\ref{thm:Bingdouble} can be generalized to include the iterated Bing doubles of more and more subtle knots, in particular knots whose classical signatures \emph{and} first-order signatures and Casson-Gordon invariants vanish. For specificity we consider the family of knots $J_n$ from Figure~\ref{fig:family}. If $n>1$ these have vanishing classical signatures, first-order signatures and Casson-Gordon invariants. Yet we find that higher-order signatures obstruct their iterated Bing doubles from being slice. For the family $J_n(K)$, these higher-order signatures can be calculated, ``up to a constant'', in terms of the classical signatures of $K$, so we formulate our results in those terms. For simplicity of exposition, we restrict $K$ to have Arf invariant zero.

\begin{thm}\label{thm:mainlink} Suppose $T$ is a trivial link of $m$ components, $k$ and $n$ are positive integers such that $1\leq k\leq n$ and $\alpha$ is an unknotted circle in $S^3-T$ that represents an element in $F^{(k)}-F^{(k+1)}$ where $F=\pi_1(S^3-T)$, $K$ is a knot with Arf($K)=0$, and  $L_n(K)$ denotes $T(\alpha,J_{n-k}(K))$, the result of infection of $T$ along $\alpha$ using the knot $J_{n-k}(K)$ shown in Figure~\ref{fig:linkfamily}. Then $L_n(K)$ is $n$-solvable. Moreover, there is a positive constant $C$ such that if $|\rho_0(K)|>C$, then $L_n(K)$ is not topologically slice in a rational homology ball (nor even rationally $(n+1)$-solvable). Moreover if $L_n(K)$ is expressed as the closure of the $m$-component string link $\mathcal{SL}$ then no non-zero multiple of $\mathcal{SL}$ has closure that is rationally $(n+1)$-solvable.
\end{thm}

\begin{figure}[htbp]
\setlength{\unitlength}{1pt}
\begin{picture}(111,150)
\put(-50,0){\includegraphics{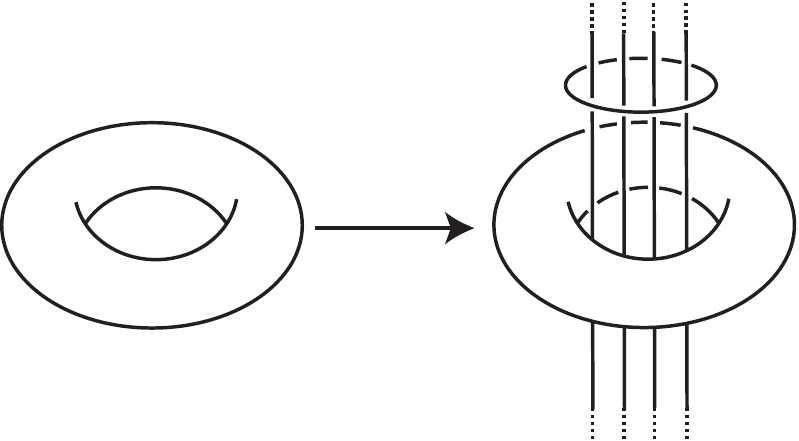}}
\put(98,100){$\tilde{\alpha}$}
\put(61,68){$g^{n-k}$}
\put(-30,19){$S^3-J_{n-k}(K)$}
\put(160,19){$T$}
\end{picture}
\caption{$L_n(K)$}\label{fig:linkfamily}
\end{figure}

\begin{rem} It is possible to show that $L_n(K)$ is not even rationally $(n.5)$-solvable, and also to choose $C$ independent of $n$, but these slight refinements require a more complicated proof.
\end{rem}

\begin{cor}\label{cor:BingdoubleJnK} For any $K$ with Arf$(K)=0$ and any $n$, there is a constant $C$ such that if the absolute value of the integral of the Levine-Tristram signatures of $K$ is greater than $C$ then the Bing double of $J_{n-1}(K)$ is $(n)$-solvable but not slice nor even rationally $(n+1)$-solvable.
\end{cor}

\begin{cor}\label{cor:iteratedbing}Suppose $k$ and $n$ are positive integers, and $K$ is a knot with Arf($K)=0$. The there is a constant $C$ such that if $|\rho_0(K)|>C$, then the $k$-fold iterated Bing double of $J_{n-k}(K))$ is $(n)$-solvable but not slice nor even rationally $(n+1)$-solvable.
\end{cor}

\begin{proof}[Proof of Corollary~\ref{cor:BingdoubleJnK}] As we have seen in Figure~\ref{fig:bingeta},  a Bing double is obtained by a single infection of the trivial link of two components along a circle $\alpha$ representing the generator of the non-zero group $F^{(1)}/F^{(2)}$ where $F$ is the free group on two letters. The result then follows directly from Theorem~\ref{thm:mainlink} with $k=1$.
\end{proof}

\begin{proof}[Proof of Corollary~\ref{cor:iteratedbing}] As discussed in the proof of Theorem~\ref{thm:Bingdouble}, the $k$-fold iterated Bing double can be obtained from the trivial $2^k$ component link $T$ by a single infection, using the knot $J_{n-k}(K))$, along a circle $\alpha$ representing, in $\pi_1(M_T)$, an element in $F^{(k)}-F^{(k+1)}$. The result then follows directly from Theorem~\ref{thm:mainlink}.
\end{proof}

\begin{proof}[Proof of Theorem~\ref{thm:mainlink}] The structure of the proof is similar to that of Theorem~\ref{thm:main}, but many of the needed results for links are not in the literature and have complexities not present in the case of knots. Without loss of generality we can assume that $L\equiv L_n(K)$ is the closure of a string link $\mathcal{SL}$ as in the last clause of the theorem. Since the closure of a multiple of $\mathcal{SL}$ is just a particular connected-sum of copies of $L$, we can (proceeding by contradiction) suppose that $\tilde L\equiv\#^M_{j=1}L$ were rationally $(n.5)$-solvable for some $M>0$.

We first establish that $L$ can be obtained from a ribbon link by multiple infections along curves lying in the $n^{th}$-derived subgroup of the ribbon group. It will follow immediately from Proposition~\ref{prop:operatorsfiltration} that $L$ is $(n)$-solvable. Specifically:

\begin{cor}\label{cor:linknsolvable} $L_n(K)$ can be obtained from the slice boundary link $L_n(U)=T(\alpha,R_{n-k})$ as the result of $2^{n-k}$ infections using the knot $K$ each time, along the clones $\alpha^{n-k}_*=\{g^{n-k}(\eta_*^{n-k})\}$ that lie in $\pi_1(S^3-L_n(U))^{(n)}$. Hence $L_n(K)$ is $n$-solvable.
\end{cor}

The proof of this will be accomplished by establishing that $L_n(K)$ has a variety of different descriptions due to its ``fractal'' nature. Recall $U$ denotes the trivial knot, and $J_0(K)\equiv K$. Suppose that we view the trivial link, $T$, the positive integer $k$ and the curve $\alpha\in F^{(k)}-F^{(k+1)}$ as fixed. Then $T(\alpha,~-)$ may be thought of as an operator from knots to $m$-component links. From this viewpoint, the proof of Proposition~\ref{prop:linkaltdescriptions} below is merely to apply this operator to the result of Proposition~\ref{prop:altdescriptions}. We give more details below.

\begin{prop}\label{prop:linkaltdescriptions}  For any knot $K$, and any $j,n$ such that $k\leq j\leq n$,  $L_n(K)$ can be obtained from $L_j(U)$ by multiple infections along the $2^{j-k}$ clones $\alpha^{j-k}_*=\{g^{j-k}(\eta_*^{j-k})\}$, using the knot $J_{n-j}(K)$ as the infecting knot in each case, and the clones lie in $\pi_1(S^3-L_j(U))^{(j)}$.
\end{prop}

Corollary~\ref{cor:linknsolvable} follows immediately.

\begin{proof}[Proof that Proposition~\ref{prop:linkaltdescriptions} implies Corollary~\ref{cor:linknsolvable}] Apply Proposition~\ref{prop:linkaltdescriptions} with $j=n$. We claim that $L_n(U)$ is a slice link since it is obtained from the slice link $T$ by infecting using the slice knot $R_{n-k}$ (this is an easy exercise for the reader). 

We claim that $L_n(U)$ is also a boundary link since we claim that infecting a boundary link $T$ by \emph{any} knot results in another boundary link. This is shown in ~\cite[p.403]{CO2}, but it also may be seen as follows. Since there is a degree one map (relative boundary) from any knot exterior to the exterior of the unknot, there is a degree one map from the exterior of the infected link $T(\alpha,K)$ to that of the boundary link $T$. Thus there is a map from the link group of $T(\alpha,K)$ to the free group sending meridians to generators, implying it is a boundary link.
\end{proof}

\begin{proof}[Proof of Proposition~\ref{prop:linkaltdescriptions}] By definition,
$$
L_n(K)\equiv T(\alpha,J_{n-k}(K)), ~L_j(U)\equiv T(\alpha,J_{j-k}(U)).
$$
Since $0\leq j-k\leq n-k$, we have from Proposition~\ref{prop:altdescriptions} that $J_{n-k}(K)$ can be obtained from $J_{j-k}(U)\cong R_{j-k}$ by multiple infections along the $2^{j-k}$ clones $\{\eta^{j-k}_*\}$, using the knot $J_{n-j}(K)$ as the infecting knot in each case. Moreover each clone $\eta^{j-k}_*$ lies in $\pi_1(S^3-R_{j-k})^{(j-k)}$. Therefore, postponing the infections as in Proposition~\ref{prop:altdescriptions}, and as suggested by Figure~\ref{fig:linkaltdescriptions}, we see that $L_n(K)\equiv T(\alpha,J_{n-k}(K))$ can be obtained from $L_j(U)\equiv T(\alpha,R_{j-k})$ by multiple infections along the clones $\{\alpha^{j-k}_*\}=\{g^{j-k}(\eta_*^{j-k})\}$, using the knot $J_{n-j}(K)$ as the infecting knot in each case. 
\begin{figure}[htbp]
\setlength{\unitlength}{1pt}
\begin{picture}(111,150)
\put(-50,0){\includegraphics{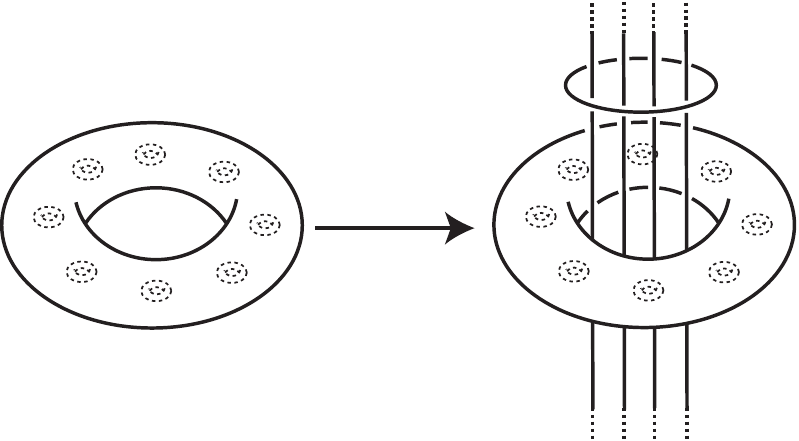}}
\put(98,100){$\tilde{\alpha}$}
\put(58,68){$g^{j-k}$}
\put(-72,62){$$}
\put(-72,52){}
\put(-105,62){$\{\eta^{j-k}_*\}\longrightarrow$}
\put(-25,19){$S^3-R_{j-k}$}
\put(160,19){$T$}
\end{picture}
\caption{$T(\alpha,J_{n-k}(K)$) obtained from $T(\alpha, R_{j-k})$}\label{fig:linkaltdescriptions}
\end{figure}
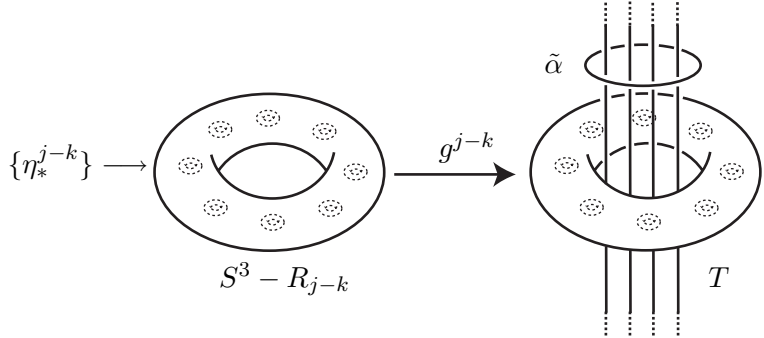

Since $\alpha\in \pi_1(S^3-T)^{(k)}$, the technical result ~\cite[proof of Theorem 8.1]{C} shows that the longitudinal push-off, $\alpha^+$, of $\alpha$ lies in $\pi_1(S^3-T(\alpha,R_{j-k})^{(k)}$. Hence, since the meridian of $R_{j-k}$ is identified with $\alpha^+$,
$$
g^{j-k}_*(\pi_1(S^3-R_{j-k}))\subset \pi_1(S^3-L_j(U))^{(k)}.
$$
Since, by Proposition~\ref{prop:altdescriptions}, each clone $\eta^{j-k}_*$ lies in $\pi_1(S^3-R_{j-k})^{(j-k)}$ , each clone $g^{j-k}(\eta^{j-k}_*)$ lies in $\pi_1(S^3-L_j(U))^{(j)}$. This completes the inductive step and the proof of Proposition~\ref{prop:linkaltdescriptions}.
\end{proof}

In summary, $L_n(K)$ can be obtained from the slice boundary link $L_n(U)=T(\alpha,R_{n-k})$, as the result of $2^{n-k}$ infections along circles that lie in $\pi_1(L_n(U))^{(n)}$ using the knot $K$ each time. 

Now we will prove the following general analog, for links, of ~\cite[Theorem 4.2]{CT} (for knots). We can apply this to our present situation with $R=T(\alpha,R_{n-k})$, $N=2^{n-k}$, $K_i=K$ for all $i$, $L=L_n(K)$. Observe that this will reduce the proof of Theorem~\ref{thm:mainlink} to proving that the hypotheses of Theorem~\ref{thm:linkinfection} are satisfied for $T(\alpha,R_{n-k})$ and the infection circles $\alpha^{n-k}_*=\{g^{n-k}f^{n-k}_*(\eta^0)\}$. This, in turn, will be accomplished by Theorem~\ref{thm:iteratedlink} below. Applying Theorem~\ref{thm:iteratedlink} shows that $T(\alpha,R_{n-k})$ satisfies the hypotheses of Theorem~\ref{thm:linkinfection} as desired. Thus the proof of 
Theorem~\ref{thm:mainlink} has been reduced to the proofs of the following two theorems.

\begin{thm}\label{thm:linkinfection} Let $R$ be a slice link of $m$ components
($n\ge1$) and $M_R$ the $0$-framed surgery on $R$. Suppose there
exists a collection of homotopy classes
\[
[\eta_i]\in\pi_1(M_R)^{(n)}, \quad 1\leq i\leq N,
\]
that has the following property: For \emph{any}
rational $(n)$-solution $W$ of $M_R$ there exists {\bf some} $i$ such that
$j_*(\eta_i)\notin\pi_1(W)^{(n+1)}_r$ where $j_*:\pi_1(M_R)\to\pi_1(W)$.

Then, for any oriented trivial link $\{\eta_1,\dots,\eta_m\}$ in $S^3\smallsetminus R$
that represents the $[\eta_i]$, and for any $N$-tuple
$\{K_1,\dots,K_N\}$ of Arf invariant zero knots for which
$\rho_0(K_i)>C_{M_R}$ (the Cheeger-Gromov constant of $M_R$), the link
\[
L=R(\eta_1,...,\eta_N,K_1,...,K_N)
\]
is $(n)$-solvable but no positive multiple of it is slice (nor even rationally $(n+1)$-solvable). (If the Arf invariant condition is dropped then $L$ is merely rationally $n$-solvable).
\end{thm}

\begin{thm}\label{thm:iteratedlink} Let $T_{n-k}\equiv T(\alpha,R_{n-k})$ be as above. Suppose $W$ is an \emph{arbitrary} rational $(n)$-solution for $M_{T_{n-k}}$. Then at least one of the $2^{n-k}$ clones $\alpha^{n-k}_*=\{g^{n-k}(\eta^{n-k}_*)\}$ maps non-trivially under the inclusion-induced map
$$
j_*:\pi_1(M_{T_{n-k}})\to \pi_1(W)/\pi_1(W)_r^{(n+1)}.
$$
\end{thm}

\begin{proof}[Proof of Theorem~\ref{thm:linkinfection}] Supposing that such $R$ and $\eta_i$ exist, let $L=R(\eta_1,...,\eta_N,K_1,...,K_N)$ for knots $K_i$ such that, for each $i$, Arf($K_i)=0$ and $\rho_0(K_i)>C_{M_R}$ (the Cheeger-Gromov constant of $M_R$). 

Since $L$ is the result of infections on an $n$-solvable) link along circles lying in the $n^{th}-$ derived subgroup $L$ is $n$-solvable (merely rationally $n$-solvable without the Arf invariant condition) by Proposition~\ref{prop:operatorsfiltration}.

Now we proceed by contradiction. Suppose that $\tilde L\equiv\#^p_{j=1}L$ were rationally $(n+1)$-solvable for some $p>0$. Then there would exist a rational $(n+1)$-solution $V$ with $\partial V= M_{\tilde L}$, the zero framed surgery on $\tilde L$. Using this we construct a particular rational $(n)$-solution $W$ for $M_{R}$ as follows (shown schematically in Figure~\ref{fig:cobordismlink}).
\begin{figure}[htbp]
\setlength{\unitlength}{1pt}
\begin{picture}(500,200)
\put(200, 57){$C$}
\put(71, 133){$Z^1_1$}
\put(107, 133){$Z^1_N$}
\put(71, 99){$E^1$}
\put(205, 99){$E^2$}
\put(337, 99){$E^p$}
\put(207, 133){$Z^2_1$}
\put(242, 133){$Z^2_N$}
\put(8,79){$M_{L}^1$}
\put(398,79){$M_{L}^p$}
\put(9,118){$M_{R}$}
\put(145,118){$M_{R}^2$}
\put(275,118){$M_{R}^p$}
\put(170, 130){$Y^2$}
\put(301, 130){$Y^p$}
\put(337, 133){$Z^p_1$}
\put(373, 133){$Z^p_N$}
\put(200,19){$V$}
\put(8,36){$M_{\tilde{L}}$}
\put(28,10){\includegraphics{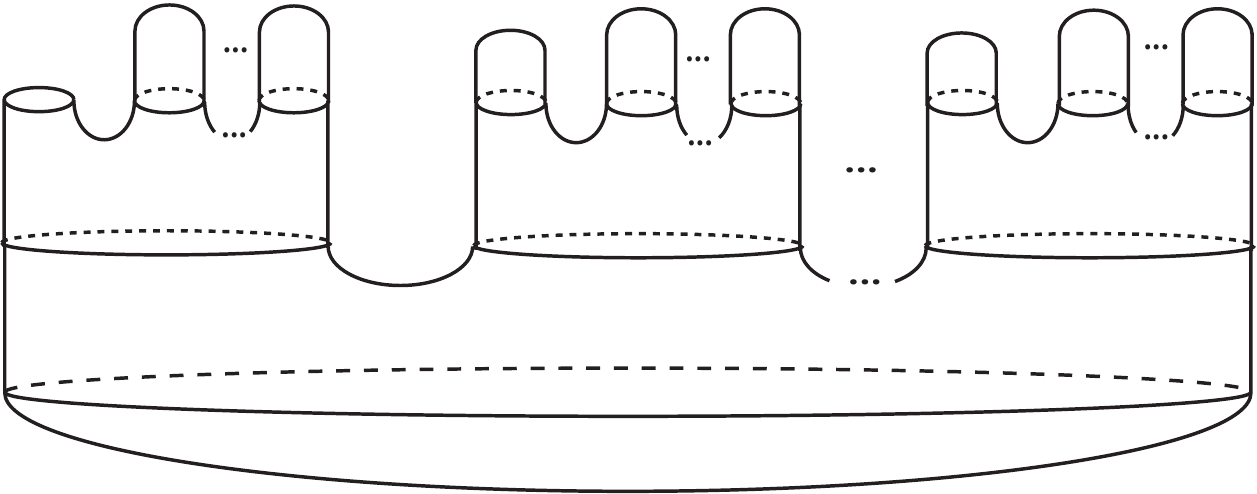}}
\end{picture}
\caption{The rational $n$-solution $W$ for $M_{R}$} \label{fig:cobordismlink}
\end{figure}
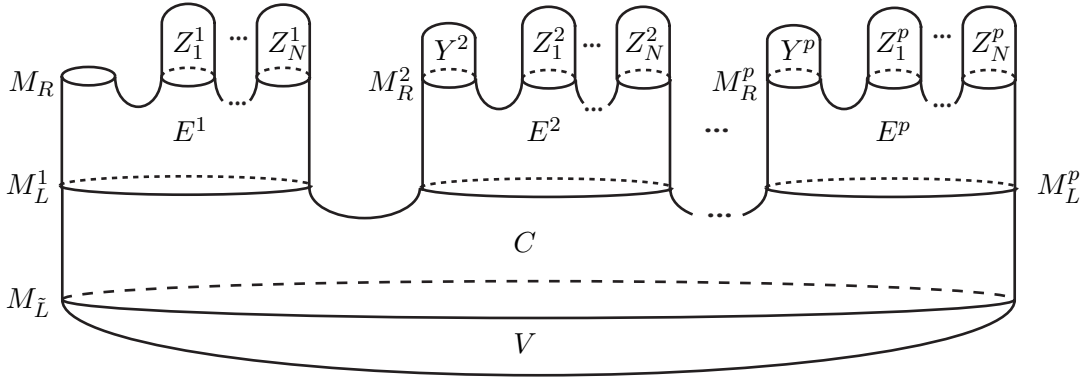
Here $C$ is the standard cobordism from $M_{\tilde L}$ to the disjoint union of $p$ copies of $M_{L}$. This cobordism is discussed in detail in ~\cite[Section 4]{COT2}. Cap off the boundary component $M_{\tilde L}$ using the rational $(n+1)$-solution $V$. Since $L$ is obtained from the link $R$ by infection on circles $\eta_i$ using the knots $K_i$, there is a cobordism $E$, as shown in Figure~\ref{fig:mickey}, such that
$$
\partial E= -M_L\sqcup M_R \sqcup_{i=1}^{N} M_{i}
$$
where we abbreviate $M_{K_i}$ by $M_i$. Add a copy of $E$ to each of the $p$ copies of $M_{L}$. We denote these copies by $E^j, 1\leq j\leq p$. Now, for each $i$, cap off each of the $p$ copies of $M_{i}$ with a $(0)$-solution $Z^j_i$ for $K_i$ (we can assume that $\pi_1(Z^j_i)=\mathbb{Z}$ by ~\cite[p.108]{COT2}~\cite[Appendix 5]{COT2}) and cap off each of the copies of $M_R$, except the ``first'', with a copy, $Y^j, 2\leq j\leq p$, of the exterior $Y$ of a set of slicing disks for the slice link $R$. The resulting manifold $W$ then has a single copy of $M_{R}$ as its boundary.

\begin{lem}\label{lem:H2} $W$ is a rational $n$-solution for $M_R$.
\end{lem}

\begin{proof}[Proof of Lemma~\ref{lem:H2}] By Definition~\ref{defn:rationalnsolvable}, we must show that
\begin{itemize}
\item $H_1(M_R;\mathbb{Q})\to H_1(W;\mathbb{Q})$ is an isomorphism, and
\item  $W$ admits a rational $(n)$-Lagrangian with rational $(n)$-duals.
\end{itemize}

First we claim that:
$$
H_2(W;\mathbb{Q})\cong H_2(V;\mathbb{Q})\oplus_{i,j} H_2(Z_i^j;\mathbb{Q}).
$$
Since $V$ is a rational $(n+1)$-solution for $M_{\tilde L}$, the inclusion-induced map
$$
j_*:~H_1(M_{\tilde L};\mathbb{Q})\to H_1(V;\mathbb{Q})
$$
is an isomorphism. It follows from duality that 
$$
j_*:~H_2(M_{\tilde L};\mathbb{Q})\to H_2(V;\mathbb{Q})
$$
is the zero map. Therefore if we examine the Mayer-Vietoris sequence with $\mathbb{Q}$-coefficients,
$$
H_2(C)\oplus H_2(V)\overset{\pi_*}{\lra} H_2(C\cup V)\to H_1(M_{\tilde L})\overset{(i_*,j_*)}{\longrightarrow}H_1(C)\oplus H_1(V),
$$
we see that $\pi_*$ induces an isomorphism
$$
(H_2(C)/(i_*(H_2(M_{\tilde L})))\oplus H_2(V)\cong H_2(V\cup C).
$$
Moreover recall that $C$ is obtained from a collar of the disjoint union of $p$ copies of $M_{L}$ by adding $p-1$ $1$-handles (to connect the components) and then adding $m(p-1)$ $2$-handles that have the effect of equating pairwise the meridional elements of the copies $L$. In this way we see that, for any of the boundary components $M_L$, $H_1(M_L;\mathbb{Q})\cong H_1(C;\mathbb{Q})\cong \mathbb{Q}^m$ generated by a set of meridians, and that $H_2(C;\mathbb{Q})\cong\oplus_{j=1}^p H_2(M_L;\mathbb{Q})$ (this is analyzed in more detail in ~\cite[p. 113-114]{COT2}). It is easy to see that a basis of $i_*(H_2(M_{\tilde L}))$ is formed from the sum, $1\leq j\leq p$ of the elements of natural bases for each $H_2(M_L;\mathbb{Q})$. Thus
$$
H_2(V\cup C;\mathbb{Q})\cong   H_2(V;\mathbb{Q})\oplus(\oplus_{j=1}^p H_2(M_L;\mathbb{Q}))/D
$$
where $D\cong \mathbb{Q}^m$ is the diagonal subgroup. Now, recall that we have analyzed the homology of $E$ in Lemma~\ref{lem:mickeyfacts}and found that,
$$
H_1(M_{L})\overset{i_*}{\longrightarrow} H_1(E)
$$
is an isomorphism. Therefore the following Mayer-Vietoris sequence with $\mathbb{Q}$-coefficients is exact,
$$
\oplus^p_{j=1} H_2(M_{L}^j)\to \oplus^p_{j=1} H_2(E^j)\oplus H_2(V\cup C)\overset{\pi_*}{\to} H_2(V\cup C\sqcup^p_{j=1} E^j)\to 0.
$$
Moreover, from property $(3)$ of Lemma~\ref{lem:mickeyfacts}, 
$$
H_2(E)\cong \oplus^N_{i=1}H_2(M_i)\oplus H_2(M_R)
$$
where the latter $H_2(M_R)\cong H_2(M_L)$ in $H_2(E)$. Combining these facts we have that
$$
H_2(V\cup C\sqcup^p_{j=1} E^j)\cong H_2(V)\oplus_{j=1}^p\oplus_{i=1}^N H_2(M_i^j)\oplus_{j=1}^p(H_2(M_R^j)/D).
$$
The next step in the formation of $W$ was the addition of the slice exteriors $Y^j$ to the copies $M_R^j$ for $2\leq j\leq p$. Since $H_1(\partial Y^j)\to H_1(Y^j)$ is an isomorphism and $H_2(Y^j)=0$, the effect on $H_2$ of adding the $Y^j$ is merely to kill all the $H_2$ carried by the boundaries $H_2(M^j_R)$, $2\leq j\leq p$. Taking into account the diagonal relation, we have
$$
H_2(V\cup C\cup E^j\cup Z_i^j)\cong H_2(V)\oplus_{j=1}^p\oplus_{i=1}^N H_2(M_i^j).
$$
The final step in the formation of $W$ was the addition of the $(0)$-solutions $Z^j_i$ to all the copies $M_i^j$ of $M_{K_i}$. Since, $Z^j_i$ is a $(0)$-solution, $H_1(M_i^j)\to H_1(Z_i^j)$ is an isomorphism and by duality $H_2(M_i^j)\to H_2(Z_i^j)$ is the zero map. Thus the effect on $H_2$ of adding the $Z_i^j$ is merely to kill all the generators of the $H_2(M_i^j)$ summand and add $H_2(Z_i^j)$. Thus we have
$$
H_2(W;\mathbb{Q})\cong H_2(V;\mathbb{Q})\oplus_{i,j}H_2(Z_i^j)
$$
This establishes the claim.

Combining some of the observations above it also follows that $H_1(M_R;\mathbb{Q})\to H_1(W;\mathbb{Q})$ is an isomorphism.

We return now to the proof that $W$ is a rational $n$-solution for $M_R$. Since $V$ is a rational $(n+1)$-solution, it is a rational $(n)$-solution. Let $\{\ell_1,\dots,\ell_g\}$ be a collection of $n$-surfaces generating a rational $n$-Lagrangian for $V$ and $\{d_1,\dots,d_g\}$ be a collection of $(n)$-surfaces generating its rational $(n)$-duals. By definition, $2g=\text{rank}_\mathbb{Q}H_2(V;\mathbb{Q})$. Similarly, for each $i$ and $j$ take a collection of such $(0)$-surfaces $\{l^{ij}_1,..,l^{ij}_k\}$, $\{d^{ij}_1,..,d^{ij}_k\}$ for the $(0)$-solutions $Z_i^j$. Now taking these surfaces for $V$ together with the collections of surfaces for the $Z_i^j$, these collections have the required \textbf{cardinality} (by the first part of the Lemma) to generate a rational $n$-Lagrangian with rational $(n)$-duals for $W$. Since $\pi_1(V)^{(n)}$ maps into $\pi_1(W)^{(n)}$, the $(n)$-surfaces for $V$ are also $n$-surfaces for $W$. We need to show that the $(0)$-surfaces for $Z_i^j$ are $(n)$-surfaces for $W$. 

The group $\pi_1(Z_j^i)\cong \mathbb{Z}$ is generated by the meridian of the knot $K^j_i$ in $M^j_i$.  This meridian is isotopic in $E_j$ to the infection curve $\eta^j_i\in M_R^j$. By hypothesis, 
$$
[\eta^j_i]\in \pi_1(M_R^j)^{(n)}.
$$
Therefore
$$
j_*(\pi_1(Z^i_j))\subset \pi_1(W)^{(n)}.
$$
Hence \emph{any} surface in $Z_i^j$ is an $(n)$-surface for $W$. Moreover, by functoriality of the intersection form with twisted coefficients these collections of surfaces have the required intersection properties to generate a rational $n$-Lagrangian with rational $(n)$-duals for $W$. Hence $W$ is a rational $(n)$-bordism for $M_R$, as was claimed. 

This completes the proof of Lemma~\ref{lem:H2}.
\end{proof}
 
We continue with the proof of Theorem~\ref{thm:linkinfection}. Now set $\Gamma= \pi_1(W)/\pi_1(W)^{(n+1)}_r$. Let $\psi:\pi_1(W)\to\Gamma$ be canonical surjection. Let $\phi:\pi_1(M_{R})\to\Gamma$ be the composition $\psi\circ j_*$. Thus by the hypothesis of Theorem~\ref{thm:linkinfection} there exists {\bf some} $i$ such that
$\phi(\eta_i)\neq 1$. We shall now compute $|\rho(M_{R},\phi)|$ using $W$, and find it to be greater than $C_R$. This contradiction will show that in fact $\tilde L\equiv\#^p_{j=1}L$ is not rationally $(n+1)$-solvable.

By definition we have,
$$
\rho(M_{R},\phi)= \sigma^{(2)}_\Gamma(W,\psi)-\sigma(W).
$$
By the additivity of the non Neumann and the ordinary signatures (~\cite[Lemma 5.9]{COT}) the latter signatures are the sums of the corresponding signatures for the submanifolds $V$, $C$, $E^j$, $Y^j$ and $Z^j_i$.

Since $V$ is a rational $(n+1)$-solution and $\Gamma^{(n+1)}=1$, by Theorem~\ref{thm:rho=0}
$$
\sigma^{(2)}_\Gamma(V)-\sigma(V)=0.
$$
Similarly, since $H_1(\partial Y^j)\to H_1(Y^j)$ is an isomorphism and $H_2(Y^j)=0$, $Y^j$ is a rational $(n+1)$-solution for any $n$. Hence
$$
\sigma^{(2)}_\Gamma(Y^j)-\sigma(Y^j)=0.
$$
By Lemma~\ref{lem:mickeysig},
$$
\sigma^{(2)}_\Gamma(E^j)-\sigma(E^j)=0.
$$
Now consider the cobordism $C$. There are several results in the literature concerning the vanishing of the signatures of $C$. None of those results can be directly applied because of different hypotheses. 

\begin{lem}\label{lem:sigC=0} For any PTFA coefficient system $\psi:\pi_1(C)\to \G$ 
$$
\sigma^{(2)}_\Gamma(C)=\sigma(C)=0.
$$
\end{lem}
\begin{proof}[Proof of Lemma~\ref{lem:sigC=0}] We have observed above that
$$
H_2(C;\mathbb{Q})/i_*(H_2(\partial C;\mathbb{Q}))=0.
$$
It follows immediately that $\sigma(C)=0$. By ~\cite[Proposition 2.7]{CH2} it follows that
$$
H_2(C;\mathcal{K}\G)/i_*(H_2(\partial C;\mathcal{K}\G))=0.
$$
By Property $1$ of Proposition~\ref{prop:rho invariants},
$$
\sigma^{(2)}_\Gamma(C)=0.
$$
\end{proof}

\noindent This leaves only the $Z^j_i$. Let $\psi^j_i$ denote the restriction of $\psi$ to $\pi_1(Z_i^j)$. Then, by definition
$$
\sigma^{(2)}_\Gamma(Z_i^j)-\sigma(Z_i^j)=\rho(M_i^j,\psi_i^j).
$$
However, since $\pi_1(Z^j_i)\cong \mathbb{Z}$, $\psi^j_i$  factors through $\mathbb{Z}$. Hence by Properties $2$,$3$ and $4$ of Proposition~\ref{prop:rho invariants}
$$
\rho(M_i^j,\psi_i^j)=\rho_0(K_i)
$$
if $\psi^j_i(\eta^j_i)\neq 1$ and is zero if  $\psi^j_i(\eta^j_i)=1$. Note that here we have used the fact that the infection circle $\eta^j_i$ (in $M_R^j$) is isotopic (in $E_j$) to the meridian of $K_i$ in $M^j_i$ (see property $(4)$ of Lemma~\ref{lem:mickeyfacts}).

Putting all of these together we have
$$
\rho(M_{R},\phi)=\sum_{i=1}^N d_i\rho_0(K_i)
$$
where $d_i$ is the number of values of $j$ for which $\psi(\eta^j_{i})\neq 1$. Since our hypothesis is that for each $i$
$$
\rho_0(K_i)> C_{M_R},
$$
this is a contradiction unless each $d_i=0$. However, we claim some $d_i>0$. For recall by Lemma~\ref{lem:H2}, $W$ is a rational $(n)$-solution for $M_R$. Thus by hypothesis there exists {\bf some} $i$ such that
$j_*(\eta^1_i)\notin\pi_1(W)^{(n+1)}_r$ where $j_*:\pi_1(M_R)\to\pi_1(W)$. Hence for some $i$, 
$$
\psi^j_i(\eta^1_i)\neq 1.
$$
This is a contradiction, completing the proof of Theorem~\ref{thm:linkinfection}.
\end{proof}

Thus the proof of Theorem~\ref{thm:main} has been reduced to the proof of Theorem~\ref{thm:iteratedlink}.

\begin{proof}[Proof of Theorem~\ref{thm:iteratedlink}]

\begin{defn}\label{defn:linkghosts} Let $\mu_j$ denote a meridian of $R_j$ for $0\leq j\leq n-k$. A ghost of $\mu_j$ , denoted $(\mu_j)_*$ is an element of the set of $2^{n-k-j}$ circles $\{g^{n-k}f^{n-k}_{\pm}\circ\dots\circ f^{j+1}_{\pm}(\mu_j)\}$. Thus, for any $j$, the ghosts of $\mu_j$ live in $S^3-T(\alpha,R_{n-k})$ and $(\mu_j)_*\in \pi_1(S^3-T(\alpha,R_{n-k}))^{(n-j)}$. These circles are precisely the meridians of the \textbf{copies} of $S^3-R_j$ that are embedded in $S^3-T(\alpha,R_{n-k}$ by the maps $\{g^{n-k}f^{n-k}_{\pm}\circ\dots\circ f^{j+1}_{\pm}\}$. Note that $\mu_0$ is the meridian of $R_0=U$ so $\mu_0=\eta^0$. Thus in particular, taking $j=0$, the ghosts of $\mu_0$ coincide with the clones $\{\alpha^{n-k}_*\}$, that is $\{(\mu_0)_*\}=\{\alpha^{n-k}_*\}$.
\end{defn}

Theorem~\ref{thm:iteratedlink} is a special case ($j=n-k$) of the following more general result. This Proposition should be viewed as a formulation of the inductive proof of Theorem~\ref{thm:iteratedlink}.

\begin{prop}\label{prop:iteratedghost} Suppose $0\leq j \leq n-k$ and $W$ is an \emph{arbitrary} rational $n-j$-solution for $T_{n-k}\equiv T(\alpha,R_{n-k})$. Then at least one of the ghosts of $\mu_{j}$ maps non-trivially under the inclusion-induced map
$$
j_*:\pi_1(M_{T_{n-k}})\to \pi_1(W)/\pi_1(W)_r^{(n-j+1)}.
$$
\end{prop}

\begin{proof}[Proof of Proposition~\ref{prop:iteratedghost}] Here we view $k$ and $n$ as fixed and proceed by downward induction on $j$. First suppose $j=n-k$. In this degenerate case the single ghost is merely the meridian of $R_{n-k}$ viewed as a circle in $T(\alpha, R_{n-k})$, which is of course identified with a push-off, $\alpha^+$, of $\alpha$ itself, and $W$ is a rational $(k)$-solution for $M_{T_{n-k}}$. We must show that $j_*(\alpha^+)\neq 1$ under the map
$$
j_*:\pi_1(M_{T_{n-k}})\to \pi_1(W)/\pi_1(W)^{(k+1)}_r.
$$

Since $T_{n-k}$ is obtained from the trivial link $T$ by infection on a curve $\alpha\in F^{(k)}$, by ~\cite[Proposition 3.1]{Lei3}, there is a degree one map $r:M_{T_{n-k}}\to M_T$ that induces an isomorphism
$$
\pi_1(M_{T_{n-k}})/(\pi_1(M_{T_{n-k}}))^{(k+1)}_r\cong F/F^{(k+1)}
$$
and sends $\alpha^+$ to $\alpha$. Since $\alpha$ is not in  $F^{(k+1)}$, $\alpha^+ \neq 1$ in $\pi_1(M_{T_{n-k}})/\pi_1(M_{T_{n-k}})^{(k+1)}_r$. This also implies that the successive terms of the derived series of $\pi_1(M_{T_{n-k}})$ agree with those of the free group (up to this value of $k$). Thus the derived series, the rational derived series and even Harvey's torsion-free derived series agree for this group (up to this value of $k$)~\cite[Section 2]{Ha2})~\cite[Proposition 2.3]{Ha2}. This is useful because we now claim that the following is a monomorphism
$$
\pi_1(M_{T_{n-k}})/\pi_1(M_{T_{n-k}})^{(k+1)}\overset{j_*}\to \pi_1(W)/\pi_1(W)^{(k+1)}_r
$$
because the composition
$$
\pi_1(M_{T_{n-k}})/\pi_1(M_{T_{n-k}})^{(k+1)}\overset{j_*}\to \pi_1(W)/\pi_1(W)^{(k+1)}_r\to \pi_1(W)/\pi_1(W)^{(k+1)}_H
$$
is a monomorphism by the following result of the authors. Here we use that $W$ is a rational $(k)$-solution for $M_{T_{n-k}}$ and that the torsion-free derived series of a free group is the same its rational derived series.

\begin{prop}\label{prop:nsolvable}[Proposition 4.11 ~\cite{CH2}] If $M$ is rationally $(k)-solvable$ via $W$ then, letting $A=\pi_{1}(M)$ and $B=\pi_{1}(W)$, the inclusion $j:M \rightarrow W$ induces a monomorphism
$$j_*: \frac{\pi_1(M)}{\pi_1(M)_{H}^{(k+1)}} \hookrightarrow  \frac{\pi_1(W)}{\pi_1(W)_{H}^{(k+1)}}.$$
\end{prop}

Tt follows that $j_*(\alpha^+)\neq 1$ as required by Proposition~\ref{prop:iteratedghost}. Thus the Proposition holds for $j=n-k$.

Now suppose that the Proposition is true for $j+1$ where $1\leq j+1 \leq n-k$. We will establish it for $j$ (downwards induction). So consider a rational $(k+j)$-solution, $W$, for $M_{T_{n-k}}$. Let $\Lambda=\pi_1(W)/\pi_1(W)^{(n-j)}_r$ and let $\psi:\pi_1(W)\to \Lambda$, and $\phi:\pi_1(M_{T_j{n-k}})\to \Lambda$ be the induced coefficient systems.

Note that $W$ is \emph{a fortiori} a rational ($n-j-1$)-solution. Therefore the inductive hypothesis applies to $W$ for the value $j+1$ and allows us to conclude that at least one ghost of $\mu_{j+1}$ does not map into $\pi_1(W)^{(n-j)}_r$ under the inclusion, that is, we have ~$\phi(\mu_{j+1})_*)\neq 1$ for some ghost of $\mu_{j+1}$. We will need this fact below.

We can apply Proposition~\ref{prop:linkaltdescriptions} with $K=U$  to deduce that $L_n(U)$ ($\equiv T(\alpha,R_{n-k})\equiv T_{n-k}$) can be obtained from $L_{n-j-1}(U)\equiv T_{n-j-k-1}$ by infections along the clones $\{\alpha^{n-j-k-1}_*\}= \{g^{n-j-k-1}(\eta_*^{n-j-k-1})\}$  using the knot $R_{j+1}$ as infecting knot in each case. Then, in the notation of Theorem~\ref{thm:nontriviality}
$$
T_{n-k}=T_{n-j-k-1}(\alpha^{n-k-j-1}_i,R_{j+1}^i, 1\leq i\leq 2^{n-k-j-1})
$$
where $(R_{j+1})^i$ is the $i^{th}$ copy of $R_{j+1}$. Applying Theorem~\ref{thm:nontriviality} we see that, for any clone such that $\phi((\alpha^{n-k-j-1}_i)^+)\neq 1$ the kernel, $P_i$ of the composition
$$
\mathcal{A}_0(R_{j+1})\overset{}\to (\mathcal{A}_0(R_{j+1}) \otimes\mathbb{Q}\Lambda)\overset{i_*}{\to} H_1(M_{T_{n-k}};\mathbb{Q}\Lambda)\overset{j_*}\to H_1(W;\mathbb{Q}\Lambda),
$$
satisfies $P_i\subset P_i^\perp$. We claim that there exists at least one such clone. For, by definition of infection, when we infect $T_{n-j-k-1}$ along $\alpha^{n-k-j-1}_i$ the push-off or longitude of such a circle, $(\alpha^{n-k-j-1}_i)^+$, is identified to the meridian of the $i^{th}$ copy of the infecting knot $(R_{j+1})^i$. This meridian, when viewed as a circle in $T_{n-k}$, is not a meridian of the abstract knot $R_{j+1}$, but rather an embedded copy of that meridian in $T_{n-k}$. Thus $(\alpha^{n-k-j-1}_i)^+$, viewed as a circle in $T_{n-k}$, is, by definition, one of the one of the \textbf{ghosts} of $\mu_{j+1}$! But we established above, by our inductive assumption, that for at least one of these ghosts, $\phi((\mu_{j+1})_*)\neq 1$. Thus we have verified that there is at least one such clone (say the $i^{th}$) for which the hypotheses of Theorem~\ref{thm:nontriviality} apply. We now restrict attention to such a value of $i$.

The two circles
$$
f^{j+1}_\pm(\mu_{j}) \in \pi_1(S^3-R_{j+1})^{(1)}
$$
as shown in the Figure~\ref{fig:twocircles}, form a generating set for $\mathcal{A}_0(R_{j+1})$ (which is isomorphic to $\mathcal{A}_0(R_{1})$ and hence nontrivial). 
\begin{figure}[htbp]
\setlength{\unitlength}{1pt}
\begin{picture}(400,200)
\put(0,0){\includegraphics[height=200pt]{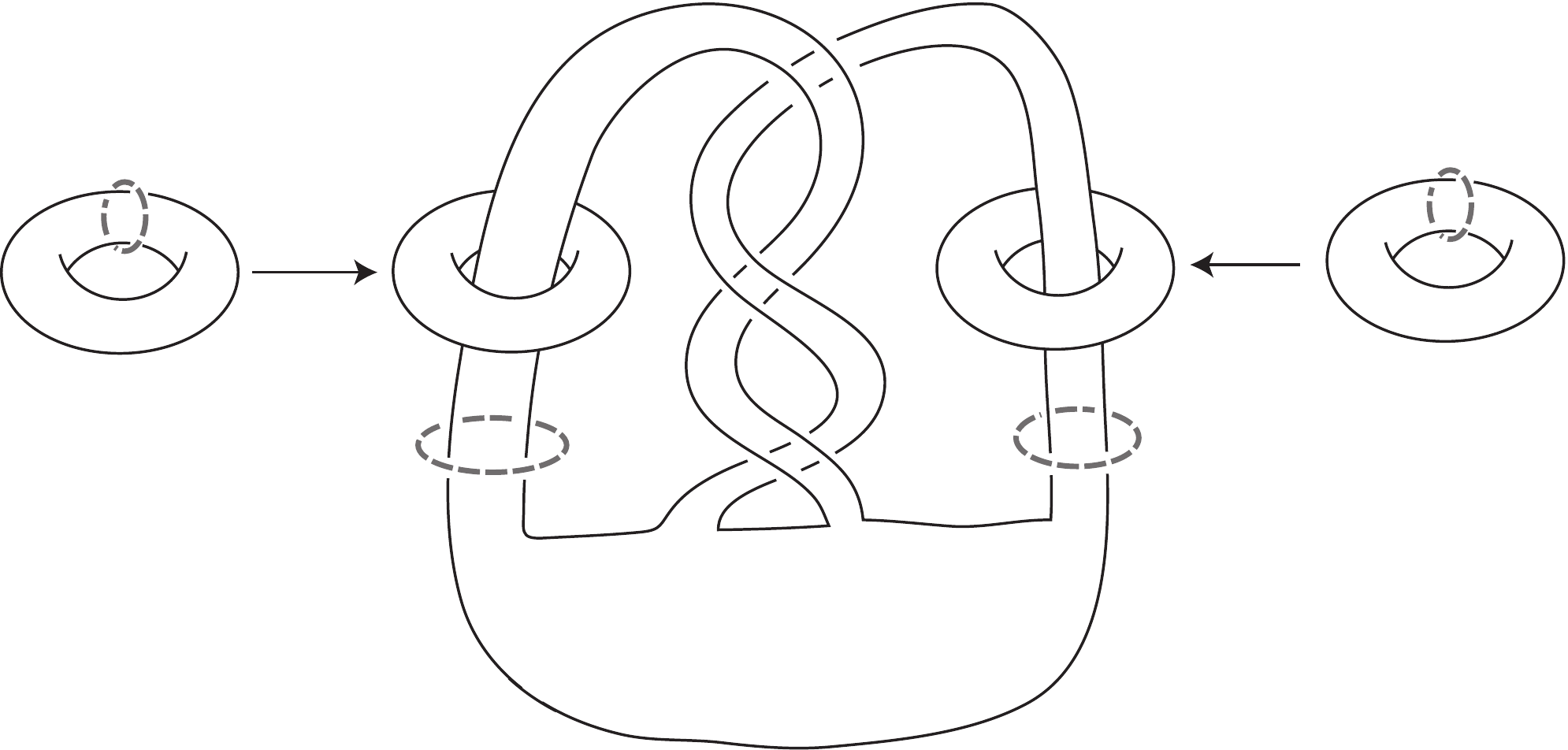}}
\put(74,136){$f^{j+1}_+$}
\put(321,137){$f^{j+1}_-$}
\put(6,93){$S^3-R_{j}$}
\put(361,95){$S^3-R_{j}$}
\put(61,69){$f^{j+1}_+(\mu_{j})$}
\put(306,71){$f^{j+1}_-(\mu_{j})$}
\put(18,160){$\mu_{j}$}
\put(370,160){$\mu_{j}$}
\end{picture}
\caption{Inside the $i^{th}$ copy of $S^3-R_{j+1}$}\label{fig:twocircles}
\end{figure}
From this we can conclude that at least one of the generators is not in $P_i$ since otherwise
$$
P_i=\mathcal{A}_0(R_{j+1})\subset\mathcal{A}_0(R_{j+1})^\perp,
$$
contradicting the nonsingularity of the classical Blanchfield form of $\mathcal{A}_0(R_{j+1})$. Finally, consider the commutative diagram below, where we abbreviate $\pi_1(W)$ by $\pi$. Recall that $H_1(W;\mathbb{Q}\Lambda)$ is identifiable as the ordinary rational homology of the covering space of $W$ whose fundamental group is the kernel of $\psi:\pi\to \Lambda$. Since this kernel is precisely $\pi^{(n-j)}_r$, we have that
$$
H_1(W;\mathbb{Q}\Lambda)\cong (\pi^{(n-j)}_r/[\pi^{(n-j)}_r,\pi^{(n-j)}_r])\otimes_{\mathbb{Z}} \mathbb{Q}
$$
as indicated in the diagram below. By the definition of the rational derived series, the far-right vertical map $j$ is injective.
\begin{equation*}
\begin{CD}
\pi_1(S^3-R_{j+1})^{(1)}      @>i_*>>    \pi_1(M_{T_{n-k}})^{(n-j)}  @>j_*>>   \pi^{(n-j)}_r  @>>>
\pi^{(n-j)}_r/\pi^{(n-j+1)}_r \\
  @VVV   @VVV        @VVV       @VVjV\\
\mathcal{A}_0(R_{j+1})     @>i_*>>  H_1(M_{T_{n-k}};\mathbb{Q}\Lambda)    @>j_*>> H_1(W;\mathbb{Q}\Lambda) @>\cong>>
  (\pi^{(n-j)}_r/[\pi^{(n-j)}_r,\pi^{(n-j)}_r])\otimes_{\mathbb{Z}} \mathbb{Q}\\
\end{CD}
\end{equation*}
Hence, since the composition in the bottom row sends one of the two homology classes $[f^{j+1}_\pm(\mu_{j})]$ to non-zero, the composition in the top row sends at least one of the two $f^{j+1}_\pm(\mu_{j})$ to non-zero under $i_*$. Now observe that the map $i_*$ in the top row above is induced by one of the compositions $g^{n-k}\circ f_\pm^{n-k}\circ\dots\circ f_\pm^{j+2}$. Thus
$$
i_*(f^{j+1}_\pm(\mu_{j})=g^{n-k}\circ f_\pm^{n-k}\circ\dots\circ f_\pm^{j+2}\circ f^{j+1}_\pm(\mu_{j}).
$$
For various values of $i$ these are precisely the ghosts of $\mu_{j}$. Hence we have shown that for at least one such ghost of $\mu_{j}$
$$
j_*((\mu_{j})_*)\neq 1 ~\text{in} ~\pi^{(n-j)}_r/\pi^{(n-j+1)}_r
$$
as desired. 

This finishes the inductive proof of Proposition~\ref{prop:iteratedghost}, hence finishing the proof of Theorem~\ref{thm:iteratedlink} and the proof of Theorem~\ref{thm:mainlink}.
\end{proof}
\end{proof}
\end{proof}

More generally, the proof shows the following:

\begin{thm}\label{thm:mainlink3} Suppose $T$ is a slice link, $\alpha$ is an unknotted circle in $S^3-T$ that represents an element in $\pi_1(S^3-T)^{(k)}$ and $R_j$, $1\leq j\leq n-k$, are slice knots and Arf($K$)$=0$, then the result, $T_{\alpha}\circ R_n\circ\dots\circ R_1(K)$, of the n-times iterated generalized doubling (applied to $K$) lies in $\mathcal{F}_{n}$. If, additionally, for each $j$, the submodule of the classical Alexander polynomial of $R_j$ generated by $\{\eta_{j1},\dots,\eta_{jm_j}\}$ contains elements $x,y$ such that $\mathcal{B}\ell_0^j(x,y)\neq 0$, where $\mathcal{B}\ell_0^j$ is the Blanchfield form of $R_j$, and $\alpha$  does not lie in $\pi_1(M_T)^{(k+1)}_H$, then there is a constant $C$, such that if the integral of the Levine signature function of $K$ is greater than $C$ in absolute value, then the resulting knot is of infinite order in the topological concordance group (moreover no multiple lies in $\mathcal{F}_{n+1}$).
\end{thm}

\section{Higher-Order Signatures as Obstructions to being Slice and the COT n-solvable Filtration }\label{sec:appendix}

\textbf{The COT n-solvable filtration}

Recall that \cite[Section 8]{COT} introduced a filtration of the concordance classes of links $\mathcal{C}$
$$
\cdots \subseteq \mathcal{F}_{n} \subseteq \cdots \subseteq
\mathcal{F}_1\subseteq \mathcal{F}_{0.5} \subseteq \mathcal{F}_{0} \subseteq \mathcal{C}.
$$
where the elements of $\mathcal{F}_{n}$ and $\mathcal{F}_{n.5}$ are called \textbf{$(n)$-solvable links} and \textbf{$(n.5)$-solvable links} respectively. In the case of knots this is a filtration by \emph{subgroups} of the knot concordance group. A slice link $L$ has the property that it's zero surgery $M_L$ bounds a $4$-manifold $W$ (namely the exterior of the slicing disks) such that $H_1(M_L)\to H_1(W)$ is an isomorphism and $H_2(W)=0$. An \textbf{n-solvable} link is one, loosely speaking, such that $M_L$ bounds a $4$-manifold $W$ such that $H_1(M_L)\to H_1(W)$ is an isomorphism and the intersection form on $H_2(W)$ ``looks'' hyperbolic modulo the $n^{th}$-term of the derived series of $\pi_1(W)$. We shall only give a detailed definition of the slightly larger class of \textbf{rationally $(n)$-solvable links}.

For a compact oriented topological 4-manifold $W$, let $W^{(n)}$ denote the covering
space of $W$ corresponding to the $n$-th derived subgroup of $\pi_1(W)$. The deck translation group of this cover is the solvable group $\pi_1(W)/\pi_1(W)^{(n)}$. Then $H_2(W^{(n)};\mathbb{Q})$ can be endowed with the structure of a right $\mathbb{Q}[\pi_1W)^{(n)}]$-module. This agrees with the homology group with twisted coefficients $H_2(W;\mathbb{Q}[\pi_1(W)^{(n)}])$. There is an equivariant
intersection form
$$
\lambda_n : H_2(W^{(n)};\mathbb{Q}) \times H_2(W^{(n)};\mathbb{Q}) \lra
\mathbb{Q}[\pi_1(W)/\pi_1(W)^{(n)}]
$$
 \cite[Chapter 5]{Wa}\cite[Section 7]{COT}. The usual intersection form is the case $n=0$. In general, these
intersection forms are singular. Let $I_n \equiv$ image($j* : H_2(\partial W^{(n)};\mathbb{Q}) \to H_2(W^{(n)};Q)$). Then this intersection form factors
through
$$
\ov{\lambda_n} : H_2(W^{(n)};\mathbb{Q})/I_n \times H_2(W^{(n)};\mathbb{Q})/I_n \lra \mathbb{Q}[\pi_1(W)/\pi_1(W)^{(n)}].
$$
We define a  \emph{rational $n$-Lagrangian} of $W$ to be a
submodule of $H_2(W;\mathbb{Q}[\pi_1W)^{(n)}]$ on which
$\ov\lambda_n$ vanishes identically and which maps onto a $\frac12$-rank
subspace of $H_2(W;\mathbb{Q})/I_0$ under the covering map. An
\emph{$n$-surface} is a based and immersed surface
in $W$ that can be lifted to $W^{(n)}$. Observe that any class in
$H_2(W^{(n)})$ can be represented by an $n$-surface and that
$\lambda_n$ can be calculated by counting intersection points in
$W$ among representative $n$-surfaces weighted appropriately by
signs and by elements of $\pi_1(W)/\pi_1(W)^{(n)}$. We say a rational
$n$-Lagrangian $L$ admits \emph{rational $m$-duals} (for $m\le n$) if $L$
is generated by (lifts of) $n$-surfaces $\ell_1,\ell_2,\ldots,\ell_g$ and
there exist $m$-surfaces $d_1,d_2,\ldots, d_g$ such that $H_2(W;\mathbb{Q})/I_0$
has rank $2g$ and $\lambda_m(\ell_i,d_j)=\delta_{i,j}$. 

Under the assumption that we will impose below, that
$$
H_1(M;\mathbb{Q})\to H_1(W;\mathbb{Q})
$$
is an isomorphism, it follows that the dual map
$$
H_3(W,M;\mathbb{Q})\to H_2(M;\mathbb{Q})
$$
is an isomorphism and hence that $I_0=0$. Thus the ``size'' of rational $(n)$-solutions is dictated by the rank of $H_2(W;\mathbb{Q})$.

\begin{defn}
\label{defn:rationalnsolvable} Let $n$ be a nonnegative integer. A compact, connected oriented topological 4-manifold $W$ with $\partial W = M$ is a \textbf{rational $n$-solution for $M$} if
\begin{itemize}
\item $H_1(M;\mathbb{Q})\to H_1(W;\mathbb{Q})$ is an isomorphism, and
\item  $W$ admits a rational $(n)$-Lagrangian with rational $(n)$-duals. 
\end{itemize}
Then we say that $M$ is \emph{rationally $(n)$-solvable via $W$}. A link $L$ is an \emph{$(n)$-solvable link} if 
$M_L$ is rationally $(n)$-solvable for some such $W$.
\end{defn}

\begin{defn}
\label{defn:rationaln.5solvable} Let $n$ be a nonnegative integer. A compact, connected oriented 4-manifold $W$ with
$\partial W = M$ is a \textbf{rational $n.5$-solution for $M$} if
\begin{itemize}
\item $H_1(M;\mathbb{Q})\to H_1(W;\mathbb{Q})$ is an isomorphism, and
\item  $W$ admits a rational $n$-Lagrangian with rational $(n+1)$-duals. 
\end{itemize}
Then we say that \textbf{$M$ is rationally $(n.5)$-solvable via $W$}. A link $L$ is an \textbf{$(n.5)$-solvable link} if $M_L$ is rationally $(n.5)$-solvable for some such $W$.
\end{defn}

A $4$-manifold $W$ is an \textbf{$(n)$-solution} (respectively an \textbf{$(n.5)$-solution}) if, in addition, it is spin, it satisfies the conditions above with $\mathbb{Q}$ replaced by $\mathbb{Z}$ and the equivariant \textbf{self-intersection form} also vanishes on the Lagrangian (see ~\cite[Section 8]{COT}.

\begin{rem}
\label{rem:n-solvable}
\begin{enumerate}
\item An $(n)$-solution is a fortiori a rational $(n)$-solution.
\item An $(n)$-solution (respectively rational $(n)$-solution) is a fortiori an $(m)$-solution (respectively rational $(m)$-solution) for any $m<n$.
\item If $L$ is slice in a topological (rational) homology $4$-ball then the complement of a set of slice disks is an (rational) $(n)$-solution for any integer or half-integer $n$.  This follows since if $H_2(W;\mathbb{Z})=0$ then the Lagrangian may be taken to be the zero submodule.
\end{enumerate}
\end{rem}

The following result is useful.

\begin{lem}\label{lem:nsolv} Suppose $L$ is a link obtained from a $(p+q)$-solvable link $R$ by infection along curves in $\pi_1(S^3-R)^{(p)}$ using knots $K_i$. Suppose the knots $K_i$ are $(q)$-solvable via $4$-manifolds $W_i$ such that $\pi_1(W_i)$ is normally generated by the meridian of $K_i$ (if $q=0$ the latter condition always holds). Then $L$ is also a $(p+q)$-solvable link.
\end{lem}
\begin{proof} One can repeat almost verbatim the proof of ~\cite[Proposition 3.1]{COT2} (see also ~\cite[Corollary 3.14]{CT}).
\end{proof}

\begin{thm}\label{thm:sliceobstr}(Cochran-Orr-Teichner~\cite[Theorem 4.2]{COT}) If a knot $K$ is rationally $(n.5)$-solvable via $W$ and  $\phi:\pi_1(M_K)\to \G$ is a PTFA coefficient system that extends to $\pi_1(W)$ and such that $\G^{(n+1)}=1$, then $\rho(M_K,\phi)=0$.
\end{thm}

For links the following recent result of the first two authors is the best known result. Note the extra rank condition.

\begin{thm}\label{thm:rho=0}[Cochran-Harvey ~\cite[Theorem 4.9, Proposition 4.11]{CH2}] Let $\Gamma$ be a PTFA group such that $\G^{(n+1)}=0$. Let $M$ be a closed, connected, oriented $3$-manifold equipped with a non-trivial coefficient system $\phi:\pi_1(M)\to \Gamma$. Suppose $\text{rank}_{\mathcal{K}\Gamma}(H_1(M;\mathcal{K}\G))= \beta_1(M)-1$. Then if $M$ is
rationally
$(n.5)$-solvable via a
$4$-manifold $W$ over which $\phi$ extends, then
$$
\rho(M,\phi)= \sigma^{(2)}_\Gamma(W)-\sigma(W)=0.
$$
Moreover, if additionally $M$ is
rationally $(n+1)$-solvable via $W$ then the extra rank condition above is automatically satisfied.
\end{thm}

\begin{proof}[Proof that Theorem~\ref{thm:rho=0} implies Theorem~\ref{thm:linksliceobstr}] Since $\G$ is PTFA, it is solvable so there exists some $n$ such that $\G^{(n+1)}=0$. Let $W$ denote the exterior of the slicing disks. By Alexander duality, $H_2(W;\mathbb{Q})=0$ and $H_1(M_L;\mathbb{Q})\to H_1(W;\mathbb{Q})$ is an isomorphism. Thus $W$ is a certainly a rational $(n+1)$-solution for $L$. The the result follows immediately from Theorem~\ref{thm:rho=0}.
\end{proof}

There is another common situation in which the extra rank condition is satisfied.

\begin{lem}\label{lem:rank} Suppose $L$ is a link obtained from the link $R$ by infections on circles $\eta_i$ using knots $K_i$. Suppose $\phi:\pi_1(M_L)\to\G$ is a nontrivial PTFA coefficient system such that $\phi(\mu_{\eta_i}\equiv l_{K_i})=1$. Then there is a coefficient system $\phi:\pi_1(M_L)\to\G$ induced on $M_R$ and
$$
\text{rank}_{\mathcal{K}\G}(H_1(M_{L};\mathcal{K}\G))\geq \text{rank}_{\mathcal{K}\G}(H_1(M_R;\mathcal{K}\G)).
$$
In particular if $R$ is the trivial link of $m$ components then 
$$
\text{rank}_{\mathcal{K}\G}(H_1(M_{L};\mathcal{K}\G))= \beta_1(M_L)-1.
$$
\end{lem}
\begin{proof}[Proof of Lemma~\ref{lem:rank}] Consider the cobordism $E_L$ of Figure~\ref{fig:mickey}. By Property $(1)$ of Lemma~\ref{lem:mickeyfacts}, the map
$$
\pi_1(M_L)\to \pi_1(E_L)
$$ 
is a surjection whose kernel is normally generated by $\{\mu_{\eta_i}\}$. Thus, as shown there, $\phi$ extends uniquely to $\pi_1(E_L)$ and hence by restriction to $\pi_1(M_R)$. Therefore there is a surjection 
$$
H_1(M_L;\mathcal{K}\G)\to H_1(E_L;\mathcal{K}\G)
$$ 
so
$$
\text{rank}_{\mathcal{K}\G}(H_1(M_L;\mathcal{K}\G))\geq \text{rank}_{\mathcal{K}\G}(H_1(E_L;\mathcal{K}\G)).
$$
Now examine the Mayer-Vietoris sequence with $\mathcal{K}\G$ coefficients for $E_L$ as in the proof of Lemma~\ref{lem:mickeysig}
$$
\oplus_i H_1(\eta_i \times D^2)\to \oplus_i H_1(M_{K_i})\oplus H_1(M_R)\to H_1(E_L)\overset{\partial_*}\to \oplus_iH_0(\eta_i \times D^2).
$$
We claim that the inclusion-induced maps
$$
H_0(\eta_i\x D^2;\mathcal{K}\G)\ra H_0(M_i;\mathcal{K}\G)
$$
are injective. In the case that $\phi(\eta_i)\neq 1$, $H_0(\eta_i\x D^2;\mathcal{K}\G)=~0$ by ~\cite[Proposition 2.9]{COT}, so injectivity holds. If $\phi(\eta_i)=1$ then, since $\eta_i$ is equated to the meridian of $K_i$, $\phi(\mu_{K_i})=1$. Since $\mu_i$ normally generates $\pi_1(M_i)$, it follows that the coefficient systems on $\eta_i\x D^2$ and $M_i$ are trivial and hence the injectivity follows from the injectivity with $\mathbb{Z}$-coefficients, which is obvious since both are path-connected. Hence $\partial_*$ is the zero map.
Similarly we claim that the inclusion-induced maps
$$
H_1(\eta_i\x D^2;\mathcal{K}\G)\ra H_1(M_{K_i};\mathcal{K}\G)
$$
are isomorphisms. In the case that $\phi(\eta_i)\neq 1$, both groups are zero by ~\cite[Lemma 2.10]{COT}. If $\phi(\eta_i)=1$ then both coefficient systems are trivial and result follows from the result for $\mathbb{Z}$-coefficients, which is obvious since $u_{K_i}$ generates $H_1(M_{K_i})\cong \mathbb{Z}$.

Armed with these observations, it now follows from the Mayer-Vietoris sequence that
$$
H_1(M_{R};\mathcal{K}\G)\cong H_1(E_L;\mathcal{K}\G).
$$
and the first result follows.

If $R$ is a trivial link then $\pi_1(M_R)$ is the free group $F$ of rank $m$. But it is easy to see from an Euler characteristic argument (~\cite[Lemma 2.12]{COT}) that 
$$
\text{rank}_{\mathcal{K}\G}(F;\mathcal{K}\G))=\beta_1(F)-1=m-1.
$$
Thus 
$$
\text{rank}_{\mathcal{K}\G}(H_1(M_L;\mathcal{K}\G))\geq \beta_1(M_L)-1
$$
but by ~\cite[Proposition 2.11]{COT}, this is also the maximum this rank can achieve, so the inequality is an equality.
\end{proof}

\bibliographystyle{plain}
\bibliography{mybib3}

\begin{thebibliography}{10}

\bibitem{CG1}
A.~Casson and C.~McA. Gordon.
\newblock On slice knots in dimension three.
\newblock {\em Proc. Symp. in Pure Math.}, XXX part 2:39--53, 1978.

\bibitem{CG2}
A.~Casson and C.~McA. Gordon.
\newblock Cobordism of classical knots.
\newblock {\em A la recherche de la topologie perdue 62}, 1986.
\newblock appeared as lecture notes, Orsay 1975.

\bibitem{Cha4}
J.C. Cha.
\newblock Link concordance, homology cobordism, and hirzebruch-type
  intersection form defects from towers of iterated p-covers.
\newblock preprint arXiv:/0705.0088.

\bibitem{Cha2}
J.C. Cha.
\newblock The structure of the rational concordance group of knots.
\newblock {\em Memoirs of American Math. Soc.}

\bibitem{Cha3}
J.C. Cha.
\newblock Topological minimal genus and $l^2$-signatures.
\newblock preprint http://xxx.lanl.gov/abs/math.GT/0609411.

\bibitem{CRL}
J.C. Cha, C.~Livingston, and D.~Ruberman.
\newblock Algebraic and heegard-floer invariants of knost with slice bing
  doubles.
\newblock {\em Math.Proc. Cambridge Phil. Soc.}
\newblock to appear, preprint http://xxx.lanl.gov/abs/math.GT/0612419.

\bibitem{ChGr1}
J.~Cheeger and M.~Gromov.
\newblock Bounds on the von neumann dimension of $l^2$-cohomology and the
  gauss-bonnet theorem for open manifolds.
\newblock {\em J. Differential Geom.}, 21:1--34, 1985.

\bibitem{Ci}
D.~Cimasoni.
\newblock Slicing bing doubles.
\newblock {\em Algebraic and Geometryic Topology}, 6:2395--2415, 2006.

\bibitem{C}
T.~Cochran.
\newblock Noncommutative knot theory.
\newblock {\em Algebr. Geom. Topol.}, 4:347--398, 2004.

\bibitem{CH2}
T.~Cochran and S.~Harvey.
\newblock Homology and derived series of groups ii: Dwyer's theorem.
\newblock preprint http://xxx.lanl.gov/abs/math.GT/0609484.

\bibitem{CHL1}
T.~Cochran, S.~Harvey, and Constance Leidy.
\newblock Knot concordance and blanchfield duality.
\newblock {\em Oberwolfach Reports}, 3(3), 2006.

\bibitem{CK}
T.~Cochran and T.~Kim.
\newblock Higher-order alexander invariants and filtrations of the knot
  concordance group.
\newblock {\em Trans.Amer.Math.Soc.}
\newblock In press; preprint http://xxx.lanl.gov/abs/math.GT/0411641.

\bibitem{COT}
T.~Cochran, K.~Orr, and P.~Teichner.
\newblock Knot concordance, whitney towers and $l^2$-signatures.
\newblock {\em Annals of Math.}, 157:433--519, 2003.

\bibitem{COT2}
T.~Cochran, K.~Orr, and P.~Teichner.
\newblock Structure in the classical knot concordance group.
\newblock {\em Comment. Math. Helv.}, pages 105--123, 2004.

\bibitem{CO2}
T.~Cochran and Kent Orr.
\newblock Homology boundary links and blanchfield forms: Concordance
  classification and new tangle-theoretic constructions.
\newblock {\em Topology}, 33:397--427, 1994.

\bibitem{CT}
T.~Cochran and P.~Teichner.
\newblock Knot concordance and von neumann $\rho$-invariants.
\newblock {\em Duke Math. Journal}, 137, no.2:337--379, 2007.

\bibitem{Fr2}
S.~Friedl.
\newblock Eta invariants as sliceness obstructions and their relation to
  casson-gordon invariants.
\newblock {\em Algebr. Geom. Topol.}, 4:893--934, 2004.

\bibitem{Fr3}
Stefan Friedl.
\newblock {$L\sp 2$}-eta-invariants and their approximation by unitary
  eta-invariants.
\newblock {\em Math. Proc. Cambridge Philos. Soc.}, 138(2):327--338, 2005.

\bibitem{FrT}
Stefan Friedl and Peter Teichner.
\newblock New topologically slice knots.
\newblock {\em Geom. Topol.}, 9:2129--2158, 2005.

\bibitem{Gi1}
P.~Gilmer.
\newblock Some interesting non-ribbon knots.
\newblock {\em Abstracts of papers presented to Amer. Math. Soc.}, 2:448, 1981.

\bibitem{GL1}
P.~Gilmer and C.~Livingston.
\newblock The {C}asson-{G}ordon invariant and link concordance.
\newblock {\em Topology}, 31(3):475--492, 1992.

\bibitem{Gi5}
Patrick~M. Gilmer.
\newblock Configurations of surfaces in {$4$}-manifolds.
\newblock {\em Trans. Amer. Math. Soc.}, 264(2):353--380, 1981.

\bibitem{Gi3}
Patrick~M. Gilmer.
\newblock Slice knots in {$S\sp{3}$}.
\newblock {\em Quart. J. Math. Oxford Ser. (2)}, 34(135):305--322, 1983.

\bibitem{Go1}
C.~McA. Gordon.
\newblock Some aspects of classical knot theory.
\newblock In {\em Knot theory (Proc. Sem., Plans-sur-Bex, 1977)}, volume 685 of
  {\em Lecture Notes in Math.}, pages 1--60. Springer, Berlin, 1978.

\bibitem{Ha2}
S.~Harvey.
\newblock Homology cobordism invariants of 3-manifolds and the
  cochran-orr-teichner filtration of the link concordance group.
\newblock preprint, http://front.math.ucdavis.edu/math.GT/0609378.

\bibitem{Ha1}
Shelly~L. Harvey.
\newblock Higher-order polynomial invariants of 3-manifolds giving lower bounds
  for the {T}hurston norm.
\newblock {\em Topology}, 44(5):895--945, 2005.

\bibitem{Ji1}
Bo~Ju Jiang.
\newblock A simple proof that the concordance group of algebraically slice
  knots is infinitely generated.
\newblock {\em Proc. Amer. Math. Soc.}, 83(1):189--192, 1981.

\bibitem{Ki1}
Taehee Kim.
\newblock Filtration of the classical knot concordance group and
  {C}asson-{G}ordon invariants.
\newblock {\em Math. Proc. Cambridge Philos. Soc.}, 137(2):293--306, 2004.

\bibitem{Lei3}
C.~Leidy.
\newblock Higher-order linking forms for 3-manifolds.
\newblock preprint.

\bibitem{Lei1}
C.~Leidy.
\newblock Higher-order linking forms for knots.
\newblock {\em Commentarii Math. Helv.}, 81:755--781, 2006.

\bibitem{Let}
Carl~F. Letsche.
\newblock An obstruction to slicing knots using the eta invariant.
\newblock {\em Math. Proc. Cambridge Philos. Soc.}, 128(2):301--319, 2000.

\bibitem{L5}
J.~Levine.
\newblock Knot cobordism groups in codimension two.
\newblock {\em Comm. Math. Helv.}, 44:229--244, 1969.

\bibitem{L6}
J.~Levine.
\newblock Link invariants via the eta invariant.
\newblock {\em Comm. Math. Helv.}, 69:82--119, 1994.

\bibitem{Le7}
J.~P. Levine.
\newblock Link concordance.
\newblock In {\em Algebra and topology 1988 (Taej\u on, 1988)}, pages 57--76.
  Korea Inst. Tech., Taej\u on, 1988.

\bibitem{Lith1}
R.~Litherland.
\newblock Cobordism of satellite knots.
\newblock (35):327--362, 1984.

\bibitem{LiM}
C.~Livingston and P.~Melvin.
\newblock Abelian invariants of satellite knots.
\newblock (1167):217--227, 1985.

\bibitem{Li7}
Charles Livingston.
\newblock Knots which are not concordant to their reverses.
\newblock {\em Quart. J. Math. Oxford Ser. (2)}, 34(135):323--328, 1983.

\bibitem{Li10}
Charles Livingston.
\newblock Links not concordant to boundary links.
\newblock {\em Proc. Amer. Math. Soc.}, 110(4):1129--1131, 1990.

\bibitem{Li6}
Charles Livingston.
\newblock Order 2 algebraically slice knots.
\newblock In {\em Proceedings of the Kirbyfest (Berkeley, CA, 1998)}, volume~2
  of {\em Geom. Topol. Monogr.}, pages 335--342 (electronic). Geom. Topol.
  Publ., Coventry, 1999.

\bibitem{Li5}
Charles Livingston.
\newblock Infinite order amphicheiral knots.
\newblock {\em Algebr. Geom. Topol.}, 1:231--241 (electronic), 2001.

\bibitem{Li1}
Charles Livingston.
\newblock A survey of classical knot concordance.
\newblock In {\em Handbook of knot theory}, pages 319--347. Elsevier B. V.,
  Amsterdam, 2005.

\bibitem{LS}
Wolfgang L{\"u}ck and Thomas Schick.
\newblock Various {$L\sp 2$}-signatures and a topological {$L\sp 2$}-signature
  theorem.
\newblock In {\em High-dimensional manifold topology}, pages 362--399. World
  Sci. Publ., River Edge, NJ, 2003.

\bibitem{P}
D.S. Passman.
\newblock {\em The Algebraic Structure of Group Rings}.
\newblock John Wiley and Sons, New York, 1977.

\bibitem{Ste}
B.~Stenstrom.
\newblock {\em Rings of Quotients}.
\newblock Springer-Verlag, New York, 1975.

\bibitem{Sto}
Neal~W. Stoltzfus.
\newblock Unraveling the integral knot concordance group.
\newblock {\em Mem. Amer. Math. Soc.}, 12(192):iv+91, 1977.

\bibitem{Wa}
C.~T.~C. Wall.
\newblock {\em Surgery on compact manifolds}, volume~69 of {\em Mathematical
  Surveys and Monographs}.
\newblock American Mathematical Society, Providence, RI, second edition, 1999.
\newblock Edited and with a foreword by A. A. Ranicki.

\end{thebibliography}
\end{document}